\documentclass{amsart} 
\usepackage{amsthm, amsmath, mathtools}
\usepackage{amssymb}
\usepackage{amscd}
\usepackage{enumerate}
\usepackage{mathrsfs} 
\usepackage[curve,matrix,arrow,frame,tips]{xy}
\usepackage{colonequals}
\usepackage{verbatim}
\usepackage{pdfsync}
\usepackage{graphicx}
\usepackage[usenames,dvipsnames]{color}

\usepackage{url}

\usepackage{fullpage}

\usepackage[normalem]{ulem}

\numberwithin{equation}{section} 
\numberwithin{table}{section}

\theoremstyle{plain}
\newtheorem{theorem}{Theorem}[section]
\newtheorem{proposition}[theorem]{Proposition}
\newtheorem{lemma}[theorem]{Lemma}

\newtheorem{corollary}[theorem]{Corollary}

\newtheorem{subprop}{Proposition}[subsection]

\theoremstyle{definition}

\newtheorem{nota}[theorem]{Notation}
\newtheorem{assume}[theorem]{Assumption}

\theoremstyle{remark}
\newtheorem{remark}[theorem]{Remark}

\newcommand{\HH}{\mathrm{H}}

\begin{document}

\title[Holomorphic poly-differentials of curves]{The Galois module structure of holomorphic poly-differentials and Riemann-Roch spaces}

\author[F. Bleher]{Frauke M. Bleher}
\address{F.B.: Department of Mathematics\\University of Iowa\\
14 MacLean Hall\\Iowa City, IA 52242-1419\\ U.S.A.}
\email{frauke-bleher@uiowa.edu}
\thanks{Both authors were supported in part by NSF Grant No.\ DMS-1801328. Frauke M. Bleher is the corresponding author.}

\author[A. Wood]{Adam Wood}
\address{A.W.: Department of Mathematics, Statistics, and Computer Science\\St. Olaf College\\1520 St. Olaf Avenue\\Northfield, MN 55057\\ U.S.A.}
\curraddr{Mathematics Department\\Breck School\\123 Ottawa Ave N\\ Golden Valley\\ MN 55422\\U.S.A.}
\email{acw8794@gmail.com}

\date{March 18, 2023}


\begin{abstract}
Suppose $X$ is a smooth projective geometrically irreducible curve over a perfect field $k$ of positive characteristic $p$.  Let $G$ be a finite group acting faithfully on $X$ over $k$ such that $G$ has non-trivial, cyclic Sylow $p$-subgroups.  If $E$ is a $G$-invariant Weil divisor on $X$ with $\mathrm{deg}(E)> 2g(X)-2$, we prove that the decomposition of $\HH^0(X,\mathcal{O}_X(E))$ into a direct sum of indecomposable $kG$-modules is uniquely determined by the class of $E$ modulo $G$-invariant principal divisors, together with the ramification data of the cover $X\to X/G$. The latter is given by the lower ramification groups and the fundamental characters of the closed points of $X$ that are ramified in the cover. As a consequence, we obtain that if $m>1$ and $g(X)\ge 2$, then the $kG$-module structure of  $\HH^0(X,\Omega_X^{\otimes m})$ is uniquely determined by the class of a canonical divisor on $X/G$ modulo principal divisors, together with the ramification data of $X\to X/G$.
This extends to arbitrary $m > 1$ the $m = 1$ case treated by the first author with T. Chinburg and A. Kontogeorgis.  
We discuss applications to the tangent space of the global deformation functor associated to $(X,G)$ and to congruences between prime level cusp forms in characteristic $0$. In particular, we complete the description of the precise $k\mathrm{PSL}(2,\mathbb{F}_\ell)$-module structure of all prime level $\ell$ cusp forms of even weight in characteristic $p=3$. 
\end{abstract}

\maketitle

\section{Introduction}
\label{s:intro}

Let $k$ be a perfect field, let $X$ be a smooth projective geometrically irreducible curve over $k$ of genus $g(X)$, and let $G$ be a finite group acting faithfully on the right on $X$ over $k$. As in \cite[\S IV.1]{Hartshorne1977}, we denote by $\Omega_X$ the sheaf of relative differentials of $X$ over $k$. Moreover, for every positive integer $m$, we let $\Omega_X^{\otimes m}$ be the $m$-fold tensor product of $\Omega_X$ with itself over $\mathcal{O}_X$. The group $G$ acts on the left on $\Omega_X^{\otimes m}$ and also on $\HH^0(X,\Omega_X^{\otimes m})$. We call the latter the space of holomorphic $m$-differentials of $X$. It is a classical problem, first posed by Hecke in \cite{Hecke1928}, to determine how $\HH^0(X,\Omega_X^{\otimes m})$ decomposes into a direct sum of indecomposable $kG$-modules and to give an explicit description of these indecomposables. When $k$ is algebraically closed and its characteristic does not divide $\#G$, this was solved by Chevalley and Weil in \cite{ChevalleyWeil1934} (see also \cite{Hurwitz1892}).

For the remainder of the paper, we assume that the characteristic of the perfect field $k$ is a prime number $p$ that divides $\#G$. Several authors have studied the $kG$-module structure of $\HH^0(X,\Omega_X^{\otimes m})$ under various assumptions on the group $G$, the parameter $m$, and the ramification of the cover $X \to X/G$. Their research has often focused on cyclic groups or abelian $p$-groups or the case when $m = 1$. 
See \cite{ValentiniMadan1981,Kani1986,Nakajima1986,RzedowskiCVillaSMadan1996,Kock2004,Borne2006,Karanikolopoulos2012,KaranikolopoulosKontogeorgis2013,MarquesWard2018,BleherChinburgKontogeorgis2020} for a sample of previous results. 

We would like to point out two of these articles, as they are closely related to our results. In \cite{Karanikolopoulos2012}, Karanikolopoulos determined the $kG$-module structure of $\HH^0(X,\Omega_X^{\otimes m})$ for $m > 1$ when $G$ is a finite cyclic $p$-group.
The main tool in \cite{Karanikolopoulos2012} was the use of so-called Boseck invariants to find a suitable $k$-basis of $\HH^0(X,\Omega_X^{\otimes m})$. In \cite{BleherChinburgKontogeorgis2020}, the authors took a more geometric approach to determine the $kG$-module structure of $\HH^0(X,\Omega_X)$ when $G$ has cyclic Sylow $p$-subgroups. 
It is the latter approach we will use in this paper. 

Let $K_{X/G}$ be a canonical divisor on $X/G$, and denote the $G$-cover $X\to X/G$ by $\gamma$ and its ramification divisor by $\mathrm{Ram}_\gamma$. Then $K_X=\gamma^*K_{X/G}+\mathrm{Ram}_\gamma$ is a $G$-invariant canonical divisor on $X$ (see \cite[Prop. IV.2.3]{Hartshorne1977}), and $\Omega_X^{\otimes m}\cong \mathcal{O}_X(mK_X)$ as $\mathcal{O}_X$-$G$-modules.
If $m>1$ and $g(X)\ge 2$ then $\mathrm{deg}(mK_X)> 2 g(X)-2$. Hence, it is natural to consider the more general case of the Riemann-Roch space $\HH^0(X,\mathcal{O}_X(E))$ associated to a $G$-invariant Weil divisor $E$ on $X$ satisfying $\mathrm{deg}(E)> 2 g(X)-2$.

This leads to our first main result.
By the $G$-divisor class of $E$, we mean the equivalence class of all Weil divisors $E'$ on $X$ such that $E-E'$ is the divisor of a $G$-invariant function in $k(X)^*$.

\begin{theorem}
\label{thm:newmain}
Suppose $G$ has non-trivial cyclic Sylow $p$-subgroups. Let $E$ be a $G$-invariant Weil divisor on $X$ with $\mathrm{deg}(E) > 2 g(X)-2$. Then the $kG$-module structure of $\HH^0(X,\mathcal{O}_X(E))$ is uniquely determined by the $G$-divisor class of $E$, together with the lower ramification groups and the fundamental characters of the closed points of $X$ that are ramified in the cover $X \to X/G$.
\end{theorem}

The $G$-divisor class of the canonical divisor $K_X=\gamma^*K_{X/G}+\mathrm{Ram}_\gamma$ is uniquely determined by the class of $K_{X/G}$ modulo principal divisors on $X/G$. Therefore, we obtain as a consequence of Theorem \ref{thm:newmain} our second main result, which extends \cite[Thm. 1.1]{BleherChinburgKontogeorgis2020} from $m=1$ to arbitrary $m>1$. Since for all $m\ge 1$, $g(X)=0$ implies $\HH^0(X,\Omega_X^{\otimes m})=0$ and $g(X)=1$ implies $\HH^0(X,\Omega_X^{\otimes m})=k$ with trivial $G$-action, the assumption $g(X)\ge 2$ is not a real restriction. 

\begin{corollary}
\label{thm:oldmain}
Suppose $m>1$ and $g(X)\ge 2$, and that $G$ has non-trivial cyclic Sylow $p$-subgroups. The $kG$-module structure of $\HH^0(X,\Omega_X^{\otimes m})$ is uniquely determined by the  class of canonical divisors on $X/G$ modulo principal divisors, together with the lower ramification groups and the fundamental characters of the closed points of $X$ that are ramified in the cover $X\to X/G$. 
\end{corollary}

The statement of Corollary \ref{thm:oldmain} differs from \cite[Thm. 1.1]{BleherChinburgKontogeorgis2020} since if $m=1$ then the $kG$-module structure of $\HH^0(X,\Omega_X)$ does not depend on the class of canonical divisors on $X/G$.

When $m=2$, Corollary \ref{thm:oldmain} can be used to determine the dimension of the tangent space of the global deformation functor associated to the pair $(X,G)$ (see \S\ref{s:deformation}). 
Similarly to the case when $m=1$, Corollary \ref{thm:oldmain} also has implications for the study of classical modular forms. More precisely, for any prime number $\ell$ with $p\ne \ell \ge 7$, let $X(\Gamma_\ell)$ be the modular curve associated to the principal congruence subgroup $\Gamma_\ell$ of level $\ell$, and let $X_p(\ell)$ be the reduction of $X(\Gamma_\ell)$ in characteristic $p$. 
For all $p\ge 3$,
we obtain non-trivial congruences modulo appropriate maximal ideals over $p$ between level $\ell$ cusp forms in characteristic $0$ of even weight $2m>2$ such that these congruences arise from isotypic components with respect to the action of $\mathrm{PSL}(2,\mathbb{F}_\ell)$ on $X(\Gamma_\ell)$ (see Theorems \ref{thm:firstmodtheorem} and \ref{thm:secondmodtheorem}). To prove this when $p=3$, we 
determine for all $m>1$, the precise $k\mathrm{PSL}(2,\mathbb{F}_\ell)$-module structure of $\HH^0(X_3(\ell),\Omega_{X_3(\ell)}^{\otimes m})$ (see Theorem \ref{thm:modularresult} and \S\ref{s:modular3}). 

In particular, Theorem \ref{thm:modularresult} and Propositions \ref{prop:full_mixed} and \ref{prop:full_equal}, together with \cite[Thm. 1.4 and Props. 6.4.1 - 6.4.4]{BleherChinburgKontogeorgis2020}, give a complete solution to Hecke's classical problem  \cite{Hecke1928,ChevalleyWeil1934} of determining the precise module structure of all prime level cusp forms of even weight in the case when $k$ has characteristic $p=3$. 
Using \cite{Kani1986} or \cite{Nakajima1986} for $p>3$, it follows that the only remaining open case is when $p=2$.

The proof of Theorem \ref{thm:newmain} follows the main outline used in the proof of \cite[Thm. 1.1]{BleherChinburgKontogeorgis2020}. The first two steps are 
(I) to use the Conlon induction theorem \cite[Thm. (80.51)]{CRII} to reduce to the case when $G$ is $p$-hypo-elementary, and 
(II) to use the representation theory of $p$-hypo-elementary groups \cite[\S II.5-II.6]{Alperin1986} to reduce to the case when $k$ is algebraically closed
(see \cite[\S3]{BleherChinburgKontogeorgis2020} for details on these two reduction steps). 

Once these two reduction steps are made, Theorem \ref{thm:main3} below gives a precise algorithm of how to determine the $kG$-module structure of $\HH^0(X,\mathcal{O}_X(E))$ when $G$ is $p$-hypo-elementary and $k$ is algebraically closed. Since information about the module structure of Riemann-Roch spaces has been useful in the context of algebraic geometry codes, Theorem \ref{thm:main3} provides new insights into this subject. Previous work has mostly focused on the case when $kG$ is semisimple, see, for example, \cite{JoynerKsir2007}, \cite{GlassJoynerKsir2010}, \cite{CaroccaVasquez2019}. 

To state Theorem \ref{thm:main3}, we need the following assumptions and notations. 

\begin{assume}
\label{ass:main3}
Let $k$ be an algebraically closed field of prime characteristic $p$. Suppose $G=P\rtimes_{\chi}C$ is a $p$-hypo-elementary group, where $P=\langle \sigma\rangle$ is a cyclic $p$-group of order $p^n>1$, $C=\langle \rho\rangle$ is a cyclic $p'$-group of order $c$, and $\chi:C\to \mathbb{F}_p^*$ is a character which we also view as a character of $G$ by inflation. 
Let $X$ be a smooth projective curve over $k$, and suppose there is a faithful right action of $G$ on $X$ over $k$. 
Let $E$ be a $G$-invariant Weil divisor on $X$ with $\mathrm{deg}(E) > 2 g(X)-2$ and write $E=\sum_{x\in X}e_x x$. 
\end{assume} 

When $k$ and $G$ are as in Assumption $\ref{ass:main3}$ then there are precisely $\#G$ isomorphism classes of indecomposable $kG$-modules. Each simple $kG$-module has $k$-dimension one and it is uniquely determined by the action of $C$ on it. Moreover, each indecomposable $kG$-module $U$ has a simple socle, which is the kernel of the action of $(\sigma -1)$ on $U$, its radical is equal to $(\sigma-1)U$, and the isomorphism class of $U$ is uniquely determined by its socle and its $k$-dimension (see \cite[\S II.5-II.6]{Alperin1986} and also \cite[Remark 3.4]{BleherChinburgKontogeorgis2020}). In other words, $U$ is a uniserial $kG$-module, in the sense that it has a unique composition series. For a discussion of uniserial modules over Artin algebras, see \cite[\S IV.2]{AuslanderReitenSmalo1997}.

\begin{nota}
\label{not:main3ind}
Let $k$ and $G$ be as in Assumption $\ref{ass:main3}$. Let $\zeta\in k^*$ be a primitive $c^{\mathrm{th}}$ root of unity. 
\begin{enumerate}
\item[(a)] For $0\le a\le c-1$, let $S_a$ be a one-dimensional $k$-vector space on which $\rho$ acts as multiplication by $\zeta^a$, and view $S_a$ as a (simple) $kG$-module by inflation. 
\item[(b)] For $0\le a\le c-1$ and $1\le b\le p^n$, let $U_{a,b}$ be an indecomposable $kG$-module with socle $S_a$ and $k$-dimension $b$.
\item[(c)] For $i\in \mathbb{Z}$ and $0\le a\le c-1$, we define $\chi^i.a$ to be the unique element in $\{0,1,\ldots,c-1\}$ such that $S_{\chi^i.a} \cong S_{\chi^i}\otimes_k S_a$, where $S_{\chi^i}$ is the simple $kG$-module with character $\chi^i$. By \cite[\S II.5-II.6]{Alperin1986}, $U_{a,b}$ has ascending composition factors $S_a,S_{\chi^{-1}.a}, S_{\chi^{-2}.a},\ldots,S_{\chi^{-b+1}.a}$.
\end{enumerate}
\end{nota}

\begin{nota}
\label{not:main3}
Suppose $k$, $G$ and $X$ are as in Assumption $\ref{ass:main3}$.
\begin{enumerate}
\item[(a)] Let $I = \langle \tau \rangle$ be the (cyclic) subgroup of $P$ that is the greatest among the Sylow $p$-subgroups of the inertia groups of all closed points of $X$.  Write $\#I = p^{n_I}$, where $0\le n_I\le n$. 
\item[(b)] For any closed point $x$ on $X$ and  integer $i \ge 0$, let $I_{x,i}$ denote the $i^{\mathrm{th}}$ lower ramification subgroup of $I$ at $x$. 
In other words, $I_{x,i}$ is the group of all elements in $I$ that fix $x$ and act trivially on $\mathcal{O}_{X,x}/\mathfrak{m}_{X,x}^{i+1}$. 
The inertia group of $x$ in $I$ is $I_x=I_{x,0}$, and we write $\# I_x=p^{n_x}$. Since $I_x$ is cyclic, there are precisely $n_x$ jumps $1\le i_1(x)<i_2(x)<\cdots < i_{n_x}(x)$ in the numbering of the lower ramification groups $I_{x,i}$, $i\ge 0$ (see \cite[Chap. IV]{SerreCorpsLocaux1968}). We have $I_x=I_{x,0}=I_{x,i_1(x)}$,  $I_{x,i_j(x)}> I_{x,i_j(x)+1}=I_{x,i_{j+1}(x)}$ for $1\le j\le n_x-1$, and $I_{x,i_{n_x}(x)}>I_{x,i_{n_x}(x)+1}=1$. 
\item[(c)] Let $Y=X/I$, and let $\pi:X\to Y$ be the corresponding quotient morphism. Define $\overline{G}=G/I$, and let $Z=X/G=Y/\overline{G}$ with corresponding quotient morphism $\lambda:Y\to Z$. Let $Z_{\mathrm{br}}$ be the set of closed points in $Z$ that are branch points of the cover $\lambda$. 
\item[(d)] For $y\in Y$, let $\overline{G}_{y}$ be the inertia group of $y$ in $\overline{G}$, and let $\theta_y: \overline{G}_{y}\to k(y)^*=k^*$ be the fundamental character of $\overline{G}_{y}$. In other words, if $\pi_y$ is a uniformizer and $\mathfrak{m}_{Y,y}$ is the maximal ideal of  $\mathcal{O}_{Y,y}$ then
$\theta_y(\overline{g}) =  \frac{\overline{g}.\pi_y}{\pi_y} \mod \mathfrak{m}_{Y,y}$ for all $\overline{g}\in \overline{G}_y$. We write $\#\overline{G}_y=c_y$. Since $c_y$ divides $c$ and since we can view the modules from Notation \ref{not:main3ind}(a) as $k\overline{G}$-modules, there exists $\varphi(y)\in \{0,1,\ldots,c_y-1\}$  such that $\theta_y$ is the character of $\mathrm{Res}^{\overline{G}}_{\overline{G}_y}S_{\varphi(y)}$. For $0\le a\le c-1$ and $i\in\mathbb{Z}$, define $\mu_{a,i}(y)$ to be $1$ if $a\equiv i \,\varphi(y)\mod  c_y$ and to be $0$ otherwise. In other words, $\mu_{a,i}(y)=1$ if and only if $\theta_y^i$ is the character of $\mathrm{Res}^{\overline{G}}_{\overline{G}_y}S_a$.
\item[(e)] Let $K_Z$ be a canonical divisor on $Z$. Let $\mathrm{Ram}_\pi$ $($resp. $\mathrm{Ram}_\lambda$$)$ be the ramification divisor of the cover $\pi$ $($resp. $\lambda$$)$. By \cite[\S IV.2]{Hartshorne1977} and \cite[\S IV.1]{SerreCorpsLocaux1968}, $\mathrm{Ram}_\pi = \sum_{x\in X}\sum_{i\ge 0} \left(\# I_{x,i} - 1\right)x$ and $\mathrm{Ram}_\lambda=\sum_{y\in Y} (c_y-1)\,y$. Define $K_Y=\lambda^*K_Z + \mathrm{Ram}_\lambda$ and $K_X=\pi^*K_Y+\mathrm{Ram}_\pi$. By \cite[Prop. IV.2.3]{Hartshorne1977}, $K_Y$ is a $\overline{G}$-invariant canonical divisor on $Y$ and $K_X$ is a $G$-invariant canonical divisor on $X$.
\end{enumerate}
\end{nota}

\begin{theorem}
\label{thm:main3}
Suppose Assumption $\ref{ass:main3}$ holds, and assume Notations $\ref{not:main3ind}$ and $\ref{not:main3}$. 
Fix $0 \le j \le p^{n_I} - 1$, and define a $G$-invariant Weil divisor $E_j=\sum_{y\in Y} e_{y,j}\,y$ on $Y$
by
\begin{equation}
\label{eq:Ej}
e_{y,j}=\left\lfloor \frac{e_{x(y)}-\sum_{\ell=1}^{n_{x(y)}}a_{\ell,t}\,p^{n_{x(y)}-\ell}\,i_\ell(x(y))}{p^{n_{x(y)}}}\right\rfloor 
\end{equation}
where $x(y)\in X$ is a point above $y$, $t\ge 0$ is the unique integer such that $p^{n_I-n_{x(y)}}\,t \le j < p^{n_I-n_{x(y)}}(t+1)$, and $a_{1,t},\ldots, a_{n_{x(y)},t}\in\{0,1,\ldots,p-1\}$ are given by the $p$-adic expansion of $t$,
$$t = a_{1,t} + a_{2,t}\,p + \cdots + a_{n_{x(y)},t}\,p^{n_{x(y)}-1}.$$
For $z\in Z$, let $y(z)\in Y$ be such that $\lambda(y(z))=z$. Define $\ell_{y(z),j}\in\{0,1,\ldots, c_{y(z)}-1\}$ by
\begin{equation}
\label{eq:ellzj}
\ell_{y(z),j}\equiv e_{y(z),j}\mod c_{y(z)},
\end{equation}
and define
\begin{equation}
\label{eq:nj}
n_j = 1-g(Z)+\sum_{z\in Z} \frac{e_{y(z),j}-\ell_{y(z),j}}{c_{y(z)} } .
\end{equation}
For $0\le a\le c-1$, define
\begin{equation}
\label{eq:naj}
n(a,j) = \sum_{z\in Z_{\mathrm{br}}} \left( \sum_{d=1}^{\ell_{y(z),j}}\mu_{a,-d}(y(z)) -  \sum_{d=1}^{c_{y(z)}-1} \frac{d}{c_{y(z)}} \mu_{a,d}(y(z))\right) + n_j.
\end{equation}
Then $n(a,j)$ is a non-negative integer. For $1\le b\le p^n$, the number $n_{a,b}$ of indecomposable direct $kG$-module summands of $\HH^0(X,\mathcal{O}_X(E))$ that are isomorphic to $U_{a,b}$ is given as 
\begin{equation}
\label{eq:nab}
n_{a,b}=\left\{\begin{array}{cl}
n(a,j)-n(a,j+1)& \mbox{if $b=(j+1)\,p^{n-n_I}$ and $0\le j \le p^{n_I}-2$,}\\
n(a,p^{n_I}-1) & \mbox{if $b=p^{n_I}\,p^{n-n_I}=p^n$,}\\
0&\mbox{otherwise.}
\end{array}\right.
\end{equation}
\end{theorem}

\begin{remark}
\label{rem:computation}
If $z\in Z-Z_\mathrm{br}$ then  $\ell_{y(z),j}=0$ for $0\le j\le p^{n_I}-1$.
Moreover, for $0\le j\le p^{n_I}-2$, we have
$$n(a,j)-n(a,j+1)=\sum_{z\in Z_{\mathrm{br}}} \left(\sum_{d=1+\mathrm{min}\{\ell_{y(z),j},\ell_{y(z),j+1}\}}^{\mathrm{max}\{\ell_{y(z),j},\ell_{y(z),j+1}\}}  \epsilon_{z,j} \,\mu_{a,-d}(y(z))\right) + n_j-n_{j+1}$$
where $\epsilon_{z,j}=1$ if $\ell_{y(z),j}\ge \ell_{y(z),j+1}$ and $\epsilon_{z,j}=-1$ otherwise.
\end{remark}

As a consequence of Theorem \ref{thm:main3}, we obtain the following result concerning holomorphic poly-differentials.

\begin{corollary}
\label{thm:main3cor}
Under the assumptions of Theorem $\ref{thm:main3}$, suppose $m>1$, $g(X)\ge 2$ and $E=mK_X$, where $K_X$ is as in Notation $\ref{not:main3}(e)$, so that $\Omega_X^{\otimes m}\cong \mathcal{O}_X(mK_X)$ as $\mathcal{O}_X$-$G$-modules. Then, for $0 \le j \le p^{n_I} - 1$, the divisor $E_j$ given by Equation $(\ref{eq:Ej})$
satisfies $E_j = mK_Y+D_j$, where $D_j=\sum_{y\in Y} d_{y,j}\,y$ is defined by
\begin{equation}
\label{eq:Dj}
d_{y,j} = \left\lfloor \frac{m\, \sum_{i\ge 0} \left(\# I_{x(y),i} - 1\right) - \sum_{\ell = 1}^{n_{x(y)}} a_{\ell,t} \,p^{n_{x(y)}-\ell}\,i_\ell(x(y))}{p^{n_{x(y)}}}
\right\rfloor.
\end{equation}
In particular, $d_{y,j}=0$ when $y$ is not a branch point of $\pi$. Using these $E_j$, Theorem $\ref{thm:main3}$ determines the precise $kG$-module structure of $\HH^0(X,\Omega_X^{\otimes m})$.
\end{corollary}

In Remark \ref{rem:refereerequest!}, we will use arguments from the proof of Theorem \ref{thm:main3} to give a more explicit version of the algorithm provided in \cite[Remark 4.4]{BleherChinburgKontogeorgis2020} concerning $\HH^0(X,\Omega_X)$.

The two main differences in the proofs of \cite[Thm. 1.1 and Remark 4.4]{BleherChinburgKontogeorgis2020} and of Theorems \ref{thm:newmain} and  \ref{thm:main3}, including Corollaries \ref{thm:oldmain} and \ref{thm:main3cor}, are as follows, when assuming $k$ and $G$ are as in Assumption \ref{ass:main3}.

\begin{enumerate}
\item[(A)] Suppose $I$ is not trivial, i.e. the cover $\pi:X\to Y$ is wildly ramified. When considering the sheaf $\Omega_X\cong\mathcal{O}_X(K_X)$ then $\HH^1(Y,\Omega_Y\otimes_{\mathcal{O}_Y}\mathcal{O}_Y(D_j))$ does not vanish when $j=\#I-1$, and neither does $\HH^1(X,\Omega_X)$. On the other hand, when considering the sheaf $\mathcal{O}_X(E)$ then $\HH^1(Y,\mathcal{O}_Y(E_j))$ vanishes for all $0\le j\le \#I-1$. Since $\HH^1(X,\mathcal{O}_X(E))=0$, this is crucial for a dimension argument needed to compare the filtrations of  $\HH^0(X,\mathcal{O}_X(E))$ and of $\mathcal{O}_X(E)$, given by the actions of the powers of the radical of $kI$ (see Notation \ref{not:another}).

\item[(B)] When considering the sheaf $\Omega_X\cong\mathcal{O}_X(K_X)$ then one can use Serre duality to 
avoid having to refer to the class of the canonical divisor $K_Z$ modulo principal divisors on $Z$. On the other hand, when $E=mK_X$ for $m>1$ then the class of $K_Z$ is essential when analyzing the tamely ramified cover $\lambda:Y\to Z$. For general $E$ as in Assumption \ref{ass:main3}, the 
$G$-divisor class of $E$ is used to analyze $\lambda$.
\end{enumerate}

The paper is organized as follows. 
In \S \ref{s:mainthm}, we prove Theorems \ref{thm:newmain} and \ref{thm:main3}, and deduce Corollary \ref{thm:main3cor}. Moreover, in Remark \ref{rem:refereerequest!}, we give a new version of the algorithm for $\HH^0(X,\Omega_X)$ from \cite[Remark 4.4]{BleherChinburgKontogeorgis2020}.
In \S\ref{s:deformation}, we show in Proposition \ref{prop:deformation} how Corollaries \ref{thm:oldmain} and \ref{thm:main3cor} can be used to find the $k$-dimension of the tangent space of the global deformation functor associated to $(X,G)$. 
In \S \ref{s:hyperelliptic}, we illustrate Corollaries \ref{thm:oldmain} and \ref{thm:main3cor} and Proposition \ref{prop:deformation} by considering a particular family of hyperelliptic curves. 
In \S \ref{s:modular}, we discuss non-trivial congruences modulo appropriate maximal ideals between prime level $\ell$ cusp forms of even weight $2m>2$; see Theorems \ref{thm:firstmodtheorem} and \ref{thm:secondmodtheorem}. If $p=3$, we use Corollaries \ref{thm:oldmain} and \ref{thm:main3cor} to fully determine the $k\mathrm{PSL}(2,\mathbb{F}_\ell)$-module structure of  $\HH^0(X_3(\ell),\Omega_{X_3(\ell)}^{\otimes m})$; see Theorem \ref{thm:modularresult} and Propositions \ref{prop:full_mixed} and \ref{prop:full_equal}, which we prove in  \S \ref{s:modular3}.

Part of this paper is the Ph.D. thesis of the second author under the supervision of the first (see \cite{WoodThesis2020}).
We would like to thank the referee for valuable comments that helped improve the paper.


\section{Proof of Theorems \ref{thm:newmain} and \ref{thm:main3}, and Corollary \ref{thm:main3cor}}
\label{s:mainthm}

By using the Conlon induction theorem \cite[Thm. (80.51)]{CRII} and the representation theory of $p$-hypo-elementary groups \cite[\S II.5-II.6]{Alperin1986}, we can reduce the proof of Theorem \ref{thm:newmain}  to the case when $G$ is $p$-hypo-elementary and $k$ is algebraically closed (see  \cite[Lemma 3.2 and Prop. 3.5]{BleherChinburgKontogeorgis2020} for details). Since $\HH^0(X,\mathcal{O}_X(E))$ and $\HH^0(X,\mathcal{O}_X(E'))$ are isomorphic $kG$-modules when $E$ and $E'$ lie in the same $G$-divisor class, Theorem \ref{thm:newmain} then follows from Theorem \ref{thm:main3}. As we have seen in the introduction, Corollary \ref{thm:oldmain} is a direct consequence of Theorem \ref{thm:newmain}.

For the remainder of this section, we suppose Assumption \ref{ass:main3} holds, and we assume Notations \ref{not:main3ind} and \ref{not:main3}. In particular, 
$$E=\sum_{x\in X}e_x x$$ 
is a $G$-invariant Weil divisor on $X$ with $\mathrm{deg}(E) > 2 g(X)-2$. The sheaf $\mathcal{O}_X(E)$ is a coherent $\mathcal{O}_X$-$G$-module that is a locally free $\mathcal{O}_X$-module of rank one, and $\HH^1(X,\mathcal{O}_X(E))=0$. 

If $x(y)\in X$ lies above $y\in Y$, then $p^{n_{x(y)}}=\#I_{x(y)}$ is the ramification index of $x(y)$ over $y$. 
Under the assumptions of Corollary \ref{thm:main3cor}, this immediately implies that if $E=mK_X$ then the divisor $E_j$ from Equation (\ref{eq:Ej}) equals $mK_Y + D_j$ with $D_j$ as in Equation (\ref{eq:Dj}). Hence Corollary \ref{thm:main3cor} follows from Theorem \ref{thm:main3}.

It remains to prove Theorem \ref{thm:main3}. We use the geometric method that was introduced in \cite[\S 4]{BleherChinburgKontogeorgis2020} when considering the holomorphic differentials $\HH^0(X,\Omega_X)$. 
The first step is to consider the cover $\pi:X\to Y$, which is wildly ramified if $\#I>1$, and to study a particular filtration of the $G$-sheaf $\pi_*\mathcal{O}_X(E)$.
We need the following notation, which was introduced in \cite[\S2]{BleherChinburgKontogeorgis2020}.

\begin{nota}
\label{not:another}
Suppose $\mathcal{F}$ is a coherent $\mathcal{O}_X$-$G$-module that is a locally free $\mathcal{O}_X$-module of finite rank, and let $\mathcal{J} = kI(\tau-1)$ be the Jacobson radical of the group ring $kI$. For $0 \le j \le \#I$, we denote by $\HH^0(X,\mathcal{F})^{(j)}$ the kernel of the action of $\mathcal{J}^j = kI(\tau-1)^j$ on $\HH^0(X,\mathcal{F})$, and we denote by $\pi_*\mathcal{F}^{(j)}$ the kernel of the action of $\mathcal{J}^j$ on $\pi_*\mathcal{F}$. 
\end{nota}

Since $I$ is a normal subgroup of $G$, $\mathcal{J}^j$ is taken to itself by the conjugation action of $G$ on $I$. It follows that $\HH^0(X,\mathcal{O}_X(E))^{(j)}$ is a $kG$-module and that $\pi_*\mathcal{O}_X(E)^{(j)}$ is a quasi-coherent $\mathcal{O}_Y$-$G$-module. Moreover, since $\pi_*\mathcal{O}_X(E)^{(j)}$ is a subsheaf of a locally free coherent $\mathcal{O}_Y$-module of finite rank, we obtain that $\pi_*\mathcal{O}_X(E)^{(j)}$ is a locally free coherent $\mathcal{O}_Y$-$G$-module. 

We now analyze the filtration $\{\pi_*\mathcal{O}_X(E)^{(j)}\}_{j=0}^{\#I}$ of the $G$-sheaf $\pi_*\mathcal{O}_X(E)$.

\begin{proposition}
\label{prop:filtersheaf}
Suppose Assumption $\ref{ass:main3}$ holds and assume Notations $\ref{not:main3ind}$, $\ref{not:main3}$ and $\ref{not:another}$. Fix $0 \le j \le p^{n_I} - 1$. The action of $\mathcal{O}_Y$ and of $G$ on $\pi_*\mathcal{O}_X(E)$ makes the quotient sheaf  $\mathcal{L}_j := \pi_*\mathcal{O}_X(E)^{(j+1)}/\pi_*\mathcal{O}_X(E)^{(j)}$ into a locally free coherent $\mathcal{O}_Y$-$\overline{G}$-module.  Let $E_j$ be the $G$-invariant Weil divisor on $Y$ given by Equation $(\ref{eq:Ej})$ in Theorem $\ref{thm:main3}$.
\begin{itemize}
\item[(i)] There is an isomorphism of locally free coherent $\mathcal{O}_Y$-$\overline{G}$-modules between  $\mathcal{L}_j$ and $S_{\chi^{-j}} \otimes_k \mathcal{O}_Y(E_j)$.
\item[(ii)] If $n_I\ge 1$, i.e. $\#I>1$, then $\mathrm{deg}(\mathcal{L}_j)=\mathrm{deg}(E_j)> 2g(Y)-2$, and $\HH^1(Y,\mathcal{L}_j)=\HH^1(Y,\mathcal{O}_Y(E_j))=0$.
\end{itemize}
\end{proposition}

\begin{proof}
Let $\mathcal{D}_{X/Y}^{-1}$ be the inverse different of $X$ over $Y$. Then $\mathcal{D}_{X/Y}^{-1}$ is a coherent $\mathcal{O}_X$-$G$-module that is a subsheaf of the constant sheaf $k(X)$. Moreover, 
$\mathcal{D}_{X/Y}^{-1}$ can be identified with the line bundle $\mathcal{O}_X(\mathrm{Ram}_\pi)$.

Fix $0\le j\le p^{n_I}-1$. Using the same arguments as in the proof of \cite[Prop. 4.1]{BleherChinburgKontogeorgis2020}, where we replace $\mathcal{D}^{-1}_{X/Y}$ by $\mathcal{O}_X(E)$, we obtain part (i).

To prove part (ii), we show that $\mathrm{deg}(E_j) > 2 g(Y)-2$ by using our assumption that $2g(X)-2<\mathrm{deg}(E)=\sum_{x\in X} e_x$. By the Riemann-Hurwitz formula (see also \cite[Eq. (4.15)]{BleherChinburgKontogeorgis2020}) and since $E$ is $G$-invariant and $a_{\ell,t}\le p-1$ for all $\ell,t$, we obtain 
\begin{eqnarray*}
2 g(Y)-2 
&=& \frac{1}{p^{n_I}}\left(2g(X)-2-\sum_{x\in X}\sum_{\ell=1}^{n_x}(p-1)p^{n_x-\ell}(i_\ell(x)+1)\right)\\
&<& \sum_{x\in X} \frac{e_x-\sum_{\ell=1}^{n_x}(p-1)p^{n_x-\ell}(i_\ell(x)+1)}{p^{n_I}}\\
&=& \sum_{y\in Y} \left(\frac{e_{x(y)}-\sum_{\ell=1}^{n_{x(y)}}(p-1)p^{n_{x(y)}-\ell}\,i_\ell(x(y))}{p^{n_{x(y)}}} - \frac{\sum_{\ell=1}^{n_{x(y)}}(p-1)p^{n_{x(y)}-\ell}}{p^{n_{x(y)}}} \right)\\
&\le& \sum_{y\in Y} \left( \frac{e_{x(y)}-\sum_{\ell=1}^{n_{x(y)}}a_{\ell,t}\,p^{n_{x(y)}-\ell}\,i_\ell(x(y))}{p^{n_{x(y)}}}-\frac{p^{n_{x(y)}}-1}{p^{n_{x(y)}}}\right)\\\
&\le& \sum_{y\in Y} e_{y,j} \;=\;\mathrm{deg}(E_j)
\end{eqnarray*}
where $x(y)$ stands for one particular point on $X$ above $y\in Y$, and we use that $(p-1)\sum_{\ell=1}^{n_{x(y)}}p^{n_{x(y)}-\ell}=p^{n_{x(y)}}-1$.
\end{proof}

The next step is to compare the filtration $\{\pi_*\mathcal{O}_X(E)^{(j)}\}_{j=0}^{\#I}$ of the $G$-sheaf $\pi_*\mathcal{O}_X(E)$ to the filtration $\{\HH^0(X,\mathcal{O}_X(E))^{(j)}\}_{j=0}^{\#I}$ of the $kG$-module $\HH^0(X,\mathcal{O}_X(E))$. 

We note that $\HH^1(Y,-)$ applied to the top filtered piece of $\pi_*\Omega_X$ does not vanish, and neither does $\HH^1(X, \Omega_X)$ (see \cite[Lemma 4.2]{BleherChinburgKontogeorgis2020}). On the other hand, if we replace $\Omega_X$ by $\mathcal{O}_X(E)$ as in Assumption \ref{ass:main3}, then part (ii) of Proposition \ref{prop:filtersheaf} shows that $\HH^1(Y,\mathcal{O}_Y(E_j))=0$ for all $0\le j\le p^{n_I}-1$. Additionally, our assumption that $\mathrm{deg}(E) > 2 g(X) -2$ implies that $\HH^1(X, \mathcal{O}_X(E))=0$. 

These two differences 
balance each other out so we are able to adapt the dimension argument in the proof of \cite[Lemma 4.2]{BleherChinburgKontogeorgis2020} 
to make it work
under Assumption \ref{ass:main3}. More precisely, replacing $\Omega_X$ by $\mathcal{O}_X(E)$ and using Proposition \ref{prop:filtersheaf} instead of \cite[Prop. 4.1]{BleherChinburgKontogeorgis2020}, the proof of the following result proceeds using the same main steps as in the proof of \cite[Lemma 4.2]{BleherChinburgKontogeorgis2020}.

\begin{lemma}
\label{lem:dimensionargument}
Suppose Assumption $\ref{ass:main3}$ holds and assume Notations $\ref{not:main3ind}$, $\ref{not:main3}$ and $\ref{not:another}$. For all $0 \le j \le p^{n_I} - 1$, there are isomorphisms
$$\HH^0(X,\mathcal{O}_X(E))^{(j+1)}/\HH^0(X,\mathcal{O}_X(E))^{(j)}\cong \HH^0(Y,\pi_*\mathcal{O}_X(E)^{(j+1)}/\pi_*\mathcal{O}_X(E)^{(j)}) \cong  S_{\chi^{-j}} \otimes_k \HH^0(Y, \mathcal{O}_Y(E_j))$$
of $k\overline{G}$-modules, where $E_j$ is the divisor from Theorem $\ref{thm:main3}$ defined by Equation $(\ref{eq:Ej})$.
\end{lemma}

We are now ready to prove Theorem \ref{thm:main3}, using Proposition \ref{prop:filtersheaf} and Lemma \ref{lem:dimensionargument}, together with the results of \cite[\S3]{Nakajima1986} (see also \cite[Thm. 4.5]{Kock2004}) applied to the tamely ramified cover $\lambda:Y\to Z$.

\begin{proof}[Proof of Theorem $\ref{thm:main3}$]
Under the assumptions of Theorem \ref{thm:main3}, fix $0\le j \le p^{n_I}-1$, and let $E_j$, $\ell_{y(z),j}$ and $n_j$ be as in Equations (\ref{eq:Ej}), (\ref{eq:ellzj}) and (\ref{eq:nj}), respectively. By \cite[Thm. 2]{Nakajima1986} and \cite[Thm. 4.5]{Kock2004} and using that $\HH^1(Y,\mathcal{O}_Y(E_j))=0$, we obtain that $\HH^0(Y,\mathcal{O}_Y(E_j))$ is a projective $k\overline{G}$-module whose Brauer character $\beta(j)$ is given as
\begin{equation}
\label{eq:betaj}
\beta(j)=\sum_{z\in Z_{\mathrm{br}}} \left( \sum_{d=1}^{\ell_{y(z),j}}\mathrm{Ind}_{\overline{G}_{y(z)}}^{\overline{G}} \theta_{y(z)}^{-d} -  \sum_{d=1}^{c_{y(z)}-1} \frac{d}{c_{y(z)}} \mathrm{Ind}_{\overline{G}_{y(z)}}^{\overline{G}} \theta_{y(z)}^{d}\right) + n_j\,\beta(k\overline{G})
\end{equation}
where $\beta(k\overline{G})$ is the Brauer character of  $k\overline{G}$. 

Let $y\in Y$ and $i\in\mathbb{Z}$. Since $c_y=\#\overline{G}_y$ is not divisible by $p$, $\theta_y^i$ is the character of a projective $k\overline{G}_y$-module, which implies that $\mathrm{Ind}_{\overline{G}_y}^{\overline{G}} \theta_y^i$ is the character of a projective $k\overline{G}$-module. Moreover, using Frobenius reciprocity and Notation \ref{not:main3}(d), we see that the socle of this projective $k\overline{G}$-module equals
$$\bigoplus_{\genfrac{}{}{0pt}{1}{0\le a\le c-1}{a\equiv i\,\varphi(y) \;\mathrm{mod}\; c_y}} S_a$$
when we view $S_a$ as a $k\overline{G}$-module. Therefore, the Brauer character $\beta(j)$ in Equation (\ref{eq:betaj}) is an explicitly given $\mathbb{Q}$-linear combination of Brauer characters of projective indecomposable $k\overline{G}$-modules, which implies that the rational coefficients must in fact be non-negative integers. It follows that, for all $0\le a\le c-1$, $n(a,j)$ from Equation (\ref{eq:naj}) equals the multiplicity of the projective indecomposable $k\overline{G}$-module with socle $S_a$ as a direct summand of the projective $k\overline{G}$-module $\HH^0(Y,\mathcal{O}_Y(E_j))$.

Let $M=\HH^0(X,\mathcal{O}_X(E))$, and let $0\le a\le c-1$ and $1\le b \le p^n$. Since by Lemma \ref{lem:dimensionargument},
$$M^{(j+1)}/M^{(j)}\cong   S_{\chi^{-j}} \otimes_k \HH^0(Y, \mathcal{O}_Y(E_j))$$
as $k\overline{G}$-modules, we have that $n(a,j)$ from Equation (\ref{eq:naj}) equals the multiplicity of the projective indecomposable $k\overline{G}$-module with socle $S_{\chi^{-j}.a}$ as a direct summand of the projective $k\overline{G}$-module $M^{(j+1)}/M^{(j)}$. 

Since the $k$-dimension of all projective indecomposable $k\overline{G}$-modules is equal to $p^{n-n_I}$, it follows that if $b$ is not a multiple of $p^{n-n_I}$ then $U_{a,b}$ is not a direct $kG$-module summand of $M$. On the other hand, if $b=(j+1)p^{n-n_I}$ then the description of the ascending composition factors of $U_{a,b}$ from Notation \ref{not:main3ind}(c) shows that $U_{a,b}^{(i+1)}/U_{a,b}^{(i)}$ is a projective $k\overline{G}$-module with socle $S_{\chi^{-i}.a}$ for $0\le i\le j$, and $U_{a,b}^{(i+1)}/U_{a,b}^{(i)}=0$ for $j+1\le i \le p^{n_I}-1$. Hence, $n(a,j)$ equals the number of indecomposable $kG$-module summands of $M=\HH^0(X,\mathcal{O}_X(E))$ whose socle equals $S_a$ and whose $k$-dimension is a multiple of $p^{n-n_I}$ that is greater than or equal to $(j+1)p^{n-n_I}$.

Therefore, for $0\le j\le p^{n_I}-2$, $n(a,j)-n(a,j+1)$ equals the number of indecomposable $kG$-module summands of $\HH^0(X,\mathcal{O}_X(E))$ that are isomorphic to $U_{a,(j+1)p^{n-n_I}}$, and $n(a,p^{n_I}-1)$ equals the number of indecomposable $kG$-module summands of $\HH^0(X,\mathcal{O}_X(E))$ that are isomorphic to $U_{a,\,p^{n_I}p^{n-n_I}}=U_{a,p^n}$. Since there are no additional indecomposable $kG$-modules occurring as direct summands of $\HH^0(X,\mathcal{O}_X(E))$, it follows that $n_{a,b}$ from Equation (\ref{eq:nab}) equals the number of indecomposable direct $kG$-module summands of $\HH^0(X,\mathcal{O}_X(E))$ that are isomorphic to $U_{a,b}$. This completes the proof of Theorem \ref{thm:main3}.
\end{proof}

Using some of the arguments from the proof of Theorem \ref{thm:main3} above, we obtain the following more explicit  version of the algorithm provided in \cite[Remark 4.4]{BleherChinburgKontogeorgis2020} concerning $\HH^0(X,\Omega_X)$. To distinguish this case from the case discussed in Theorem \ref{thm:main3}, we add decorations $\Omega$ as needed.

\begin{remark}
\label{rem:refereerequest!}
Suppose $k$, $G$ and $X$ are as in Assumption $\ref{ass:main3}$, and assume Notations $\ref{not:main3ind}$ and $\ref{not:main3}$. 

For $0\le j\le p^{n_I}-1$, define a $G$-invariant Weil divisor $D_j^\Omega=\sum_{y\in Y} d_{y,j}^\Omega\,y$ by
\begin{equation}
\label{eq:DjBCK}
d_{y,j}^\Omega = \left\lfloor \frac{\sum_{\ell = 1}^{n_{x(y)}} p^{n_{x(y)}-\ell}\,\left(p-1+(p-1-a_{\ell,t})\,i_\ell(x(y))\right)}{p^{n_{x(y)}}}
\right\rfloor.
\end{equation}
where $x(y)\in X$ is a point above $y$, $t\ge 0$ is the unique integer with $p^{n_I-n_{x(y)}}\,t \le j < p^{n_I-n_{x(y)}}(t+1)$, and $t = a_{1,t} + a_{2,t}\,p + \cdots + a_{n_{x(y)},t}\,p^{n_{x(y)}-1}$ is its $p$-adic expansion, with $a_{1,t},\ldots, a_{n_{x(y)},t}\in\{0,1,\ldots,p-1\}$.
For $z\in Z$, let $y(z)\in Y$ be such that $\lambda(y(z))=z$. Define $\ell_{y(z),j}^\Omega\in\{0,1,\ldots, c_{y(z)}-1\}$ by
\begin{equation}
\label{eq:ellzjBCK}
\ell_{y(z),j}^\Omega\equiv -d_{y(z),j}^\Omega\mod c_{y(z)},
\end{equation}
and define
\begin{equation}
\label{eq:njBCK}
n_j^\Omega = g(Z)-1+\sum_{z\in Z} \frac{d_{y(z),j}^\Omega+\ell_{y(z),j}^\Omega}{c_{y(z)}} .
\end{equation}
For $0\le a\le c-1$, define
\begin{equation}
\label{eq:najBCK}
n(a,j)^\Omega = \delta_{j,p^{n_I}-1}\cdot \delta_{a,0}+\sum_{z\in Z_{\mathrm{br}}} \left( \sum_{d=1}^{c_{y(z)}-1} \frac{d}{c_{y(z)}} \mu_{a,d}(y(z)) - \sum_{d=1}^{\ell_{y(z),j}^\Omega}\mu_{a,-d}(y(z)) \right) + n_j^\Omega ,
\end{equation}
and define $\mu_{a,\chi}$ to be 1 if $S_a \cong S_\chi$ and to be 0 otherwise.

Then, for $1\le b\le p^n$, the number $n_{a,b}^\Omega$ of indecomposable direct $kG$-module summands of $\HH^0(X,\Omega_X)$ that are isomorphic to $U_{a,b}$ is given as 
\begin{equation}
\label{eq:nabBCK}
n_{a,b}^\Omega=\left\{\begin{array}{cl}
n(a,j)^\Omega-n(a,j+1)^\Omega & \mbox{if $b=(j+1)\,p^{n-n_I}$ and $0\le j \le p^{n_I}-2$,}\\
\mu_{a,\chi}&\mbox{if $b=(p^{n_I}-1)\,p^{n-n_I}+1$ and $n_I<n$,}\\
n(a,p^{n_I}-1)^\Omega + (\delta_{n,n_I}-1)\cdot\mu_{a,\chi} & \mbox{if $b=p^{n_I}\,p^{n-n_I}=p^n$,}\\
0 &\mbox{otherwise.}
\end{array}\right.
\end{equation}
As in Remark \ref{rem:computation}, if $z\in Z-Z_\mathrm{br}$ then  $\ell_{y(z),j}^\Omega=0$ for $0\le j\le p^{n_I}-1$. Moreover, for $0\le j\le p^{n_I}-2$, 
$$n(a,j)^\Omega-n(a,j+1)^\Omega=-\delta_{j,p^{n_I}-2}\cdot \delta_{a,0}+\sum_{z\in Z_{\mathrm{br}}} \left( \sum_{d=1+\mathrm{min}\{\ell_{y(z),j}^\Omega,\ell_{y(z),j+1}^\Omega\}}^{\mathrm{max}\{\ell_{y(z),j}^\Omega,\ell_{y(z),j+1}^\Omega\}}\epsilon_{z,j}^\Omega\,\mu_{a,-d}(y(z))\right) + n_j^\Omega-n_{j+1}^\Omega$$
where $\epsilon_{j,z}^\Omega=1$ if $\ell_{y(z),j}^\Omega\le \ell_{y(z),j+1}^\Omega$ and $\epsilon_{j,z}^\Omega=-1$ otherwise.

To prove the formula in Equation (\ref{eq:nabBCK}), we note that the definitions in Equations (\ref{eq:DjBCK}), (\ref{eq:ellzjBCK}) and (\ref{eq:njBCK}) coincide with the corresponding definitions in \cite[Remark 4.4]{BleherChinburgKontogeorgis2020}, where we use the Riemann-Hurwitz formula to obtain Equation (\ref{eq:njBCK}). Let $M=\HH^0(X,\Omega_X)$, and let $0\le a\le c-1$ and $1\le b \le p^n$. 

Suppose first that $j<p^{n_I}-1$ or $n_I=n$. By Step (2) of \cite[Remark 4.4]{BleherChinburgKontogeorgis2020}, it follows that $M^{(j+1)}/M^{(j)}$ is a projective $k\overline{G}$-module. Moreover, since $S_{c-a}$ is the $k$-dual of $S_a$, we can use Step (2) of \cite[Remark 4.4]{BleherChinburgKontogeorgis2020}, together with similar arguments as in the proof of Theorem \ref{thm:main3} above, to see that $n(a,j)^\Omega$ from Equation (\ref{eq:najBCK}) equals the multiplicity of the projective indecomposable $k\overline{G}$-module with socle $S_{\chi^{-j}.a}$ as a direct summand of the projective $k\overline{G}$-module $M^{(j+1)}/M^{(j)}$. 

Suppose next that $j=p^{n_I}-1$ and $n_I<n$. Since the order of the character $\chi$ divides $p-1$, $\chi^{-(p^{n_I}-1)}$ is the trivial character. By Step (2) of \cite[Remark 4.4]{BleherChinburgKontogeorgis2020}, it follows that $M^{(p^{n_I})}/M^{(p^{n_I}-1)}$ is a direct sum of $S_\chi$ and a projective $k\overline{G}$-module. Moreover, we can use Step (2) of \cite[Remark 4.4]{BleherChinburgKontogeorgis2020} and similar arguments as in the proof of Theorem \ref{thm:main3} above, to see that $n(a,p^{n_I}-1)^\Omega-\mu_{a,\chi}$ equals the multiplicity of the projective indecomposable $k\overline{G}$-module with socle $S_{\chi^{-(p^{n_I}-1)}.a}=S_a$ as a direct summand of $M^{(p^{n_I})}/M^{(p^{n_I}-1)}$. Moreover, $n(a,p^{n_I}-1)^\Omega$ equals the multiplicity of $S_{\chi^{-(p^{n_I}-1)}.a}=S_a$ as a direct summand of the socle of $M^{(p^{n_I})}/M^{(p^{n_I}-1)}$. 

Using similar arguments as in the proof of Theorem \ref{thm:main3}, it follows that $n_{a,b}^\Omega$ from Equation (\ref{eq:nabBCK}) equals the number of indecomposable direct $kG$-module summands of $\HH^0(X,\Omega_X)$ that are isomorphic to $U_{a,b}$. 
\end{remark}

We end this section with a computational remark, which will be useful in the next sections.

\begin{remark}
\label{rem:combine}
Suppose Assumption \ref{ass:main3} holds, and assume Notations \ref{not:main3ind} and \ref{not:main3}. 
\begin{enumerate}
\item[(a)]
By applying the Riemann-Hurwitz formula, $n_j$ from Equation (\ref{eq:nj}) can also be written as
$$n_j = \frac{  \mathrm{deg}(E_j)+1-g(Y)}{\#\overline{G}}  + \sum_{z\in Z_{\mathrm{br}}} \frac{1}{c_{y(z)}}\left( \frac{c_{y(z)}-1}{2}-\ell_{y(z),j}\right)$$
which in some applications may be easier to use (see, for example, \S\ref{ss:N1structure}).
\item[(b)]
Let $y\in Y$ be a ramification point of the tame cover $\lambda:Y\to Z$, and let $x\in X$ be a point above it. If $I_{x}$ (resp. $G_{x}$) is the inertia group of $x$ inside $I$ (resp. $G$), we obtain $\overline{G}_y\cong G_{x}/I_{x}$. The fundamental character $\theta_x:G_{x}\to k(x)^*=k^*$ is given by $\theta_x(g) = \frac{g.\pi_x}{\pi_x}\mod \mathfrak{m}_{X,x}$ where $\pi_x$ is a uniformizer and $\mathfrak{m}_{X,x}$ is the maximal ideal of $\mathcal{O}_{X,x}$. Because $k^*$ has no non-identity elements of $p$-power order, $\theta_x$ factors through $G_{x}/I_{x}$ which we identify with $\overline{G}_y$. 
Since $\mathrm{ord}_x(\pi_y)=\#I_{x}$ and since $k(y)=k(x)=k$ is algebraically closed,
we can identify $\theta_y = \left(\theta_x\right)^{\#I_{x}}$ when viewing both as characters of $\overline{G}_y=G_{x}/I_{x}$.
\end{enumerate}
\end{remark}


\section{An application to deformation theory}
\label{s:deformation}

In this section, we show an application of Corollaries \ref{thm:oldmain} and \ref{thm:main3cor} to deformation theory and the equivariant deformation problem introduced in \cite{BertinMezard2000}. Let $X$ be a smooth projective curve defined over an algebraically closed field $k$ of prime characteristic $p$, and suppose $G$ is a finite group acting faithfully on the right on $X$ over $k$. Suppose $\mathcal{C}$ is the category of commutative local Artinian $k$-algebras with residue field $k$. Let $D_{\mathrm{gl}}:\mathcal{C}\to\mathrm{Sets}$ be the global deformation functor associated to the pair $(X,G)$ as introduced in \cite[\S 2]{BertinMezard2000}, and let $t_{D_{\mathrm{gl}}}$ be its tangent space. Using the arguments in \cite[\S3]{Kontogeorgis2007B} and in \cite[\S7]{KockTait2015}, we obtain, as in the beginning of the proof of \cite[Thm. 7.1]{KockTait2015}, that
\begin{equation}
\label{eq:dimtangent}
\mathrm{dim}_k\; t_{D_{\mathrm{gl}}} = \mathrm{dim}_k\; \HH^0(X,\Omega_X^{\otimes 2})_G
\end{equation}
where the subscript $G$ denotes the $G$-coinvariants of a $kG$-module.

\begin{proposition}
\label{prop:deformation}
Suppose $G$ has non-trivial cyclic Sylow $p$-subgroups. If $g(X)\ge 2$ then the $k$-dimension of the tangent space of the global deformation functor $D_{\mathrm{gl}}$ is uniquely determined by the class of canonical divisors on $X/G$ modulo principal divisors, together with the lower ramification groups and the fundamental characters of the closed points of $X$ that are ramified in the cover $X\to X/G$.

Under the hypotheses of Theorem $\ref{thm:main3}$,
we obtain
\begin{equation}
\label{eq:tangentdeformation}
\mathrm{dim}_k\; t_{D_{\mathrm{gl}}} = 
\sum_{j=0}^{p^{n_I}-1} n_{\chi^j.0,\,(j+1)p^{n-n_I}}
\end{equation}
where $\chi^j.0$ is as in Notation $\ref{not:main3ind}(c)$ and $n_{\chi^j.0,\,(j+1)p^{n-n_I}}$  is given by Equation $(\ref{eq:nab})$ from Theorem $\ref{thm:main3}$ applied to $E=2K_X$.
\end{proposition}

\begin{proof}
The first paragraph of the statement follows from Corollary \ref{thm:oldmain} using Equation (\ref{eq:dimtangent}). 

Assume now the hypotheses of Theorem $\ref{thm:main3}$ and let $E=2K_X$. By Equation (\ref{eq:nab}) in Theorem \ref{thm:main3}, the number of indecomposable direct $kG$-module summands of $\HH^0(X,\Omega_X^{\otimes 2})$ that are isomorphic to $U_{a,b}$ (for $0\le a \le c-1$ and $1\le b\le p^n$) is given by $n_{a,b}$. Since taking $G$-coinvariants defines an additive functor $(-)_G$ from the category of finitely generated $kG$-modules to the category of $k$-vector spaces, it follows from Equation (\ref{eq:dimtangent}) that 
$$\mathrm{dim}_k\; t_{D_{\mathrm{gl}}} = \sum_{a=0}^{c-1}\sum_{b=1}^{p^n} n_{a,b}\,\mathrm{dim}_k\;(U_{a,b})_G$$
where $n_{a,b}$ is as in Equation (\ref{eq:nab}) from Theorem $\ref{thm:main3}$ applied to $E=2K_X$. As discussed prior to Notation \ref{not:main3ind}, each $U_{a,b}$ is uniserial and its radical is equal to $(\sigma-1)U_{a,b}$. This implies that $(U_{a,b})_G\neq 0$ if and only if $U_{a,b}/\mathrm{rad}(U_{a,b})\cong S_0$ is the trivial simple $kG$-module. Using the description of the ascending composition factors of $U_{a,b}$ from Notation \ref{not:main3ind}(c), we have
$$U_{a,b}/\mathrm{rad}(U_{a,b})\cong S_{\chi^{-b+1}.a}$$
The latter module is isomorphic to $S_0$ if and only if $\chi^{-b+1}.a=0$, which is equivalent to $a=\chi^{b-1}.0$. We obtain 
$$\mathrm{dim}_k\; t_{D_{\mathrm{gl}}} = \sum_{b=1}^{p^n} n_{\chi^{b-1}.0,\,b}.$$
By Equation (\ref{eq:nab}) from Theorem $\ref{thm:main3}$, $n_{\chi^{b-1}.0,\,b}=0$ unless $b=(j+1)p^{n-n_I}$ for $0\le j\le p^{n_I}-1$. Since the order of the character $\chi$ divides $p-1$, it follows that $\chi^{p^{n-n_I}}=\chi$, which implies $\chi^{(j+1)p^{n-n_I}-1} =\chi^j$. Hence the formula in Equation (\ref{eq:tangentdeformation}) follows.
\end{proof}

\begin{remark}
\label{rem:deformation2}
Proposition \ref{prop:deformation} generalizes \cite[Cor. 2.3]{KockKontogeorgis2012} from cyclic $p$-groups to groups $G$ with cyclic Sylow $p$-subgroups, and it gives a formula when $G$ is $p$-hypo-elementary. In \cite[Thm. 7.1]{KockTait2015}, the $k$-dimension of $t_{D_{\mathrm{gl}}}$ has been determined for any finite group $G$ satisfying the additional assumption that $\mathrm{dim}_k\, M^G = \mathrm{dim}_k\, M_G$ for all finitely generated $kG$-modules $M$. However, for a $p$-hypo-elementary group $G$, this additional assumption does not hold in general.
\end{remark}


\section{A family of hyperelliptic curves}
\label{s:hyperelliptic}

In this section, we illustrate Corollaries \ref{thm:oldmain} and \ref{thm:main3cor} and Proposition \ref{prop:deformation} by considering a particular family of hyperelliptic curves. We freely utilize the notation introduced in \S\ref{s:intro} and \S \ref{s:mainthm} without referring to these sections explicitly. We make the following assumptions throughout this section.

\begin{assume}
\label{ass:hyperelliptic}
Let $m>1$ be an integer, let $p>3m$ be a prime number, and let $k$ be an algebraically closed field of characteristic $p$. Write the function field of $\mathbb{P}^1_k$ as $k(\mathbb{P}^1_k)=k(t)$, and consider the quadratic cover $X$ of $\mathbb{P}^1_k$ with function field $k(X)=k(t)(\sqrt{f})$ where 
$$f(t)=t^{p^2}-t=\prod_{\alpha\in\mathbb{F}_{p^2}}(t-\alpha).$$
Let $P=\langle \sigma\rangle$ be a cyclic group of order $p$, and let $C=\langle \rho\rangle$ be a cyclic group of order $2(p-1)$. Fix an isomorphism $\xi:C\xrightarrow{\;\cong\;} \mu_{2(p-1)}\subset k^*$ and define $\chi:C\to \mu_{p-1}=\mathbb{F}_p^*\subset k^*$ to be $\chi = \xi^2$. Define $G$ to be the semidirect product $G=P\rtimes_{\chi}C$, and let the generators $\sigma$ and $\rho$ of $G$ act on $k(X)$ as
$$\left\{\begin{array}{rclcrcl}
\sigma(t) &=& t+1,&\quad& \sigma(\sqrt{f}) &=&\sqrt{f},\\
\rho(t)&=&\chi^{-1}(\rho)\,t,&\quad& \rho(\sqrt{f}) &=&\xi^{-1}(\rho)\,\sqrt{f}.
\end{array}\right.$$
\end{assume}

In \S\ref{ss:datahyperelliptic}, we first determine the data needed to apply Theorem \ref{thm:main3} when $E=mK_X$. In \S\ref{ss:applyhyperelliptic}, we then use this theorem to determine the precise $kG$-module structure for $\HH^0(X,\Omega_X^{\otimes m})$. In \S\ref{ss:deformhyperelliptic}, we use Proposition \ref{prop:deformation} to determine the $k$-dimension of the tangent space $t_{D_{\mathrm{gl}}}$.


\subsection{Lower ramification groups, fundamental characters, and canonical divisors}
\label{ss:datahyperelliptic}

We first determine the lower ramification groups of the action of $G$ on $X$. Define $\nu=\rho^{p-1}\in C$. Then $\nu$ has order 2 and acts on $k(X)$ as $\nu(t)=t$ and $\nu(\sqrt{f}) = -\sqrt{f}$, which implies $k(X)^{\langle \nu\rangle}=k(t)$. Using that $\sigma(t^p-t)=t^p-t$ and that the character $\chi$ has order $p-1$, it moreover follows that $k(X)^{\langle \sigma,\nu\rangle}=k(t^p-t)$ and $k(X)^G=k( (t^p-t)^{p-1})$. Hence we can identify $X/\langle \nu\rangle$, $X/\langle \sigma,\nu\rangle$ and $X/G=Z$ with $\mathbb{P}^1_k$. We see that the cover $X\to X/\langle \nu \rangle$ has precisely $p^2+1$ branch points given by $\mathbb{A}^1(\mathbb{F}_{p^2})\cup\{\infty\}\subset X/\langle \nu\rangle$, which are all totally ramified. On the other hand, the cover $X/\langle \nu\rangle \to X/\langle \sigma,\nu\rangle$ has precisely one branch point, which is totally ramified. This branch point lies below $\infty\in X/\langle \nu\rangle$. Finally, the cover $X/ \langle \sigma,\nu\rangle\to X/G=Z$ has two branch points, which are totally ramified. These branch points lie below $\mathbb{A}^1(\mathbb{F}_p)\cup\{\infty\}\subset X/\langle \nu\rangle$. Hence, the branch points of the cover $X\to X/G=Z$ are the points on $Z$ lying below $\mathbb{A}^1(\mathbb{F}_{p^2})\cup\{\infty\}\subset X/\langle \nu\rangle$. We note that $\alpha^p-\alpha=(\Phi-1)(\alpha)$ for all $\alpha\in k$, where $\Phi$ is the Frobenius automorphism of $k$. Since $\Phi$ restricts to the unique order 2 automorphism of $\mathbb{F}_{p^2}$, it follows that $(\alpha^p-\alpha)^{p-1}$ equals $0$ (resp. $-1$) if $\alpha\in \mathbb{F}_p$ (resp.  $\alpha\in\mathbb{F}_{p^2}-\mathbb{F}_p$). Therefore, the cover $X\to X/G=Z$ has 3 branch points given by $z_0=0$, $z_1=-1$ and $z_\infty=\infty$. 

For $u\in\{0,1,\infty\}$, let $x_u\in X$ be above $z_u$. We claim we have the following lower ramification groups
\begin{equation}
\label{eq:ramhyper}
\left\{\begin{array}{c@{\;}c@{\;}c@{\;}c@{\;}cc@{\;}c@{\;}cc@{\;}c@{\;}cc@{\;}c@{\;}c}
G_{x_0}&=&G_{x_0,0}&=&\langle \rho\rangle,& G_{x_0,1}&=&1,\\
G_{x_1}&=&G_{x_1,0}&=&\langle \nu\rangle,& G_{x_1,1}&=&1,\\
G_{x_\infty}&=&G_{x_\infty,0}&=&G,&G_{x_\infty,1}&=&P, &G_{x_\infty,2}&=&P, & G_{x_\infty,3}&=&1.
\end{array}\right.
\end{equation}
The discussion in the previous paragraph gives us the inertia groups $G_{x_u}=G_{x_u,0}$ for $u\in\{0,1,\infty\}$. This leads to the first two rows of Equation (\ref{eq:ramhyper}) since the inertia groups of $x_0$ and $x_1$ are $p'$-groups. For the third row, we use that $f=t^{p^2+1}(t^{-1}-t^{-p^2})$ to obtain that $\pi_\infty = \sqrt{f}\, t^{-(p^2+1)/2}$ is a uniformizer of $\mathcal{O}_{X,x_\infty}$. We have
$$\frac{\sigma.\pi_\infty}{\pi_\infty} = \frac{\sqrt{f}\, (t+1)^{-(p^2+1)/2}}{\sqrt{f} \, t^{-(p^2+1)/2}} = 
(1+t^{-1})^{-(p^2+1)/2} = 1 - \frac{p^2+1}{2}\,t^{-1} + \cdots$$
and hence
$$\sigma . \pi_\infty \equiv \pi_\infty \mod \left(\frac{p^2+1}{2}\,t^{-1}\pi_\infty\right).$$
Since
$$\mathrm{ord}_{x_\infty}\left(\frac{p^2+1}{2}\,t^{-1}\pi_\infty\right) = 2+1 = 3,$$
we obtain $P_{x_\infty,i}=P$ for $i\in\{0,1,2\}$ and $P_{x_\infty,i}=1$ for $i\ge 3$, which leads to the third row of Equation (\ref{eq:ramhyper}).

Since $x_\infty\in X$ is the only point for which the order of its inertia group is divisible by $p$, it follows that $I=P$ is the greatest among the Sylow $p$-subgroups of the inertia groups of all closed points on $X$. Hence $I=I_{x_\infty}=P=\langle\sigma\rangle$ and $x_\infty$ is the unique ramification point of the cover $\pi:X\to Y=X/I$. 

For $u\in\{0,1,\infty\}$, let $y_u=\pi(x_u)$. If $\overline{G}=G/I$, the orbits $\overline{G}.y_0$ and $\overline{G}.y_\infty$ are each singletons, and the orbit $\overline{G}.y_1$ contains $p-1$ points of $Y$. The points in these orbits are the only ramification points of the cover $\lambda:Y\to Z=X/G$. Identifying $\overline{G}=\langle \rho\rangle$, we obtain the following inertia groups:
\begin{equation}
\label{eq:inertiatamehyperell}
\overline{G}_{y_0} = \langle \rho \rangle = \overline{G}_{y_\infty} 
\quad\mbox{and}\quad \overline{G}_{y_1} = \langle \nu \rangle.
\end{equation}
By the Riemann-Hurwitz formula, it follows that $g(Y)=(p-1)/2$. For $u\in\{0,1,\infty\}$, we obtain the following uniformizer $\pi_u$ of $\mathcal{O}_{X,x_u}$, where $\pi_\infty$ is as above:
$$\pi_0=\sqrt{f}=\pi_1\quad\mbox{and}\quad\pi_\infty=\sqrt{f} \,t^{-(p^2+1)/2}.$$
By Remark \ref{rem:combine}(b), for each $u\in\{0,1,\infty\}$, we can identify the corresponding fundamental character $\theta_{y_u}$ with $(\theta_{x_u})^{\#I_{x_u}}$ on $\overline{G}_{y_u}=G_{x_u}/I_{x_u}$. Considering the action of $\rho$ on the uniformizer $\pi_u$, this leads to the fundamental characters
\begin{equation}
\label{eq:fundamentalhyperell}
\left\{\begin{array}{c@{\;}c@{\;}l}
\theta_{y_0}&=&\xi^{-1},\\
\theta_{y_1}&=&\mathrm{Res}^{\langle\rho\rangle}_{\langle \nu\rangle}\left(\xi^{-1}\right),\\
\theta_{y_\infty}&=&\left(\xi^{-1}\cdot \chi^{(p^2+1)/2}\right)^p=\xi^{p^3}=\xi^p,
\end{array}\right.
\end{equation}
where we use that the order of $\xi$ is $2p-2$ and that $p^3\equiv p \mod 2p-2$. Using Equations ({\ref{eq:inertiatamehyperell}) and (\ref{eq:fundamentalhyperell}}), we obtain the following values for $c_{y_u}$ and $\varphi(y_u)$ from Notation \ref{not:main3}(d):
\begin{equation}
\label{eq:newformulahyper}
\left\{\begin{array}{r@{\;}c@{\;}lr@{\;}c@{\;}l}
c_{y_0} & = &  2p-2, & \varphi(y_0) &=& 2p-3,\\
c_{y_1} & = &  2, & \varphi(y_1) &=& 1,\\
c_{y_\infty} &=& 2p-2, & \varphi(y_{\infty}) &=& p.
\end{array}\right.
\end{equation}

Since $z_\infty=\infty$ on $Z=\mathbb{P}^1_k$, we have that $K_Z=-2\,z_\infty$ is a canonical divisor on $Z$. Therefore, $K_Y=\lambda^*(-2\,z_\infty) + \mathrm{Ram}_\lambda$ is a canonical divisor on $Y$. Using Equation (\ref{eq:inertiatamehyperell}) to determine $\mathrm{Ram}_\lambda$, we  obtain
\begin{equation}
\label{eq:KYhyperell}
K_Y = (2p-3)y_0 + \sum_{y \in \overline{G}.y_1} y \;-\; (2p-1)y_\infty.
\end{equation}


\subsection{The $kG$-module structure of the holomorphic poly-differentials}
\label{ss:applyhyperelliptic}

We use Corollary \ref{thm:main3cor}, and in particular Theorem \ref{thm:main3} applied to $E=mK_X$, to determine the $kG$-module structure for $\HH^0(X,\Omega_X^{\otimes m})$ for all $m>1$. As seen in \S\ref{ss:datahyperelliptic}, $I=P=\langle \sigma\rangle$, and hence $n_I=1$. Moreover, $y_\infty\in Y$ is the only branch point of $\lambda:X\to Y$, and $I_{x_\infty,i}=I$ for $i\in \{0,1,2\}$ and $I_{x_\infty,i}=1$ for $i\ge 3$. By Corollary \ref{thm:main3cor}, we have  $p$ divisors $D_j$ on $Y$, for $0\le j\le p-1$, which are given as $D_j=d_{y_\infty,j}\, y_\infty$ where
$$d_{y_\infty,j} = \left\lfloor \frac{3m(p-1)-2j}{p}\right\rfloor = \left\lfloor 3m - \frac{3m+2j}{p}\right\rfloor.$$
Since we assume $p>3m$, we can rewrite this as follows.
Define
\begin{equation}
\label{eq:deltahyperell}
\delta_m=\left\{\begin{array}{ll} 
0, & \mbox{$m$ even},\\
1, & \mbox{$m$ odd},
\end{array}\right.
\end{equation}
and define
\begin{eqnarray*}
A_1 &=& \left\{0,1,\ldots,\frac{p-3m-1+\delta_m}{2}\right\},\\
A_2 &=& \left\{\frac{p-3m-1+\delta_m}{2}+1,\ldots, \frac{2p-3m-\delta_m}{2}\right\},\\
A_3 &=& \left\{\frac{2p-3m-\delta_m}{2}+1,\ldots, p-1\right\}.
\end{eqnarray*}
Then $D_j = (3m-\ell)\,y_\infty$ if $j\in A_\ell$ for $1\le \ell \le 3$. 

We now apply Theorem \ref{thm:main3} to $E=mK_X$. We have $Z_{\mathrm{br}}=\{z_0,z_1,z_\infty\}$. Using that $E_j = mK_Y+D_j$, for $0\le j \le p-1$, together with Equations (\ref{eq:newformulahyper}) and (\ref{eq:KYhyperell}), we obtain, for $0\le j\le p-1$, 
$$\ell_{y_u,j}=\left\{ \begin{array}{ll}
2p-2-m,&\mbox{$u=0$},\\[1ex]
\delta_m, & \mbox{$u=1$},\\[1ex]
2m-\ell, & \mbox{$u=\infty$, $j\in A_\ell$ for $1\le \ell \le 3$}.
\end{array}\right.$$
Since $\ell_{y,j}=0$ for all $y\in Y$ above $z\in Z- Z_{\mathrm{br}}$, we have, using Equation (\ref{eq:KYhyperell}), 
$$n_j = 1-0 + \frac{m(2p-3)-(2p-2-m)}{2p-2} + \frac{m-\delta_m}{2} + \frac{-m(2p-1)+(3m-\ell)-(2m-\ell)}{2p-2}=\frac{m-\delta_m}{2}.$$

Using Equation (\ref{eq:naj}) and Remark \ref{rem:computation}, it follows immediately that 
$$n_{a,b}=0\qquad\mbox{for $b\not\in\{\frac{p-3m-1+\delta_m}{2}+1,\frac{2p-3m-\delta_m}{2}+1,p\}$}.$$
Additionally,
$$n_{a,\frac{p-3m+1+\delta_m}{2}}=\mu_{a,-(2m-1)}(y_\infty)\quad\mbox{and}\quad
n_{a,\frac{2p-3m+2-\delta_m}{2}}=\mu_{a,-(2m-2)}(y_\infty).$$
By Equation (\ref{eq:newformulahyper}), $\mu_{a,-2m+1}(y_\infty)=1$ if and only if $a\equiv (-2m+1)p \mod (2p-2)$ and $\mu_{a,-2m+2}(y_\infty)=1$ if and only if $a\equiv (-2m+2)p \mod (2p-2)$. Since $(-2m+1)p=-m(2p-2)+p-2m$ and $(-2m+2)p=-m(2p-2)+2p-2m$, the only non-projective indecomposable $kG$-modules that are direct summands of $\HH^0(X,\Omega_X^{\otimes m})$ are $U_{p-2m,\frac{p-3m+1+\delta_m}{2}}$ and $U_{2p-2m,\frac{2p-3m+2-\delta_m}{2}}$ and they each occur with multiplicity one.

It remains to determine the number $n_{a,p}$ of projective indecomposable direct $kG$-module summands of $\HH^0(X,\Omega_X^{\otimes m})$ of the form $U_{a,p}$, for $0\le a\le 2p-3$. Since $n_{a,p}=n(a,p-1)$, we obtain, using Equation (\ref{eq:naj}),
\begin{eqnarray*}
n_{a,p}&=& \sum_{u\in\{0,1,\infty\}}\left( \sum_{d=1}^{\ell_{y_u,p-1}}\mu_{a,-d}(y_u) -  \sum_{d=1}^{c_{y_u}-1} \frac{d}{c_{y_u}}\, \mu_{a,d}(y_u)\right) + n_{p-1}\\
&=& \left(\sum_{d=1}^{2p-2-m} \mu_{a,-d}(y_0) -\sum_{d=1}^{2p-3}\frac{d}{2p-2}\,\mu_{a,d}(y_0)\right)
+ \left(\delta_m\,\mu_{a,-1}(y_1) - \frac{1}{2}\,\mu_{a,1}(y_1)\right)\\
&&+ \left(\sum_{d=1}^{2m-3}\mu_{a,-d}(y_\infty)  - \sum_{d=1}^{2p-3}\frac{d}{2p-2}\,\mu_{a,d}(y_\infty) \right)
+ \frac{m-\delta_m}{2}.
\end{eqnarray*}
Since $2p-3\equiv -1 \mod (2p-2)$, we obtain, by Equation (\ref{eq:newformulahyper}), that $\mu_{a,-d}(y_0)=1$ if and only if $a=d$ and $\mu_{a,d}(y_0)=1$ if and only if $a=2p-2-d$. Moreover, $\mu_{a,-1}(y_1)=\mu_{a,1}(y_1)=1$ if and only if $a$ is odd. To determine, $\mu_{a,-d}(y_\infty)$ and $\mu_{a,d}(y_\infty)$, we use  the following congruences:
$$\begin{array}{rccll}
-t&\equiv&2p-2-t&\mod (2p-2),&\mbox{ for $0\le t\le 2p-3$},\\
p(2i)&\equiv&2i&\mod (2p-2),&\mbox{ for $0\le i\le p-2$},\\
p(2i+1)&\equiv&2i+p&\mod (2p-2),&\mbox{ for $0\le i\le \frac{p-3}{2}$},\\
p(2i+1)&\equiv&2i+2-p&\mod (2p-2),&\mbox{ for $\frac{p-1}{2}\le i\le p-2$}.
\end{array}$$
For $0\le i\le p-2$, this implies  $\mu_{a,-2i}(y_\infty)=1$ if and only if $a=2p-2-2i$, and $\mu_{a,2i}(y_\infty)=1$ if and only if $a=2i$. Similarly, for $0\le i\le \frac{p-3}{2}$,  we obtain $\mu_{a,-(2i+1)}(y_\infty)=1$ if and only if $a=p-2-2i$, and $\mu_{a,2i+1}(y_\infty)=1$ if and only if $a=2i+p$. Finally, for $\frac{p-1}{2}\le i\le p-2$, we have $\mu_{a,-(2i+1)}(y_\infty)=1$ if and only if $a=3p-4-2i$ and $\mu_{a,2i+1}(y_\infty)=1$ if and only if $a=2i+2-p$.

Suppose first that $0\le i\le p-2$. Then
$$n_{2i,p}=\left(1-\frac{2p-2-2i}{2p-2}\right) +(0-0)+\left(1-\frac{2i}{2p-2}\right)+\frac{m-\delta_m}{2} = 1+\frac{m-\delta_m}{2}$$
provided $2i\le 2p-2-m$ and $2p-2-2i\le 2m-3$, which is equivalent to $p-m+1\le i\le p-1-\frac{m+\delta_m}{2}$. For all other $i$ with $0\le i\le p-2$, we have to subtract 1, leading to $n_{2i,p}=\frac{m-\delta_m}{2}$.

Suppose next that $0\le i\le \frac{p-3}{2}$. Since we have assumed that $p>3m$, it follows that $2i+1\le 2p-2-m$. Therefore,
$$n_{2i+1,p}=\left(1-\frac{2p-2-2i-1}{2p-2}\right) +\left(\delta_m-\frac{1}{2}\right)+\left(1-\frac{2i+p}{2p-2}\right)+\frac{m-\delta_m}{2} = \frac{m+\delta_m}{2}$$
provided $p-2-2i\le 2m-3$, which is equivalent to $\frac{p+1}{2}-m\le i$. For all other $i$ with $0\le i\le \frac{p-3}{2}$, we have to subtract 1, leading to $n_{2i+1,p}=\frac{m+\delta_m}{2}-1$.

Finally, suppose that $\frac{p-1}{2}\le i\le p-2$. Since we have assumed that $p>3m$, it follows that $3p-4-2i> 2m-3$. Therefore,
$$n_{2i+1,p}=\left(1-\frac{2p-2-2i-1}{2p-2}\right) +\left(\delta_m-\frac{1}{2}\right)+\left(0-\frac{2i+2-p}{2p-2}\right)+\frac{m-\delta_m}{2} = \frac{m+\delta_m}{2}$$
provided $2i+1\le 2p-2-m$, which is equivalent to $i\le p-2-\frac{m-\delta_m}{2}$. For all other $i$ with $\frac{p-1}{2}\le i\le p-2$, we have to subtract 1, leading to $n_{2i+1,p}=\frac{m+\delta_m}{2}-1$.

We obtain the following result.

\begin{proposition}
\label{prop:hyperell}
Under Assumption $\ref{ass:hyperelliptic}$, there is an isomorphism of $kG$-modules
\begin{eqnarray*}
\HH^0(X,\Omega_X^{\otimes m}) &\cong &
U_{p-2m,\frac{p-3m+1+\delta_m}{2}} \;\oplus\; U_{2p-2m,\frac{2p-3m+2-\delta_m}{2}}\\
&&\oplus\; \bigoplus_{i\in E_1}\left(\frac{m-\delta_m}{2}+1\right)U_{2i,p} \;\oplus\;
\bigoplus_{i\in E_2}\left(\frac{m-\delta_m}{2}\right)U_{2i,p} \\
&&\oplus\; \bigoplus_{i\in O_1}\left(\frac{m+\delta_m}{2}\right)U_{2i+1,p} \;\oplus\;
\bigoplus_{i\in O_2}\left(\frac{m+\delta_m}{2}-1\right)U_{2i+1,p}
\end{eqnarray*}
where $\delta_m$ is as in Equation $(\ref{eq:deltahyperell})$ and 
$$\begin{array}{rclrcl}
E_1&=&\{p-m+1,p-m+2,\ldots, p-1-\frac{m+\delta_m}{2}\},&
E_2&=&\{0,1,\ldots,p-2\}-E_1,
\\[1ex]
O_1&=&\{\frac{p+1}{2}-m,\frac{p+1}{2}-m+1,\ldots, p-2-\frac{m-\delta_m}{2}\},&
O_2&=&\{0,1,\ldots,p-2\}-O_1.
\end{array}$$
\end{proposition}


\subsection{The tangent space of the global deformation functor}
\label{ss:deformhyperelliptic}

Under Assumption \ref{ass:hyperelliptic}, we now use Proposition \ref{prop:deformation} to determine the dimension of the tangent space $t_{D_{\mathrm{gl}}}$ of the global deformation functor $D_{\mathrm{gl}}$ associated to the pair $(X,G)$. We use Proposition \ref{prop:hyperell} for $m=2$.

Since $\chi=\xi^2$, we have $\chi^j.0=2j$ for $0\le j\le p-2$ and $\chi^{p-1}.0=0$. Because $n=n_I=1$, Equation (\ref{eq:tangentdeformation}) becomes
$$\mathrm{dim}_k\; t_{D_{\mathrm{gl}}} = \sum_{j=0}^{p-2} n_{2j,j+1} \; + \; n_{0,p}.$$
By Proposition \ref{prop:hyperell}, applied to $m=2$, the coefficient $n_{0,p}$ of $U_{0,p}$ equals $\frac{m-\delta_m}{2}=1$. Moreover, the only non-zero coefficients of non-projective direct $kG$-module summands of $\HH^0(X,\Omega_X^{\otimes 2})$ are
$$n_{p-4,\frac{p-5}{2}}\quad\mbox{and}\quad n_{2p-4,p-2}.$$
Since $p-4$ is odd and since $2p-4=2(p-2)$ but $p-2\ne (p-2)+1$, we obtain the following result.

\begin{proposition}
\label{prop:deformhyperell}
Suppose Assumption $\ref{ass:hyperelliptic}$ holds, and let $P_0$ be the projective $kG$-module cover of the trivial simple $kG$-module $S_0$. The $k$-dimension of the tangent space $t_{D_{\mathrm{gl}}}$ of the global deformation functor $D_{\mathrm{gl}}$ associated to the pair $(X,G)$ is equal to the multiplicity of $P_0$ as a direct summand of $\HH^0(X,\Omega_X^{\otimes 2})$, which equals one.
\end{proposition}


\section{Holomorphic poly-differentials of the modular curves $X(\Gamma_\ell)$ modulo $p$}
\label{s:modular}

In this section and the next, we consider modular curves in prime characteristic $p$. We show that Corollary \ref{thm:oldmain} has applications to congruences between prime level holomorphic cusp forms in characteristic $0$ of even weight $2m>2$ (see Theorems \ref{thm:firstmodtheorem} and \ref{thm:secondmodtheorem}). To prove these results, we give a solution to Hecke's classical problem \cite{Hecke1928,ChevalleyWeil1934} of determining the module structure of  such cusp forms when the characteristic is $p=3$ (see Theorem \ref{thm:modularresult} and Propositions \ref{prop:full_mixed} and \ref{prop:full_equal}). Together with the results in \cite{BleherChinburgKontogeorgis2020}, and using \cite{Kani1986} or \cite{Nakajima1986} for $p>3$, it follows that the only remaining open case for Hecke's problem in all characteristics is when $p = 2$.

For the remainder of this section, we will use the following assumptions and notations (see \cite[\S5]{BleherChinburgKontogeorgis2020}).

\begin{assume}
\label{ass:modular}
We assume $\ell \ge 3$ is a prime number with $\ell\ne p$. Let $F$ be a number field that is unramified over $p$ and that contains a primitive $\ell^{\mathrm{th}}$ root of unity $\zeta_\ell$. We assume $A\subset F$ is a Dedekind domain whose fraction field equals $F$ and that contains $\mathbb{Z}[\frac{1}{\ell},\zeta_\ell]$. Let $\mathcal{V}(F,p)$ be the set of places of $F$ over $p$. For each $v\in \mathcal{V}(F,p)$, let $\mathcal{O}_{F,v}$ be the ring of integers of the completion $F_v$ of $F$ at $v$ and let $\mathfrak{m}_{F,v}$ be its maximal ideal. For all $v \in \mathcal{V}(F,p)$, we assume that $A \subset \mathcal{O}_{F,v}$, and we let $k(v)$ denote the (finite) residue field of $A$ modulo the maximal ideal $A\cap\mathfrak{m}_{F,v}$. 
\end{assume}

\begin{nota}
\label{not:modular}
Suppose Assumption \ref{ass:modular} holds.
Let $\Gamma=\mathrm{SL}(2,\mathbb{Z})$, and let $\Gamma_\ell$ be the principal congruence subgroup of $\Gamma$ of level $\ell$.
Let $\overline{\Gamma}=\mathrm{SL}(2,\mathbb{F}_\ell) = \Gamma/\Gamma_\ell$, and let $G=\mathrm{PSL}(2,\mathbb{F}_\ell)=\Gamma/\langle \Gamma_\ell, \pm \,\mathrm{I}\,\rangle$ where $\mathrm{I}$ denotes the $2\times 2$ identity matrix. 
\end{nota}

Let $X(\Gamma_\ell)$ be the (compactified) modular curve associated to  $\Gamma_\ell$. It follows, for example, from \cite{Katz1973,KatzMazur1985}  that there is a proper smooth canonical model $\mathcal{X}_A(\ell)$ of $X(\Gamma_\ell)$ over $A$.  Moreover, $\HH^0(\mathcal{X}_A(\ell),\Omega_{\mathcal{X}_A(\ell)}^{\otimes m})$ is naturally identified with the $A$-lattice $\mathcal{S}_{2m}(A)$ of holomorphic  cusp forms for $\Gamma_\ell$ of weight $2m$ that have  $q$-expansion coefficients in $A$ at all the cusps, in the sense of \cite[\S1.6]{Katz1973}. 

Let $v \in \mathcal{V}(F,p)$ and let $k$ be an algebraically closed field containing $k(v)$. We define
\begin{equation}
\label{eq:reductionmodular}
X_p(\ell):=k\otimes_{k(v)}(k(v)\otimes_A\mathcal{X}_A(\ell))
\end{equation} 
to be the reduction of $\mathcal{X}_A(\ell)$ modulo $p$ over $k$ relative to $v$. 

By \cite[Thm. 1.1]{BendingCaminaGuralnick2005}, if $\ell\ge 7$ then $\mathrm{Aut}(X_p(\ell))=G$ unless $p=3$ and $\ell\in\{7,11\}$. Moreover, $\mathrm{Aut}(X_3(7))\cong \mathrm{PGU}(3,\mathbb{F}_3)$ and $\mathrm{Aut}(X_3(11))\cong M_{11}$. If $\ell<7$ then $g(X_p(\ell))=0$, whereas if $\ell\ge 7$ then $g(X_p(\ell))\ge 3$  (see, for example, \cite[Cor. 3.2]{BendingCaminaGuralnick2005} for an explicit formula). 

In \cite[\S 5.6]{Moreno1993}, it is shown that the genus of $X_p(\ell)/G$ is zero, and the lower ramification groups associated to the cover $X_p(\ell) \rightarrow X_p(\ell)/G$ are determined. If $p>3$ then the cover $X_p(\ell)\to X_p(\ell)/G$ is tamely ramified, and if $p\in\{2,3\}$ then the cover has wildly ramified points. While the Sylow $3$-subgroups of $G$ are cyclic, its Sylow $2$-subgroups are dihedral $2$-groups.


\subsection{Local results and congruences}
\label{ss:modularresults}

Let $v \in \mathcal{V}(F,p)$, let $k$ be an algebraically closed field containing $k(v)$, and let  $X_p(\ell)$ be as in Equation (\ref{eq:reductionmodular}).
We obtain the following result about the structure of the holomorphic poly-differentials of $\mathcal{X}_A(\ell)$ viewed as a module for $G$ over the local ring $\mathcal{O}_{F,v}$.

\begin{theorem} 
\label{thm:firstmodtheorem}
Suppose Assumption $\ref{ass:modular}$ holds, and let $p\ge 3$ and $m>1$.
The $\mathcal{O}_{F,v}G$-module
$$\mathcal{O}_{F,v} \otimes_A \HH^0(\mathcal{X}_A(\ell),\Omega^{\otimes m}_{\mathcal{X}_A(\ell)}) = \mathcal{O}_{F,v}\otimes_A \mathcal{S}_{2m}(A)$$
is a direct sum over blocks $B$ of $\mathcal{O}_{F,v}G$ of modules of the form $P_B \oplus U_B$ in which $P_B$ is a projective $B$-module and $U_B$ is either the zero module or a single indecomposable non-projective $B$-module.  If $p>3$ then $U_B=\{0\}$. For all $p\ge 3$, $P_B$ is uniquely determined by the lower ramification groups and the fundamental characters of the closed points of $X_p(\ell)$ that are ramified in the cover $X_p(\ell)\to X_p(\ell)/G$. Moreover, if $p=3$ then also the reduction $\overline{U}_B$ of $U_B$ modulo $\mathfrak{m}_{F,v}$ is uniquely determined this way.
\end{theorem}

Theorem \ref{thm:firstmodtheorem} extends \cite[Thm. 1.2]{BleherChinburgKontogeorgis2020} from $m=1$ to arbitrary $m>1$. Since $X_p(\ell)/G\cong\mathbb{P}^1_k$, we do not need to refer to the class of canonical divisors on $X_p(\ell)/G$. 

Similarly to \cite{BleherChinburgKontogeorgis2020}, we can use the approach in \cite{Ribet1984} to define congruences modulo $p$ between modular forms. 
Namely, let $\mathbb{T}$ be a ring of Hecke operators acting on $\mathcal{S}_{2m}(F):= F \otimes_A \mathcal{S}_{2m}(A)$. By  \cite[\S6.1-6.2]{Shimura1971}, together with flat base change, it follows that $\mathcal{S}_{2m}(F)$ coincides with the space of all weight $2m$ cusp forms for $\Gamma_\ell$ whose Fourier expansions with respect to $e^{2\pi iz/\ell}$ have coefficients in $F$. Suppose there exists a decomposition
\begin{equation}
\label{eq:decomp}
\mathcal{S}_{2m}(F) = E_1 \oplus E_2
\end{equation}
into a direct sum of $\mathbb{T}$-stable $F$-subspaces. Let $\mathfrak{a}$ be an ideal of $A$.  Following \cite{Ribet1984}, we define a non-trivial congruence modulo $\mathfrak{a}$ linking $E_1$ and $E_2$ to be a pair $(f_1,f_2)$ of forms $f_i \in \mathcal{S}_{2m}(A) \cap E_i$, for $i\in\{1,2\}$, such that 
\begin{equation}
\label{eq:congruence}
f_1 \equiv f_2 \mod \mathfrak{a} \cdot \mathcal{S}_{2m}(A) \quad \mathrm{but} \quad f_1 \not \in \mathfrak{a}\cdot \mathcal{S}_{2m}(A).
\end{equation}
Congruences of this kind have played an important role in the development of the theory of modular forms, Galois representations and arithmetic geometry (see, for example \cite{Diamond1997,DieulefaitUrrozRibet2015}).

We call a $\mathbb{T}$-stable decomposition into $F$-subspaces of the form given in Equation (\ref{eq:decomp}) $G$-isotypic if there exist two orthogonal central idempotents of $FG$ such that $1=e_1+e_2$ in $FG$ and 
\begin{equation}
\label{eq:centralidem}
E_i = e_i  \mathcal{S}_{2m}(F) \quad \mbox{for} \quad i\in\{1,2\}.
\end{equation} 

The following theorem extends \cite[Thm. 1.3]{BleherChinburgKontogeorgis2020} from $m = 1$ to arbitrary $m > 1$.

\begin{theorem}
\label{thm:secondmodtheorem}
With the assumptions of Theorem $\ref{thm:firstmodtheorem}$, suppose further that $F$ contains a root of unity of order equal to the prime-to-$p$ part of the order of $G$.  Let $\mathfrak{a}$ be the maximal ideal over $p$ in $A$ associated to $v \in \mathcal{V}(F,p)$.  A $\mathbb{T}$-stable decomposition of $\mathcal{S}_{2m}(F)$, as given in Equation $(\ref{eq:decomp})$, that is $G$-isotypic, in the sense of Equation $(\ref{eq:centralidem})$, results in non-trivial congruences modulo $\mathfrak{a}$, as given in Equation $(\ref{eq:congruence})$, if and only if the following is true. There is a block $B$ of $\mathcal{O}_{F,v}G$ such that when $P_B$ and $U_B$ are as in Theorem $\ref{thm:firstmodtheorem}$, $M_B = P_B\oplus U_B$ is not equal to the direct sum $(M_B \cap e_1 M_B) \oplus (M_B \cap e_2 M_B)$.  For a given $B$, there will be orthogonal idempotents $e_1$ and $e_2$ for which this is true if and only if $B$ has non-trivial defect groups, and either $P_B \ne \{0\}$ or $F_v \otimes_{\mathcal{O}_{F,v}} U_B$ has two non-isomorphic irreducible constituents.
\end{theorem}

We will prove Theorems \ref{thm:firstmodtheorem} and \ref{thm:secondmodtheorem} for $p>3$ in \S\ref{ss:proofgreater3}. The case when $p=3$ will be discussed in \S\ref{ss:proof3} and \S\ref{s:modular3}, using Corollaries \ref{thm:oldmain} and \ref{thm:main3cor}, together with Theorem \ref{thm:main3} applied to $E=mK_X$.

\subsection{Isotypic Hecke stable decompositions of even weight cusp forms}
\label{ss:Hecke}

In this subsection, we extend the results of \cite[\S7]{BleherChinburgKontogeorgis2020} to  construct non-trivial $G$-isotypic $\mathbb{T}$-stable decompositions of the space of cusp forms of even weight $2m>2$ when $\mathbb{T}$ is the ring of Hecke operators that have index prime to the level $\ell$. 

Let
$\Delta'_\ell$ be the set of all matrices $\alpha\in\mathrm{Mat}(2,\mathbb{Z})$ such that $\mathrm{det}(\alpha)>0$ and $\alpha \equiv \left(\begin{smallmatrix}1&0\\0&x\end{smallmatrix}\right)$ mod $\ell$ for some $x\in(\mathbb{Z}/\ell)^*$. Let
$R(\Gamma_\ell,\Delta_\ell')$ be the free $\mathbb{Z}$-module generated by the double cosets $\Gamma_\ell \alpha  \Gamma_\ell$ for $\alpha\in\Delta'_\ell$. By \cite[\S3.1]{Shimura1971}, $R(\Gamma_\ell,\Delta_\ell')$ is a ring. For each positive integer $n$ that is relatively prime to $\ell$, let $\rho'_\ell(n)$ be a set of representatives $\alpha\in\Delta'_\ell$ of all distinct double cosets in $\Gamma_\ell\backslash\Delta'_\ell/\Gamma_\ell$ such that $\mathrm{det}(\alpha)=n$, and define $T'(n):=\sum_{\alpha\in\rho'_\ell(n)} \,\Gamma_\ell \alpha \Gamma_\ell$. By \cite[Thm. 3.34]{Shimura1971}, $\mathbb{T}:= R(\Gamma_\ell,\Delta_\ell')\otimes_{\mathbb{Z}}\mathbb{Q}$ is the $\mathbb{Q}$-algebra generated by all $T'(n)$ when $n$ ranges over all positive integers that are relatively prime to $\ell$. 

Suppose $m>1$,
and
let $f \in \mathcal{S}_{2m}(F)$. For a matrix $\gamma=\left(\begin{smallmatrix} a&b\\c&d\end{smallmatrix}\right)\in\mathrm{GL}(2,\mathbb{Q})$ and $z$ in the complex upper half plane, we define
\begin{equation}
\label{eq:action}
(f|\gamma )(z) := \mathrm{det}(\gamma)^m\,(cz+d)^{-2m} \,f\left(\frac{az+b}{cz+d}\right).
\end{equation}
This leads to well-defined right actions by $\overline{\Gamma}=\mathrm{SL}(2,\mathbb{F}_\ell)$ and $G=\mathrm{PSL}(2,\mathbb{F}_\ell)$
on $\mathcal{S}_{2m}(F)$,
which can be made into left actions by defining the left action of a group element to be the right action of its inverse.
Moreover, Equation (\ref{eq:action}) leads to the following right action of $R(\Gamma_\ell,\Delta'_\ell)$, and hence of $\mathbb{T}$, on $\mathcal{S}_{2m}(F)$. For $\alpha \in \Delta_\ell'$, write $\Gamma_\ell \alpha \Gamma_\ell = \bigcup_i \Gamma_\ell \alpha_i$ as a finite disjoint union of right cosets, and define
$$f\big|\,\Gamma_\ell \alpha \Gamma_\ell := \mathrm{det}(\alpha)^{m-1}\cdot \textstyle \sum_i f|\alpha_i.$$
With these definitions, we can use similar arguments as in \cite[\S7]{BleherChinburgKontogeorgis2020} to obtain the following result.

\begin{proposition} 
\label{prop:Heckeresult} 
Let $m>1$, and suppose $e_1, e_2$ are orthogonal central idempotents of $FG$ such that $1 = e_1 + e_2$ and each $e_i$ is fixed by the conjugation action of $\mathrm{PGL}(2,\mathbb{F}_\ell)$ on $G$.  Then setting $E_i = \mathcal{S}_{2m}(F) e_i$ for $i \in\{1,2\}$ gives a $G$-isotypic $\mathbb{T}$-stable decomposition of $\mathcal{S}_{2m}(F)$, as defined in Equations $(\ref{eq:decomp})$ and $(\ref{eq:centralidem})$.
\end{proposition}

\subsection{Proofs of Theorems $\ref{thm:firstmodtheorem}$ and $\ref{thm:secondmodtheorem}$ when $p>3$}
\label{ss:proofgreater3}

These proofs follow the same main steps as the proofs of \cite[Thms. 1.2 and 1.3]{BleherChinburgKontogeorgis2020}, where we replace \cite[Lemma 5.2]{BleherChinburgKontogeorgis2020} by the following result.

\begin{lemma}
\label{lem:tamemodular}
Suppose $m>1$, $p>3$ and $p\neq \ell\ge 7$. Let $v \in \mathcal{V}(F,p)$, let $k$ be an algebraically closed field containing $k(v)$, and let  $X:=X_p(\ell)$ be as in Equation $(\ref{eq:reductionmodular})$.
\begin{itemize}
\item[(i)] The $kG$-module $\HH^0(X,\Omega_X^{\otimes m})$ is projective.
\item[(ii)] Let $k_1$ be a perfect field containing $k(v)$, and let $k$ be an algebraic closure of $k_1$. Define $X_1:=k_1\otimes_{k(v)}(k(v)\otimes_A\mathcal{X}_A(\ell))$. The $k_1G$-module $\HH^0(X_1,\Omega_{X_1}^{\otimes m})$ is projective.
\end{itemize}
The $kG$-module structure of $\HH^0(X,\Omega_X^{\otimes m})$ as in $(i)$ and the $k_1G$-module structure of  $\HH^0(X_1,\Omega_{X_1}^{\otimes m})$ as in $(ii)$ are both determined by the lower ramification groups and the fundamental characters associated to the cover $X\to X/G$.
\end{lemma}

Lemma \ref{lem:tamemodular} is proved using similar arguments to the ones used in the proof of \cite[Lemma 5.2]{BleherChinburgKontogeorgis2020}. The main difference is that since $m>1$ and since $g(X)\ge 3$ by \cite[Cor. 3.2]{BendingCaminaGuralnick2005}, it follows that $\mathrm{deg}(\Omega_X^{\otimes m})>2g(X)-2$, which implies $\HH^1(X,\Omega_X^{\otimes m})=0$. This forces $\HH^0(X,\Omega_X^{\otimes m})$ to be a projective $kG$-module by \cite[Thm. 2]{Nakajima1986}.

\subsection{Proofs of Theorems $\ref{thm:firstmodtheorem}$ and $\ref{thm:secondmodtheorem}$ when $p=3$}
\label{ss:proof3}

These proofs follow the same main steps as the proofs of \cite[Thms. 1.2 and 1.3]{BleherChinburgKontogeorgis2020}, where we replace \cite[Thm. 1.1 and Remark 4.4]{BleherChinburgKontogeorgis2020} by Corollaries \ref{thm:oldmain} and \ref{thm:main3cor}, and we replace \cite[Thm. 1.4]{BleherChinburgKontogeorgis2020} by Theorem \ref{thm:modularresult} below, which gives a description of the $kG$-module structure of the holomorphic $m$-differentials of $X=X_3(\ell)$ when $m>1$. 

The proof of Theorem \ref{thm:modularresult} is more complicated than the proof of \cite[Thm. 1.4]{BleherChinburgKontogeorgis2020}, for the following reasons. Let $m>1$.
\begin{itemize}
\item[(A)] When restricting $\HH^0(X,\Omega_X^{\otimes m})$ to $3$-hypo-elementary subgroups $\Gamma$ of $G$, we need to determine canonical divisors on $X/\Gamma$.
\item[(B)] The indecomposable non-projective $kG$-modules that occur as direct summands of $\HH^0(X,\Omega_X^{\otimes m})$  depend on the congruence class of $m$ modulo $6$.
\item[(C)] The values of the Brauer character of $\HH^0(X,\Omega_X^{\otimes m})$ at elements of $G$ of order $\ell$ depend on the congruence class of $m$ modulo $\ell$. 
\item[(D)] There are indecomposable non-projective $kG$-modules that are not uniserial and that occur as direct summands of $\HH^0(X,\Omega_X^{\otimes m})$.
\end{itemize}

Because of (D), our notation for the isomorphism classes of indecomposable $kG$-modules that occur as direct summands of $\HH^0(X,\Omega_X^{\otimes m})$ is also more complicated, see Notation \ref{not:PSL} below. We refer to the work in \cite{Burkhardt1976} from which a description of the isomorphism classes of all indecomposable $kG$-modules follows. 

\begin{nota}
\label{not:PSL}
Let $\ell\ge 7$ be a prime number, let $G = \mathrm{PSL}(2,\mathbb{F}_\ell)$, and let $k$ be an algebraically closed field of characteristic $p=3$. 
\begin{enumerate}
\item[(a)]
If $T$ is a simple $kG$-module, then $U_{T,b}^{(G)}$ denotes a uniserial $kG$-module with $b$ composition factors whose socle is isomorphic to $T$. The isomorphism class of $U_{T,b}^{(G)}$ is uniquely determined by $T$ and $b$ (see, for example, \cite{Burkhardt1976}). 
\item[(b)] 
If $\ell\equiv -1 \mod 3$ then there exist indecomposable $kG$-modules that are not uniserial. These modules all belong to the principal block of $kG$. There are precisely two isomorphism classes of simple $kG$-modules belonging to the principal block, represented by the trivial simple $kG$-module $T_0$ and a simple $kG$-module $\widetilde{T}_0$ of $k$-dimension $\ell-1$. It follows, for example, from \cite{Burkhardt1976} that the isomorphism classes of the non-uniserial indecomposable $kG$-modules are all uniquely determined by their socles and their tops $($i.e. radical quotients$)$, together with the number of their composition factors that are isomorphic to $\widetilde{T}_0$. There are three different types of such modules. We use the notation $U^{(G)}_{T_0,T_0,b}$ $($resp. $U^{(G)}_{T_0,\widetilde{T}_0,b}$, resp.  $U^{(G)}_{\widetilde{T}_0,T_0,b}$$)$ to denote an indecomposable $kG$-module that has $b$ composition factors isomorphic to $\widetilde{T}_0$, whose socle is isomorphic to $T_0\oplus\widetilde{T}_0$ $($resp. $T_0\oplus\widetilde{T}_0$, resp. $\widetilde{T}_0$$)$ and whose top is isomorphic to $T_0\oplus\widetilde{T}_0$ $($resp. $\widetilde{T}_0$, resp. $T_0\oplus\widetilde{T}_0$$)$.
\end{enumerate}
\end{nota}

\begin{theorem}
\label{thm:modularresult}
Suppose $m>1$ and $\ell\ge 7$. Let $v \in \mathcal{V}(F,3)$, let $k$ be an algebraically closed field containing $k(v)$, and let  $X:=X_3(\ell)$ be as in Equation $(\ref{eq:reductionmodular})$.
Let $\epsilon=\pm 1$ be such that $\ell\equiv \epsilon \mod 3$. Write $\ell-\epsilon = 2\cdot 3^n\cdot n'$ where $3$ does not divide $n'$. For $i\in\{0,1\}$, define $\delta_i$ to be $1$ if $m\equiv i\mod 3$ and to be $0$ otherwise. Define $\delta_m\in\{0,1\}$ by $m\equiv \delta_m\mod 2$.

\begin{itemize}
\item[(i)] There exists a projective $kG$-module $Q_\ell$, depending on $\ell$, such that the following is true:
\begin{itemize}
\item[(1)] Suppose $\ell\equiv 1\mod 4$ and $\ell\equiv -1 \mod 3$. Let $T_0$ denote the trivial simple $kG$-module. For $0\le t\le (n'-1)/2$, let $\widetilde{T}_t$ be representatives of simple $kG$-modules of $k$-dimension $\ell-1$ such that $\widetilde{T}_0$ belongs to the principal block of $kG$. As a $kG$-module,
\begin{eqnarray*}
\HH^0(X,\Omega_X^{\otimes m})&\cong &\delta_0(1-\delta_m)\,U_{T_0,T_0,3^{n-1}+1}^{(G)}\;\oplus \;\delta_0\delta_m\,U_{\widetilde{T}_0,3^{n-1}}^{(G)}  \\
\nonumber
&&\oplus\;\delta_1\delta_m\,U_{T_0,\widetilde{T}_0,(3^{n-1}+1)/2}^{(G)}\;\oplus\; \delta_1(1-\delta_m)\,U_{\widetilde{T}_0,T_0,(3^{n-1}+1)/2}^{(G)}  \\
\nonumber
&&\oplus\;\bigoplus_{t=1}^{(n'-1)/2} \delta_0\,U_{\widetilde{T}_t,2\cdot 3^{n-1}}^{(G)}\;\oplus\; \bigoplus_{t=1}^{(n'-1)/2} \delta_1\,U_{\widetilde{T}_t,3^{n-1}}^{(G)} \;\oplus\; Q_\ell.
\end{eqnarray*}
\item[(2)] Suppose $\ell\equiv -1\mod 4$ and $\ell\equiv 1 \mod 3$. Let $T_0$ denote the trivial simple $kG$-module, and let $T_1$ be a simple $kG$-module of $k$-dimension $\ell$. For $1\le t\le (n'-1)/2$, let $\widetilde{T}_t$ be representatives of simple $kG$-modules of $k$-dimension $\ell+1$. As a $kG$-module,
\begin{eqnarray*}
\HH^0(X,\Omega_X^{\otimes m})&\cong & \delta_0(1-\delta_m)\,U_{T_0,3^{n-1}}^{(G)} \;\oplus\; \delta_0\delta_m\,U_{T_1,3^{n-1}}^{(G)} \\
\nonumber
&&\oplus\; \delta_1\delta_m\,U_{T_0,2\cdot 3^{n-1}}^{(G)}\;\oplus\; \delta_1(1-\delta_m)\,U_{T_1,2\cdot 3^{n-1}}^{(G)} \\
\nonumber
&&\oplus\,\bigoplus_{t=1}^{(n'-1)/2} \delta_0\,U_{\widetilde{T}_t,3^{n-1}}^{(G)}  \;\oplus\; \bigoplus_{t=1}^{(n'-1)/2} \delta_1\,U_{\widetilde{T}_t,2\cdot 3^{n-1}}^{(G)} \;\oplus\; Q_\ell.
\end{eqnarray*}
\item[(3)] Suppose $\ell\equiv 1\mod 4$ and $\ell\equiv 1 \mod 3$. Let $T_0$ denote the trivial simple $kG$-module, and let $T_1$ be a simple $kG$-module of $k$-dimension $\ell$. For $1\le t\le (n'/2-1)$, let $\widetilde{T}_t$ be representatives of simple $kG$-modules of $k$-dimension $\ell+1$. There exist simple $kG$-modules $T_{0,1}$ and $T_{1,0}$ of $k$-dimension $(\ell+1)/2$ such that, as a $kG$-module,
\begin{eqnarray*}
\HH^0(X,\Omega_X^{\otimes m})&\cong & \delta_0(1-\delta_m)\left(U_{T_0,3^{n-1}}^{(G)}\oplus U_{T_{0,1},3^{n-1}}^{(G)}\right)\;\oplus\;\delta_0 \delta_m\left(U_{T_1,3^{n-1}}^{(G)} \oplus U_{T_{1,0},3^{n-1}}^{(G)}\right)\\
\nonumber
&&\oplus\; \delta_1\delta_m\left(U_{T_0,2\cdot 3^{n-1}}^{(G)}\oplus U_{T_{0,1},2\cdot 3^{n-1}}^{(G)} \right)\;\oplus\;\delta_1 (1-\delta_m)\left(U_{T_1,2\cdot 3^{n-1}}^{(G)} \oplus U_{T_{1,0},2\cdot 3^{n-1}}^{(G)}\right)\\
\nonumber
&&\oplus\;  \bigoplus_{t=1}^{(n'/2-1)} \delta_0\;U_{\widetilde{T}_t,3^{n-1}}^{(G)} \;\oplus\;  \bigoplus_{t=1}^{(n'/2-1)} \delta_1\;U_{\widetilde{T}_t,2\cdot3^{n-1}}^{(G)} \;\oplus\; Q_\ell.
\end{eqnarray*}
\item[(4)] Suppose $\ell\equiv -1\mod 4$ and $\ell\equiv -1 \mod 3$. Let $T_0$ denote the trivial simple $kG$-module. For $0\le t\le (n'/2-1)$, let $\widetilde{T}_t$ be representatives of simple $kG$-modules of $k$-dimension $\ell-1$ such that $\widetilde{T}_0$ belongs to the principal block of $kG$. There exist simple $kG$-modules $T_{0,1}$ and $T_{1,0}$ of $k$-dimension $(\ell-1)/2$ such that, as a $kG$-module,
\begin{eqnarray*}
\HH^0(X,\Omega_X^{\otimes m})&\cong & \delta_0(1-\delta_m)\left(U_{T_0,T_0,3^{n-1}+1}^{(G)}\oplus U_{T_{0,1},2\cdot 3^{n-1}}^{(G)}\right) \;\oplus\; \delta_0\delta_m\left(U_{\widetilde{T}_0,3^{n-1}}^{(G)} \oplus U_{T_{1,0},2\cdot 3^{n-1}}^{(G)}\right)\\
\nonumber
&&\oplus\; \delta_1\delta_m\left(U_{T_0,\widetilde{T}_0,\frac{3^{n-1}+1}{2}}^{(G)}\oplus U_{T_{0,1},3^{n-1}}^{(G)} \right)\;\oplus\; \delta_1(1-\delta_m)\left(U_{\widetilde{T}_0,T_0,\frac{3^{n-1}+1}{2}}^{(G)} \oplus U_{T_{1,0},3^{n-1}}^{(G)}\right) \\
\nonumber
&&\oplus\; \bigoplus_{t=1}^{(n'/2-1)} \delta_0\; U_{\widetilde{T}_t,2\cdot 3^{n-1}}^{(G)} \;\oplus\; \bigoplus_{t=1}^{(n'/2-1)} \delta_1\;U_{\widetilde{T}_t,3^{n-1}}^{(G)} \;\oplus\; Q_\ell.
\end{eqnarray*}
\end{itemize}
The multiplicities of the projective indecomposable $kG$-modules in $Q_\ell$ are known explicitly.  In parts $(3)$ and $(4)$, there are two conjugacy classes of subgroups of $G$, represented by $H_1$ and $H_2$, that are isomorphic to the symmetric group $\Sigma_3$ such that the conjugates of $H_1$ $($resp. $H_2$$)$ occur $($resp. do not occur$)$ as inertia groups of closed points of $X$. This characterizes the simple $kG$-module $T_{0,1}$ in parts $(3)$ and $(4)$ as follows. The restriction of $T_{0,1}$ to $H_1$ $($resp. $H_2$$)$ is a direct sum of a projective module and a non-projective indecomposable module whose socle is the trivial simple module $($resp. the simple module corresponding to the sign character$)$.

\item[(ii)] Let $k_1$ be a perfect field of characteristic $3$ containing $k(v)$, and let $k$ be an algebraic closure of $k_1$. Define $X_1:=k_1\otimes_{k(v)}(k(v)\otimes_A\mathcal{X}_A(\ell))$. Then
$$k\otimes_{k_1} \HH^0(X_1,\Omega_{X_1}^{\otimes m})\cong \HH^0(X,\Omega_X^{\otimes m})$$
as $kG$-modules, and the decomposition of $\HH^0(X_1,\Omega_{X_1}^{\otimes m})$ into indecomposable $k_1G$-modules is uniquely determined by the decomposition of $\HH^0(X,\Omega_X^{\otimes m})$ into indecomposable $kG$-modules. Moreover, the $k_1G$-module $\HH^0(X_1,\Omega_{X_1}^{\otimes m})$ is a direct sum over blocks $B_1$ of $k_1G$ of modules of the form $P_{B_1} \oplus U_{B_1}$ in which $P_{B_1}$ is a projective $B_1$-module and $U_{B_1}$ is either the zero module or a single indecomposable non-projective $B_1$-module.  In addition, one can determine $P_{B_1}$ and $U_{B_1}$ from the lower ramification groups and the fundamental characters of the closed points of $X$ that are ramified in the cover $X\to X/G$. \end{itemize}
\end{theorem}

Theorem \ref{thm:modularresult} extends \cite[Thm. 1.4]{BleherChinburgKontogeorgis2020} from $m=1$ to arbitrary $m>1$. Detailed descriptions of the projective modules $Q_\ell$ in part (i) of Theorem \ref{thm:modularresult} are given in Propositions \ref{prop:full_mixed} and \ref{prop:full_equal}. 
Because of its computational
nature, we defer the proof of Theorem \ref{thm:modularresult} to the next section.


\section{Holomorphic poly-differentials of the modular curves $X_3(\ell)$}
\label{s:modular3}

We make the following assumptions throughout this section.

\begin{assume}
\label{ass:modular3}
Let $m > 1$ be an integer, let $\ell\ge 7$ be a prime number, and let $p = 3$. 
Let $v \in \mathcal{V}(F,3)$, let $k$ be an algebraically closed field containing $k(v)$, and let  $X=X_3(\ell)$ be as in Equation $(\ref{eq:reductionmodular})$. Let $G = \mathrm{PSL}(2,\mathbb{F}_\ell)$. 
\end{assume}

The goal of this section is to determine the precise $kG$-module structure of $\HH^0(X,\Omega_X^{\otimes m})$. In particular, we will prove Theorem \ref{thm:modularresult}. We will adapt the strategy used in \cite[\S6]{BleherChinburgKontogeorgis2020} to prove \cite[Thm. 1.4]{BleherChinburgKontogeorgis2020} when $m=1$ to our situation when $m>1$.  As we pointed out in the previous section, the proof of Theorem \ref{thm:modularresult} is more involved than the proof of \cite[Thm. 1.4]{BleherChinburgKontogeorgis2020}.

By \cite[p. 193]{Moreno1993}, the ramification points of the cover $X\to X/G$ that are wildly ramified have inertia groups that are isomorphic to the symmetric group $\Sigma_3$ on three letters. We use that there is precise knowledge of the subgroup structure of $G = \mathrm{PSL}(2,\mathbb{F}_\ell)$ (see, for example, \cite[\S II.8]{Huppert1967}).

\begin{nota}
\label{not:modular3}
Suppose Assumption \ref{ass:modular3} holds. Write
$$m=3\cdot m'+i_m\quad\mbox{where $i_m\in\{0,1,2\}$.}$$
For $i\in\{0,1\}$, define $\delta_i$ to be $1$ when $i=i_m$ and $0$ otherwise. Define $\delta_m$ to be $1$ if $m$ is odd and $0$ otherwise. Let $\epsilon,\epsilon'\in\{\pm 1\}$ be such that 
$$\ell\equiv \epsilon \mod 3\quad\mbox{ and }\quad\ell\equiv \epsilon' \mod 4.$$
Write $\ell - \epsilon =3^n\cdot 2\cdot n'$ where $3$ does not divide $n'$.
\end{nota}

\begin{remark}
\label{rem:modular3}
As in \cite[\S6]{BleherChinburgKontogeorgis2020}, we fix the following $3$-hypo-elementary subgroups of $G$:
\begin{enumerate}
\item[(a)] a cyclic subgroup $V=\langle v\rangle$ of order $(\ell-\epsilon)/2=3^n\cdot n'$;
\item[(b)] two dihedral groups $\Delta_1=\langle v',s\rangle$ and $\Delta_2=\langle v',vs\rangle$ of order $2\cdot 3^n$, where $v'=v^{n'}\in V$ is an element of order $3^n$ and $s\in N_G(V)-V$ is an element of order 2;
\item[(c)] a cyclic subgroup $W=\langle w\rangle$ of order $(\ell+\epsilon)/2$;
\item[(d)] a cyclic subgroup $R$ of order $\ell$.
\end{enumerate}
We let $\tau=(v')^{3^{n-1}}$ and $I=\langle \tau\rangle$, so that $I$ is the unique subgroup of order 3 in each of $V,\Delta_1,\Delta_2$. If $\epsilon=-\epsilon'$, i.e. if $\ell\equiv -\epsilon\mod 4$, we also let $\Delta=\Delta_1$. Moreover, we let $P=\langle v'\rangle$, $P_1=I$, and $N_1=N_G(P_1)$. It follows, for example, from \cite[\S II.8]{Huppert1967} that $N_1=\langle v, s\rangle$ is a dihedral group of order $\ell-\epsilon$. 
\end{remark}

Note that not all $3$-hypo-elementary subgroups of $G$ are conjugate to one of the groups in (a)-(d) of Remark \ref{rem:modular3}. Rather than using the Conlon induction theorem, which we used to prove Theorem \ref{thm:newmain} and Corollary \ref{thm:oldmain} in general and which requires all $3$-hypo-elementary subgroups of $G$, we will adapt to our situation the approach used in \cite[\S6]{BleherChinburgKontogeorgis2020} which only needs the $3$-hypo-elementary subgroups in (a)-(d) of Remark \ref{rem:modular3}. The main observations that make this approach work are as follows (see \cite[\S6]{BleherChinburgKontogeorgis2020} for details):
\begin{itemize}
\item Every $kN_1$-module is uniquely determined by its restrictions to the subgroups in $\{V,\Delta_1,\Delta_2\}$.
\item There is a \emph{stable} equivalence between the module categories of $kG$ and $kN_1$, i.e. an equivalence of these categories modulo projective modules (see \cite[\S V.17]{Alperin1986}). This allows us to use \cite[\S X.1]{AuslanderReitenSmalo1997} to detect the non-projective indecomposable modules for $kG$ and $kN_1$, respectively, that correspond to each other under the Green correspondence  \cite[Thm. V.17.3]{Alperin1986}.
\item The \emph{stable} $kG$-module structure of $\HH^0(X,\Omega_X^{\otimes m})$ and 
its Brauer character determine the full $kG$-module structure of $\HH^0(X,\Omega_X^{\otimes m})$.
\end{itemize}

In \S\ref{ss:canonical}, we determine necessary information about the canonical divisors on $X/\Gamma$ and $X/I\cap\Gamma$  when $\Gamma\in\{\Delta_1,\Delta_2,V,W,R\}$. In \S\ref{ss:N1structure}, we use Corollary \ref{thm:main3cor} to determine the precise $k\Gamma$-module structure of $\mathrm{Res}^G_{\Gamma}\HH^0(X,\Omega_X^{\otimes m})$ for $\Gamma\in\{V,\Delta_1,\Delta_2\}$. We then use this to determine the \emph{stable} $kN_1$-module structure of $\mathrm{Res}^G_{N_1}\HH^0(X,\Omega_X^{\otimes m})$.
In \S\ref{ss:Greencorrespondence}, we use the Green correspondence to determine the \emph{stable} $kG$-module structure of $\HH^0(X,\Omega_X^{\otimes m})$. In \S\ref{ss:brauer}, we determine the Brauer character of the $kG$-module $\HH^0(X,\Omega_X^{\otimes m})$. In \S\ref{ss:fullmodular}, we use \S\ref{ss:Greencorrespondence} and \S\ref{ss:brauer} to determine the full $kG$-module structure of $\HH^0(X,\Omega_X^{\otimes m})$ and to complete the proof of Theorem \ref{thm:modularresult}.


\subsection{Canonical divisors}
\label{ss:canonical}

We consider the canonical divisors on $X/\Gamma$ and $X/I\cap\Gamma$ when
$$\Gamma\in \{\Delta_1,\Delta_2,V,W,R\},$$
and we determine the information needed to be able to apply Corollary \ref{thm:main3cor}.
Let $Y=X/I\cap\Gamma$ and $Z=X/\Gamma$, and let $\pi:X\to Y$ and $\lambda:Y\to Z$ be the corresponding Galois covers with Galois groups $I\cap \Gamma$ and $\overline{\Gamma}=\Gamma/I\cap\Gamma$, respectively. Let $Z_{\mathrm{br}}$ be the set of closed points in $Z$ that are branch points of $\lambda$. The cover $X\to X/G\cong\mathbb{P}^1_k$ factors into the Galois cover $\lambda\circ\pi:X\to Z$ followed by a separable morphism $f:Z=X/\Gamma \to X/G$ which is not Galois. A canonical divisor on $X/G\cong\mathbb{P}^1_k$ is given by $K_{X/G}=-2\infty$, and a canonical divisor on $Z$ is given by
$$K_Z=f^*(-2\infty) +\mathrm{Ram}_f$$
where $\mathrm{Ram}_f$ is the ramification divisor of $f$ (see \cite[Prop. IV.2.3]{Hartshorne1977}). Therefore,
\begin{equation}
\label{eq:KYmodular}
K_Y=\lambda^*K_Z+\mathrm{Ram}_\lambda =(f\circ\lambda)^*(-2\infty) + \lambda^*\mathrm{Ram}_f+\mathrm{Ram}_\lambda.
\end{equation}

For each $\Gamma$ as above, we now determine $\mathrm{ord}_{y(z)}(K_Y)$ for all points $y(z)\in Y$ above $z\in Z_{\mathrm{br}}$. We give precise details in one of the cases and then list the remaining cases in Table \ref{tab:canonicalmodular}.

Suppose first that $\epsilon=-\epsilon'$ and that $\Gamma=\Delta=\Delta_1$. Then $I\cap\Gamma=I$, $Y=X/I$ and $\overline{\Gamma}=\Gamma/I$. Let $z\in Z_{\mathrm{br}}$, and let $y(z)\in Y$ and $x(z)\in X$ be points above it. By \cite[p. 193]{Moreno1993}, $x(z)$, $y(z)$ and $z$ all lie above $0\in X/G\cong\mathbb{P}^1_k$ and $G_{x(z),0}\cong\Sigma_3$. Hence, in Equation (\ref{eq:KYmodular}), the coefficient of $y(z)$ in $(f\circ\lambda)^*(-2\infty)$ is zero. If $I_{x(z)}$ (resp. $\Gamma_{x(z)}$) is the inertia group of $x(z)$ inside $I$ (resp. $\Gamma$), then $\overline{\Gamma}_{y(z)}\cong \Gamma_{x(z)}/I_{x(z)}$. Because $z\in Z_{\mathrm{br}}$, $\overline{\Gamma}_{y(z)}$ must be a non-trivial group of order prime to $3$, which implies $\#\overline{\Gamma}_{y(z)}=2$. Hence the coefficient of $y(z)$ in $\mathrm{Ram}_\lambda$ is 1. On the other hand, the coefficient of $y(z)$ in $\lambda^*\mathrm{Ram}_f$ equals $e_{y(z)/z}\cdot d_{z/0}$, where $e_{y(z)/z}= 2$ is the ramification index of $y(z)$ over $z$ and $d_{z/0}$ is the different exponent of $z$ over $0 \in X/G \cong \mathbb{P}^1_k$. By \cite[p. 193]{Moreno1993}, $d_{x(z)/0} = (6-1)+(3-1)=7$. Using the transitivity of the different, we have
$$d_{x(z)/0} = e_{x(z)/z}\cdot d_{z/0} + d_{x(z)/z}.$$
It follows from 
\cite[\S6.1.1 and \S6.2.1]{BleherChinburgKontogeorgis2020}
that there is precisely one point $z_1\in Z_{\mathrm{br}}$ with $e_{x(z_1)/z_1}=6$ and $d_{x(z_1)/z_1}=7$, and there are precisely $\frac{\ell-\epsilon'}{2}-1$ points $z_2,\ldots,z_{(\ell-\epsilon')/2}\in Z_{\mathrm{br}}$ with $e_{x(z_i)/z_i}=2$ and $d_{x(z_i)/z_i}=1$ for $2\le i \le \frac{\ell-\epsilon'}{2}$. Therefore, we obtain $d_{z_1/0}=0$ and $d_{z_i/0}=3$ for $2\le i \le \frac{\ell-\epsilon'}{2}$, which means
$$\mathrm{ord}_{y(z_i)}(K_Y)=\left\{\begin{array}{ll}1,&i=1,\\7,&2\le i\le \frac{\ell-\epsilon'}{2},\end{array}\right. \qquad \mbox{if $\epsilon=-\epsilon'$ and $\Gamma=\Delta$}.$$

We can use similar arguments for the remaining $\Gamma$. In Table \ref{tab:canonicalmodular}, we list for all $z\in Z_{\mathrm{br}}$, the values for $\mathrm{ord}_{y(z)}(K_Y)$ together with the orders of the inertia groups $\overline{\Gamma}_{y(z)}$ and $\Gamma_{x(z)}$.
\renewcommand{\arraystretch}{1.25}
\begin{table}[ht]
\caption{Canonical divisor $K_Y$ on $Y=X/I\cap\Gamma$ when $\Gamma\in\{\Delta_1,\Delta_2,V,W,R\}$.}
\label{tab:canonicalmodular}
$\begin{array}{cl||c|c|c|c|c}
\multicolumn{2}{c||}{\Gamma}&I\cap \Gamma&\mbox{all }z\in Z_{\mathrm{br}}&\mathrm{ord}_{y(z)}(K_Y)&\#\overline{\Gamma}_{y(z)}&\#\Gamma_{x(z)}\\ \hline\hline
\Delta=\Delta_1 &(\epsilon=-\epsilon')&I&\begin{array}{c}z_1\\z_i\quad (2\le i\le \frac{\ell-\epsilon'}{2})\end{array}&\begin{array}{c}1\\7\end{array}&\begin{array}{c}2\\2\end{array}&\begin{array}{c}6\\2\end{array}\\ \hline
\Delta_1 &(\epsilon=\epsilon')&I&\begin{array}{c}z_i\quad (i=1,2)\\z_i\quad (3\le i\le \frac{\ell-\epsilon'}{2})\end{array}&\begin{array}{c}1\\7\end{array}&\begin{array}{c}2\\2\end{array}&\begin{array}{c}6\\2\end{array}\\ \hline
\Delta_2 &(\epsilon=\epsilon')&I&z_i\quad (1\le i\le \frac{\ell-\epsilon'}{2})&7&2&2\\ \hline
V &(\epsilon=-\epsilon') & I & Z_{\mathrm{br}}=\emptyset&&&\\ \hline
V &(\epsilon=\epsilon') & I &z_i\quad (i=1,2)&7&2&2\\ \hline
W &(\epsilon=-\epsilon') & 1 &z_i\quad (i=1,2)&7&2&2\\ \hline
W &(\epsilon=\epsilon') & 1& Z_{\mathrm{br}}=\emptyset&&&\\ \hline
R&(\epsilon=\pm\epsilon')&1&z_i\quad (1\le i\le \frac{\ell-1}{2})&-\ell-1&\ell&\ell
\end{array}$
\end{table}
\renewcommand{\arraystretch}{1}


\subsection{The stable $kN_1$-module structure of the holomorphic poly-differentials}
\label{ss:N1structure}

We use a similar strategy to \cite[\S6.2]{BleherChinburgKontogeorgis2020}. More precisely, we first calculate the stable $k\Gamma$-module structure of $\mathrm{Res}^G_{\Gamma}\HH^0(X,\Omega_X^{\otimes m})$ for the subgroups $\Gamma\in\{V,\Delta_1,\Delta_2\}$ of $N_1$. In other words, we find the non-projective indecomposable direct $k\Gamma$-module summands of $\mathrm{Res}^G_{\Gamma}\HH^0(X,\Omega_X^{\otimes m})$, together with their multiplicities. We then use that every $kN_1$-module is uniquely determined by its restrictions to these subgroups to find the stable $kN_1$-module structure of $\mathrm{Res}^G_{N_1}\HH^0(X,\Omega_X^{\otimes m})$. 

The main differences when $m>1$ are that Corollary \ref{thm:main3cor}, and also Theorem \ref{thm:main3} applied to $E=mK_X$, depend on $m$ and that we need the information about the canonical divisors $K_Y$ given in Table \ref{tab:canonicalmodular}. It follows that the indecomposable non-projective $kN_1$-modules that occur as direct summands of $\mathrm{Res}^G_{N_1}\HH^0(X,\Omega_X^{\otimes m})$ depend on the congruence class of $m$ modulo $6$.

Let $\Gamma\in\{V,\Delta_1,\Delta_2\}$. As in \cite[\S6.2]{BleherChinburgKontogeorgis2020}, the subgroup of $\Gamma$ that is the greatest among the Sylow 3-subgroups of the inertia groups of all closed points on $X$ is equal to the unique subgroup $I=\langle \tau\rangle$ of $\Gamma$ of order 3. Moreover, there are precisely $3^{n-1}\cdot n'$ closed points $x$ on $X$ such that $3$ divides $\#\Gamma_x$, which is equivalent to $\Gamma_{x}\ge I$. Using \cite[Prop. IV.2 and Cor. 4 of Prop. IV.7]{SerreCorpsLocaux1968}, it follows from \cite[p. 193]{Moreno1993} that $I_{x,0}=I_{x,1}=I$ and $I_{x,i}=1$ for $i\ge 2$ for these points $x$. Let $y_{t,1},\ldots, y_{t,3^{n-1}}$, for $1\le t\le n'$, be points on $Y=X/I$ that lie below these points on $X$. Then, for $0\le j\le 2$, the divisor $D_j$ from Corollary \ref{thm:main3cor} is given as
\begin{equation}
\label{eq:Djmodular}
D_j=\sum_{t=1}^{n'}\sum_{i=1}^{3^{n-1}} \left\lfloor \frac{4m-j}{3}\right\rfloor y_{t,i}.
\end{equation}
Writing $m=3 m'+i_m$ as in Notation \ref{not:modular3}, where $i_m\in\{0,1,2\}$, we have
\begin{equation}
\label{eq:Djcoefficient}
\left\lfloor \frac{4m-j}{3}\right\rfloor = 4m' + \left\lfloor \frac{4i_m-j}{3}\right\rfloor = \left\{\begin{array}{ll}
4m'-1, & i_m=0, j\in\{2,1\},\\
4m', & (i_m=0, j=0)\mbox{ or } (i_m=1, j=2),\\
4m'+1, & i_m=1, j\in\{1,0\},\\
4m'+2, &i_m=2, j\in\{2,1,0\}.
\end{array}\right.
\end{equation}
This leads to the following equalities, for $j\in\{0,1\}$:
\begin{equation}
\label{eq:Djmodulardegree}
\frac{\mathrm{deg}(D_j)-\mathrm{deg}(D_{j+1})}{3^{n-1}\cdot n'}=
\left\lfloor \frac{4m-j}{3}\right\rfloor-\left\lfloor \frac{4m-(j+1)}{3}\right\rfloor =
\left\{\begin{array}{cl}
1,&(i_m,j)\in\{(0,0),(1,1)\},\\
0,&\mbox{otherwise}.
\end{array}\right.
\end{equation}


\subsubsection{The stable $kV$-module structure and the values of the Brauer character at $3$-regular elements of $V$}
\label{sss:V}

There are $n'$ isomorphism classes of simple $kV$-modules. As in \cite[\S6.2]{BleherChinburgKontogeorgis2020}, we denote representatives of these by $S_a^{(V)}$ for $0\le a\le n'-1$. Similarly to Notation \ref{not:main3ind}(b), we write $U_{a,b}^{(V)}$ for an indecomposable $kV$-module whose socle is isomorphic to $S_a^{(V)}$ and whose $k$-dimension is equal to $b$ with $1\le b\le 3^n$. The Galois cover $\lambda: Y=X/I \to Z=X/V$ is unramified if $\epsilon=-\epsilon'$ (resp. tamely ramified if $\epsilon=\epsilon'$) with Galois group $\overline{V}=V/I$. We now apply Theorem \ref{thm:main3} to $E=mK_X$ and the group $V$, where we will write $n_j(V)$ for the integers $n_j$ from Equation (\ref{eq:nj}) and $n_{a,b}^{(V)}$ for the integers $n_{a,b}$ from Equation (\ref{eq:nab}). Note that $n_I=1$ since $\#(I\cap V)=3$. As in Table \ref{tab:canonicalmodular}, if $\epsilon=-\epsilon'$ then $Z_{\mathrm{br}}=\emptyset$, and if $\epsilon=\epsilon'$ then $Z_{\mathrm{br}}=\{z_1,z_2\}$. In the latter case, we use that $E_j=mK_Y+D_j$, together with the information from Table \ref{tab:canonicalmodular}, to obtain that $\ell_{y(z_i),j}=\delta_m$ for $i=1,2$ and $0\le j\le 2$. By Remark \ref{rem:computation}, this immediately implies that for both cases $\epsilon=-\epsilon'$ and $\epsilon=\epsilon'$, we have $n(a,j)-n(a,j+1)=n_j(V)-n_{j+1}(V)$ for $j\in\{0,1\}$. Using the alternative formula for $n_j(V)$ from Remark \ref{rem:combine}(a), together with Equation (\ref{eq:Djmodulardegree}), we obtain, for $j\in\{0,1\}$,
\begin{equation}
\label{eq:needV}
n_j(V)-n_{j+1}(V) 
= \frac{\mathrm{deg}(D_j)-\mathrm{deg}(D_{j+1})}{\#\overline{V}}=
\left\{\begin{array}{cl}
1,&(i_m,j)\in\{(0,0),(1,1)\},\\
0,&\mbox{otherwise},
\end{array}\right.
\end{equation}
for all $\epsilon,\epsilon'\in\{\pm 1\}$. By Equation (\ref{eq:nab}),  $n_{a,(j+1)3^{n-1}}^{(V)}=n_j(V)-n_{j+1}(V)$, for $j\in\{0,1\}$ and $0\le a\le n'-1$, and these are all possible non-zero $n_{a,b}^{(V)}$ for $1\le b \le 3^n-1$.
Therefore, 
the non-projective indecomposable direct $kV$-module summands of $\mathrm{Res}^G_V\HH^0(X,\Omega_X^{\otimes m})$, with their multiplicities, are given by the direct sum
\begin{equation}
\label{eq:Vstable}
\delta_0\bigoplus_{a=0}^{n'-1}U_{a,3^{n-1}}^{(V)}
\;\oplus\; \delta_1\bigoplus_{a=0}^{n'-1}U_{a,2\cdot 3^{n-1}}^{(V)}
\end{equation}
where $\delta_0$ and $\delta_1$ are as in Notation \ref{not:modular3}.

We now use Equation (\ref{eq:Vstable}) to determine the values of the Brauer character $\beta$ of $\HH^0(X,\Omega_X^{\otimes m})$ at elements of $V$ that are $3$-regular, i.e. whose orders are not divisible by $3$. We write $V=\langle v\rangle$ as in Remark \ref{rem:modular3}, and let $v''=v^{3^n}$, so that $v''$ has order $n'$. Let $\xi_{n'}$ be a fixed primitive $(n')^{\mathrm{th}}$ root of unity so that,  for $0\le a\le n'-1$, $v''$ acts as multiplication by $\xi_{n'}^a$ on $S_a^{(V)}$. Applying Notation \ref{not:main3ind}(c) to the group $V$, it follows that each $U_{a,b}^{(V)}$ has $b$ composition factors, which are all isomorphic to $S_a^{(V)}$. Therefore, letting $\widetilde{\beta}^{(V)}$ be the Brauer character of the maximal projective direct $kV$-module summand of $\mathrm{Res}^G_V\HH^0(X,\Omega_X^{\otimes m})$, we obtain that the value at $(v'')^i$, $1\le i < \frac{n'}{2}$, of the Brauer character of the direct sum in Equation (\ref{eq:Vstable}) is equal to
\begin{equation}
\label{eq:VBrauer1}
(\beta-\widetilde{\beta}^{(V)})((v'')^i) = \delta_0\sum_{a=0}^{n'-1} 3^{n-1}\xi_{n'}^{ia} + \delta_1\sum_{a=0}^{n'-1} 2\cdot 3^{n-1}\xi_{n'}^{ia} = 0.
\end{equation}
Here, the last equality follows since $\xi_{n'}^i$ is an $(n')^{\mathrm{th}}$ root of unity, and hence $\sum_{a=0}^{n'-1}\xi_{n'}^{ia}=0$ for all $1\le i < \frac{n'}{2}$. To find $\widetilde{\beta}^{(V)}((v'')^i)$, we need to determine $n_{a,3^n}^{(V)}$ for $0\le a\le n'-1$. Since the Brauer character of $kV$ has zero value at all non-identity elements of $V$, it follows from Equations (\ref{eq:naj}) and (\ref{eq:nab}) that it suffices to determine $n_{a,3^n}^{(V)}-n_2(V)$. Using that, for $\epsilon=\epsilon'$ and $i=1,2$, $\mu_{a,-1}(y(z_i))=\mu_{a,1}(y(z_i))=1$ if and only if $a$ is odd, we get,  for $0\le a\le n'-1$,
$$n_{a,3^n}^{(V)}-n_2(V) = \left\{\begin{array}{cl}
2\left(\delta_m-\frac{1}{2}\right),&\mbox{$\epsilon=\epsilon'$ and $a$ odd},\\
0,&\mbox{$\epsilon=-\epsilon'$ or $a$ even}.
\end{array}\right.$$
If $\epsilon=\epsilon'$ then $n'$ is even, and hence $\lfloor\frac{n'}{2}-1\rfloor=\frac{n'}{2}-1$ and $\xi_{n'}^2$ is a primitive $(\frac{n'}{2})^{\mathrm{th}}$ root of unity.
Therefore, we obtain, for all $\epsilon,\epsilon'\in\{\pm 1\}$ and all $1\le i < \frac{n'}{2}$,
\begin{equation}
\label{eq:VBrauer2}
\widetilde{\beta}^{(V)}((v'')^i) =\delta_{\epsilon,\epsilon'} \,(2\delta_m-1)\sum_{t=0}^{\lfloor\frac{n'}{2}-1\rfloor} 3^n\xi_{n'}^{i(2t+1)} = 
\delta_{\epsilon,\epsilon'} \,(2\delta_m-1)\,3^n\,\xi_{n'}^i\sum_{t=0}^{\lfloor\frac{n'}{2}-1\rfloor} (\xi_{n'}^2)^{it} = 0.
\end{equation}


\subsubsection{The stable $k\Delta_1$- and $k\Delta_2$-module structures and the values of the Brauer character at elements of order $2$}
\label{sss:Delta}

Let $\Gamma\in \{\Delta_1,\Delta_2\}$, where we only consider $\Gamma=\Delta=\Delta_1$ if $\epsilon=-\epsilon'$. There are precisely two isomorphism classes of simple $k\Gamma$-modules. As in \cite[\S6.2]{BleherChinburgKontogeorgis2020}, we denote representatives of these by $S_a^{(\Gamma)}$ for $a\in\{0,1\}$. Similarly to Notation \ref{not:main3ind}(b), we write $U_{a,b}^{(\Gamma)}$ for an indecomposable $k\Gamma$-module whose socle is isomorphic to $S_a^{(\Gamma)}$ and whose $k$-dimension is equal to $b$ with $1\le b\le 3^n$. The Galois cover $\lambda: Y=X/I \to Z=X/\Gamma$ is tamely ramified with Galois group $\overline{\Gamma}=\Gamma/I$. We now apply Theorem \ref{thm:main3} to $E=mK_X$ and the group $\Gamma$, where we will write $n_j(\Gamma)$ for the integers $n_j$ from Equation (\ref{eq:nj}) and $n_{a,b}^{(\Gamma)}$ for the integers $n_{a,b}$ from Equation (\ref{eq:nab}). Note that $n_I=1$ since $\#(I\cap \Gamma)=3$. As in Table \ref{tab:canonicalmodular},  $Z_{\mathrm{br}}=\{z_1,\ldots,z_{(\ell-\epsilon')/2}\}$. Define
\begin{equation}
\label{eq:i0}
i_0=\left\{\begin{array}{ll}1,&\epsilon=-\epsilon',\Gamma=\Delta=\Delta_1,\\
2, & \epsilon=\epsilon',\Gamma=\Delta_1,\\
0, & \epsilon=\epsilon',\Gamma=\Delta_2.\end{array}\right.
\end{equation}
By Table \ref{tab:canonicalmodular}, there are precisely $i_0$ points among the points $y_{t,i}$ ($1\le t\le n', 1\le i \le 3^{n-1}$) on $Y$ occurring in Equation (\ref{eq:Djmodular}) that lie above points in $Z_{\mathrm{br}}$. Hence, we obtain by Table \ref{tab:canonicalmodular} that
$$\mathrm{ord}_{y(z_i)}(mK_Y+D_j)=\left\{\begin{array}{cl}m+ \lfloor \frac{4m-j}{3}\rfloor,&1\le i \le i_0,\\7m,&i_0+1\le i\le (\ell-\epsilon')/2.\end{array}\right. $$
By Equation (\ref{eq:Djcoefficient}), it follows that
\begin{equation}
\label{eq:ellz1}
\ell_{y(z_i),j} = \left\{\begin{array}{cl}
1-\delta_m, & (i_m,j)\in\{(0,1),(0,2), (1,0),(1,1)\}\mbox{ and } 1\le i\le i_0,\\
\delta_m, & \mbox{otherwise.}
\end{array}\right.
\end{equation}
Using the alternative formula for $n_j(\Gamma)$ from Remark \ref{rem:combine}(a), together with Equations (\ref{eq:Djmodulardegree}) and (\ref{eq:ellz1}), we have, for $j\in\{0,1\}$,
\begin{eqnarray*}
n_j(\Gamma)-n_{j+1}(\Gamma) 
&=&\frac{\mathrm{deg}(D_j)-\mathrm{deg}(D_{j+1})}{2\cdot 3^{n-1}}-
\sum_{i=1}^{i_0}\frac{\ell_{y(z_i),j}-\ell_{y(z_i),j+1}}{2}\\
&=&\left\{\begin{array}{cl}
\displaystyle \frac{n'+i_0(1-2\delta_m)}{2},&(i_m,j)=(0,0),\\[1.5ex]
\displaystyle \frac{n'-i_0(1-2\delta_m)}{2},&(i_m,j)=(1,1),\\[1.5ex]
0\,&\mbox{otherwise}.
\end{array}\right.
\end{eqnarray*}
Using Equation (\ref{eq:ellz1}), and that $\mu_{0,-1}(y(z_i)) = 0$ and $\mu_{1,-1}(y(z_i)) = 1$ for $1\le i\le i_0$, we obtain, for $j\in\{0,1\}$,
\begin{eqnarray*}
\sum_{z\in Z_{\mathrm{br}}}\left(\sum_{d=1+\mathrm{min}\{\ell_{y(z_i),j},\ell_{y(z_i),j+1}\}}^{\mathrm{max}\{\ell_{y(z_i),j},\ell_{y(z_i),j+1}\}}  \epsilon_{z_i,j} \,\mu_{a,-d}(y(z_i))\right) &=&
\left\{\begin{array}{rl}
-i_0(1-2\delta_m), & \mbox{$(i_m,j)=(0,0)$ and $a=1$},\\
i_0(1-2\delta_m), & \mbox{$(i_m,j)=(1,1)$ and $a=1$},\\
\multicolumn{1}{c}{0,}&\mbox{otherwise}.\end{array}\right.
\end{eqnarray*}
By Equation (\ref{eq:nab}) and Remark \ref{rem:computation}, it therefore follows that the non-projective indecomposable direct $k\Gamma$-module summands of $\mathrm{Res}^G_\Gamma\HH^0(X,\Omega_X^{\otimes m})$, with their multiplicities, are given by the direct sum
\begin{eqnarray}
\label{eq:Deltastable}
&&\delta_0\left(\frac{n'+i_0(1-2\delta_m)}{2}U_{0,3^{n-1}}^{(\Gamma)}
\oplus \frac{n'-i_0(1-2\delta_m)}{2}U_{1,3^{n-1}}^{(\Gamma)}\right)\\
\nonumber
&\oplus&\delta_1\left(\frac{n'-i_0(1-2\delta_m)}{2}U_{0,2\cdot 3^{n-1}}^{(\Gamma)}
\oplus \frac{n'+i_0(1-2\delta_m)}{2}U_{1,2\cdot 3^{n-1}}^{(\Gamma)}\right)
\end{eqnarray}
where $\delta_0,\delta_1,\delta_m$ are as in Notation \ref{not:modular3} and $i_0$ is as in Equation (\ref{eq:i0}).

Similarly to Equations (\ref{eq:VBrauer1}) and (\ref{eq:VBrauer2}), we now use Equation (\ref{eq:Deltastable}) to determine the values of the Brauer character $\beta$ of $\HH^0(X,\Omega_X^{\otimes m})$ at elements of order 2 in $\Gamma$. Since, by \cite[\S II.8]{Huppert1967}, all elements in $G$ of order 2 are conjugate to the element $s\in \Delta_1$ from Remark \ref{rem:modular3}, we only have to consider the element $s$ and $\Gamma= \Delta_1$. Because the order of $s$ is $2$, $s$ acts on $S_0^{(\Delta_1)}$ (resp. $S_1^{(\Delta_1)}$) as multiplication by $1$ (resp. $-1$). Applying Notation \ref{not:main3ind}(c) to the group $\Delta_1$, it follows that, for $a\in\{0,1\}$, each $U_{a,b}^{(\Delta_1)}$ has $\lfloor \frac{b+1}{2}\rfloor$ composition factors that are isomorphic to $S_a^{(\Delta_1)}$ and $\lfloor \frac{b}{2}\rfloor$ composition factors that are isomorphic to $S_{1-a}^{(\Delta_1)}$. Therefore, letting $\widetilde{\beta}^{(\Delta_1)}$ be the Brauer character of the maximal projective direct $k\Delta_1$-module summand of $\mathrm{Res}^G_{\Delta_1}\HH^0(X,\Omega_X^{\otimes m})$, the value at $s$ of the Brauer character of the direct sum in Equation (\ref{eq:Deltastable}) is equal to
\begin{equation}
\label{eq:sBrauer1}
(\beta-\widetilde{\beta}^{(\Delta_1)})(s) = \delta_0\left(\frac{n'+i_0(1-2\delta_m)}{2}-\frac{n'-i_0(1-2\delta_m)}{2}\right) + \delta_1(0+0) = \delta_0i_0(1-2\delta_m).
\end{equation}
To find $\widetilde{\beta}^{(\Delta_1)}(s)$, we need to determine $n_{a,3^n}^{(\Delta_1)}$ for $a\in\{0,1\}$. Since the Brauer character of $k\Delta_1$ has zero value at all non-identity elements of $\Delta_1$, it follows from Equations (\ref{eq:naj}) and (\ref{eq:nab}) that it suffices to determine $n_{a,3^n}^{(\Delta_1)}-n_2(\Delta_1)$. Using that, for $1\le i \le \frac{\ell-\epsilon'}{2}$, $\mu_{a,-1}(y(z_i))=\mu_{a,1}(y(z_i))=1$ if and only if $a=1$, we get
$$n_{a,3^n}^{(\Delta_1)}-n_2(\Delta_1) = \left\{\begin{array}{cl}
\displaystyle (1-2\delta_m)\left(\delta_0 i_0-\frac{\ell-\epsilon'}{4}\right),&a=1,\\[2ex]
0,&a=0.
\end{array}\right. $$
Therefore, we get, since $3^n$ is odd,
\begin{equation}
\label{eq:sBrauer2}
\widetilde{\beta}^{(\Delta_1)}(s) =-(1-2\delta_m)\left(\delta_0 i_0-\frac{\ell-\epsilon'}{4}\right).
\end{equation}


\subsubsection{The stable $kN_1$-module structure when $\epsilon=-\epsilon'$}
\label{sss:N1_mixed}

We use the notation from \cite[\S6.2.1]{BleherChinburgKontogeorgis2020} for the isomorphism classes of indecomposable $kN_1$-modules. Since $P$ is a normal Sylow $3$-subgroup of $N_1$, all simple $kN_1$-modules are inflated from simple $kN_1/P$-modules. Moreover, $N_1/P$ is isomorphic to a dihedral group of order $2 n'$, which is not divisible by $3$. Hence we can use ordinary character theory to see the following. There are $2+\frac{n'-1}{2}$ isomorphism classes of simple $kN_1$-modules, represented by 2 one-dimensional $kN_1$-modules $S_i^{(N_1)}$, for $i\in\{0,1\}$, such that $S_i^{(N_1)}$ restricts to $S_i^{(\Delta)}$ and to $S_0^{(V)}$, together with $\frac{n'-1}{2}$ two-dimensional simple $kN_1$-modules $\widetilde{S}_t^{(N_1)}$, for $1\le t \le \frac{n'-1}{2}$, such that $\widetilde{S}_t^{(N_1)}$ restricts to $S_t^{(V)}\oplus S_{n'-t}^{(V)}$ and to $S_0^{(\Delta)}\oplus S_1^{(\Delta)}$. For $i\in\{0,1\}$, we write $U_{i,b}^{(N_1)}$ for an indecomposable $kN_1$-module whose socle is isomorphic to $S_i^{(N_1)}$ and whose $k$-dimension is equal to $b$ with $1\le b\le 3^n$. For $t\in\{1,\ldots,\frac{n'-1}{2}\}$, we write $\widetilde{U}_{t,b}^{(N_1)}$ for  an indecomposable $kN_1$-module whose socle is isomorphic to $\widetilde{S}_t^{(N_1)}$ and whose $k$-dimension is equal to $2b$ with $1\le b\le 3^n$. 

By \S\ref{sss:V} and by \S\ref{sss:Delta} for $\Gamma=\Delta$, we can write the non-projective indecomposable direct summands of $\mathrm{Res}_{N_1}^G\,\HH^0(X,\Omega_X^{\otimes m})$, with their multiplicities, as a direct sum
\begin{equation}
\label{eq:whyohwhy1}
\bigoplus_{i\in\{0,1\}} \bigoplus_{j\in\{0,1\}} n_{i,(j+1)3^{n-1}}^{(N_1)} U_{i,(j+1)3^{n-1}}^{(N_1)} 
\;\oplus\; \bigoplus_{t=1}^{(n'-1)/2} \bigoplus_{j\in\{0,1\}} \widetilde{n}_{t,(j+1)3^{n-1}}^{(N_1)}\widetilde{U}_{t,(j+1)3^{n-1}}^{(N_1)}
\end{equation}
where we need to determine the coefficients. 
Let $j\in\{0,1\}$. Restricting Equation (\ref{eq:whyohwhy1}) to $V$, Equation (\ref{eq:Vstable}) leads to the conditions
\begin{equation}
\label{eq:restrictV}
\left\{\begin{array}{ccl}
n_{0,(j+1)3^{n-1}}^{(N_1)} + n_{1,(j+1)3^{n-1}}^{(N_1)} &=& \delta_0 \delta_{j,0} + \delta_1 \delta_{j,1},\\
\widetilde{n}_{t,(j+1)3^{n-1}}^{(N_1)} &=& \delta_0 \delta_{j,0} + \delta_1 \delta_{j,1}\qquad
\mbox{for }t\in\{1,\ldots,\frac{n'-1}{2}\}.
\end{array}\right.
\end{equation}
On the other hand, restricting Equation (\ref{eq:whyohwhy1}) to $\Delta$, Equation (\ref{eq:Deltastable}), with $i_0=1$, leads to the conditions
\begin{equation}
\label{eq:restrictDelta}
\left\{\begin{array}{ccl}
\displaystyle n_{0,(j+1)3^{n-1}}^{(N_1)} + \sum_{t=1}^{(n'-1)/2} \widetilde{n}_{t,(j+1)3^{n-1}}^{(N_1)} &=&
\displaystyle\delta_0 \delta_{j,0} \frac{n'+(1-2\delta_m)}{2}+ \delta_1 \delta_{j,1} \frac{n'-(1-2\delta_m)}{2},\\[2ex]
\displaystyle n_{1,(j+1)3^{n-1}}^{(N_1)} + \sum_{t=1}^{(n'-1)/2} \widetilde{n}_{t,(j+1)3^{n-1}}^{(N_1)} &=&
\displaystyle\delta_0 \delta_{j,0} \frac{n'-(1-2\delta_m)}{2}+ \delta_1 \delta_{j,1} \frac{n'+(1-2\delta_m)}{2}.
\end{array}\right.
\end{equation}
Using Equations (\ref{eq:restrictV}) and (\ref{eq:restrictDelta}), for $j\in\{0,1\}$, it follows that the non-projective indecomposable direct summands of $\mathrm{Res}_{N_1}^G\,\HH^0(X,\Omega_X^{\otimes m})$, with their multiplicities, are given by the direct sum
\begin{eqnarray}
\label{eq:Greenmixed1}
&& \delta_0\left((1-\delta_m)\,U_{0,3^{n-1}}^{(N_1)}\oplus \delta_m\,U_{1,3^{n-1}}^{(N_1)} \oplus \bigoplus_{t=1}^{(n'-1)/2} \widetilde{U}_{t,3^{n-1}}^{(N_1)}\right)\\
\nonumber
&\oplus&\delta_1\left(\delta_m\,U_{0,2\cdot 3^{n-1}}^{(N_1)}\oplus (1-\delta_m)\,U_{1,2\cdot 3^{n-1}}^{(N_1)} \oplus \bigoplus_{t=1}^{(n'-1)/2} \widetilde{U}_{t,2\cdot3^{n-1}}^{(N_1)}\right)
\end{eqnarray}
where $\delta_0,\delta_1,\delta_m$ are as in Notation \ref{not:modular3}.


\subsubsection{The stable $kN_1$-module structure when $\epsilon=\epsilon'$}
\label{sss:N1_equal}

We use the notation from \cite[\S6.2.2]{BleherChinburgKontogeorgis2020} for the isomorphism classes of indecomposable $kN_1$-modules. Arguing similarly to \S\ref{sss:N1_mixed}, there are $4+(\frac{n'}{2}-1)$ isomorphism classes of simple $kN_1$-modules, represented by 4 one-dimensional $kN_1$-modules $S_{i_1,i_2}^{(N_1)}$, for $i_1,i_2\in\{0,1\}$, such that $S_{i_1,i_2}^{(N_1)}$ restricts to $S_{i_1}^{(\Delta_1)}$ and to $S_{i_2}^{(\Delta_2)}$ and to $S_0^{(V)}$ if $i_1=i_2$ and to $S_{n'/2}^{(V)}$ if $i_1\ne i_2$, together with $(\frac{n'}{2}-1)$ two-dimensional simple $kN_1$-modules $\widetilde{S}_t^{(N_1)}$, for $1\le t \le (\frac{n'}{2}-1)$, such that $\widetilde{S}_t^{(N_1)}$ restricts to $S_t^{(V)}\oplus S_{n'-t}^{(V)}$ and to $S_0^{(\Delta_1)}\oplus S_1^{(\Delta_1)}$ and to $S_0^{(\Delta_2)}\oplus S_1^{(\Delta_2)}$. For $i_1,i_2\in\{0,1\}$, we write $U_{i_1,i_2,b}^{(N_1)}$ for an indecomposable $kN_1$-module whose socle is isomorphic to $S_{i_1,i_2}^{(N_1)}$ and whose $k$-dimension is equal to $b$ with $1\le b\le 3^n$. For $t\in\{1,\ldots,(\frac{n'}{2}-1)\}$, we write $\widetilde{U}_{t,b}^{(N_1)}$ for  an indecomposable $kN_1$-module whose socle is isomorphic to $\widetilde{S}_t^{(N_1)}$ and whose $k$-dimension is equal to $2b$ with $1\le b\le 3^n$.  By \S\ref{sss:V} and by \S\ref{sss:Delta} for $\Gamma\in\{\Delta_1,\Delta_2\}$, we obtain, using similar arguments as in \S\ref{sss:N1_mixed}, that the non-projective indecomposable direct summands of $\mathrm{Res}_{N_1}^G\,\HH^0(X,\Omega_X^{\otimes m})$, with their multiplicities, are given by the direct sum
\begin{eqnarray}
\label{eq:Greenequal3}
&& \delta_0\left((1-\delta_m)\left(U_{0,0,3^{n-1}}^{(N_1)}\oplus U_{0,1,3^{n-1}}^{(N_1)}\right)
\oplus \delta_m\left(U_{1,1,3^{n-1}}^{(N_1)} \oplus U_{1,0,3^{n-1}}^{(N_1)}\right)
\oplus\; \bigoplus_{t=1}^{(n'/2-1)} \widetilde{U}_{t,3^{n-1}}^{(N_1)}\right) \\
\nonumber
&\oplus&\delta_1\left(\delta_m\left(U_{0,0,2\cdot 3^{n-1}}^{(N_1)}\oplus U_{0,1,2\cdot 3^{n-1}}^{(N_1)} \right)
\oplus (1-\delta_m)\left(U_{1,1,2\cdot 3^{n-1}}^{(N_1)} \oplus U_{1,0,2\cdot 3^{n-1}}^{(N_1)}\right) 
\oplus\; \bigoplus_{t=1}^{(n'/2-1)} \widetilde{U}_{t,2\cdot3^{n-1}}^{(N_1)}\right).
\end{eqnarray}
where $\delta_0,\delta_1,\delta_m$ are as in Notation \ref{not:modular3}.


\subsection{The stable $kG$-module structure of the holomorphic poly-differentials}
\label{ss:Greencorrespondence} 

We use the same strategy as in \cite[\S6.4]{BleherChinburgKontogeorgis2020}, i.e. we determine the non-projective indecomposable $kG$-modules that are direct summands of $\HH^0(X,\Omega_X^{\otimes m})$, together with their multiplicities. These $kG$-modules are precisely the Green correspondents of the non-projective indecomposable direct $kN_1$-module summands of $\mathrm{Res}^G_{N_1}\HH^0(X,\Omega_X^{\otimes m})$. The two main ingredients we use to find these Green correspondents are  \S \ref{sss:N1_mixed}-\ref{sss:N1_equal}, together with \cite[\S III-\S VI]{Burkhardt1976}. 

As in \cite[\S6.4]{BleherChinburgKontogeorgis2020}, it is important that there is a stable equivalence between the module categories of $kG$ and $kN_1$, which allows us to use the results from \cite[\S X.1]{AuslanderReitenSmalo1997} on almost split sequences to be able to detect the Green correspondents. More precisely, we first identify the indecomposable $kG$-modules that are the Green correspondents of the simple $kN_1$-modules, since each of these lies at the end of the component of the stable Auslander-Reiten quiver to which it belongs. We then follow the irreducible homomorphisms in the appropriate component until we reach the Green correspondent of each of the non-projective indecomposable direct $kN_1$-module summands of $\mathrm{Res}^G_{N_1}\HH^0(X,\Omega_X^{\otimes m})$. As in \cite[\S6.4]{BleherChinburgKontogeorgis2020}, we have to consider four cases.

The main difference when $m>1$ is that the indecomposable non-projective $kG$-modules that occur as direct summands of $\HH^0(X,\Omega_X^{\otimes m})$ depend on the congruence class of $m$ modulo 6. Moreover, if $\ell\equiv -1 \mod 3$ then there are indecomposable non-projective $kG$-modules that are not uniserial and that occur as direct summands of $\HH^0(X,\Omega_X^{\otimes m})$ (see Notation \ref{not:PSL}(b)).


\subsubsection{The stable $kG$-module structure when $\epsilon=-\epsilon'$ and $\epsilon'=1$}
\label{sss:Green_mixed1}

This is the case when $\ell\equiv 1\mod 4$ and $\ell\equiv -1\mod 3$. By \S\ref{sss:N1_mixed}, the non-projective indecomposable direct summands of $\mathrm{Res}_{N_1}^G\,\HH^0(X,\Omega_X^{\otimes m})$, with their multiplicities, are those in the direct sum in Equation (\ref{eq:Greenmixed1}). As recorded in  \cite[\S 6.4.1]{BleherChinburgKontogeorgis2020}, it follows from \cite[\S IV]{Burkhardt1976} that there are $1+(n'-1)/2$ blocks of $kG$ of maximal defect $n$, consisting of the principal block $B_0$ and $(n'-1)/2$ blocks $B_1,\ldots,B_{(n'-1)/2}$, and there are $1+(\ell-1)/4$ blocks of $kG$ of defect 0. There are precisely two isomorphism classes of simple $kG$-modules that belong to $B_0$, represented by the trivial simple $kG$-module $T_0$ and a simple $kG$-module $\widetilde{T}_0$ of $k$-dimension $\ell-1$. For each $t\in\{1,\ldots,(n'-1)/2\}$, there is precisely one isomorphism class of simple $kG$-modules belonging to $B_t$, represented by a simple $kG$-module $\widetilde{T}_t$ of $k$-dimension $\ell-1$. 

Using Notation \ref{not:PSL}, it follows, as in \cite[\S 6.4.1]{BleherChinburgKontogeorgis2020}, that the Green correspondent of $S_0^{(N_1)}$ is $T_0$, the Green correspondent of  $S_1^{(N_1)}$ is $U_{\widetilde{T}_0,(3^n-1)/2}$, and the Green correspondent of $\widetilde{S}_t^{(N_1)}$ is $U_{\widetilde{T}_t,3^n-1}$, for $1\le t\le (n'-1)/2$. Following the irreducible homomorphisms in the stable Auslander-Reiten quiver of $B_0$ (resp. $B_t$, for $1\le t\le (n'-1)/2$), we obtain the indecomposable $kG$-modules that are the Green correspondents of the indecomposable $kN_1$-modules occurring in Equation (\ref{eq:Greenmixed1}). Because of our assumptions on $\epsilon$ and $\epsilon'$, the Green correspondents of $U_{0,3^{n-1}}^{(N_1)}$, $U_{0,2\cdot 3^{n-1}}^{(N_1)}$ and $U_{1,2\cdot 3^{n-1}}^{(N_1)}$  are indecomposable $kG$-modules that are not uniserial. More precisely, using Notation \ref{not:PSL}, the Green correspondent of $U_{0,3^{n-1}}^{(N_1)}$ is $U_{T_0,T_0,3^{n-1}+1}^{(G)}$, the Green correspondent of $U_{0,2\cdot 3^{n-1}}^{(N_1)}$ is $U_{T_0,\widetilde{T}_0,(3^{n-1}+1)/2}^{(G)}$, and the Green correspondent of $U_{1,2\cdot 3^{n-1}}^{(N_1)}$ is $U_{\widetilde{T}_0,T_0,(3^{n-1}+1)/2}^{(G)}$. Therefore, we obtain that  the non-projective indecomposable direct $kG$-module summands of $\HH^0(X,\Omega_X^{\otimes m})$, with their multiplicities, are as stated in part (i)(1) of Theorem \ref{thm:modularresult}.
It follows that for every possible value of $\delta_0,\delta_1,\delta_m$, the $kG$-module $\HH^0(X,\Omega_X^{\otimes m})$ has, for each block $B$ of $kG$, at most one non-projective indecomposable $kG$-module belonging to $B$ as a direct summand.


\subsubsection{The stable $kG$-module structure when $\epsilon=-\epsilon'$ and $\epsilon'=-1$}
\label{sss:Green_mixed2}

This is the case when $\ell\equiv -1\mod 4$ and $\ell\equiv 1\mod 3$. By \S\ref{sss:N1_mixed}, the non-projective indecomposable direct summands of $\mathrm{Res}_{N_1}^G\,\HH^0(X,\Omega_X^{\otimes m})$, with their multiplicities, are again those in the direct sum in Equation (\ref{eq:Greenmixed1}). As recorded in  \cite[\S 6.4.2]{BleherChinburgKontogeorgis2020}, it follows from \cite[\S V]{Burkhardt1976} that there are $1+(n'-1)/2$ blocks of $kG$ of maximal defect $n$, consisting of the principal block $B_0$ and $(n'-1)/2$ blocks $B_1,\ldots,B_{(n'-1)/2}$, and there are $1+(\ell+1)/4$ blocks of $kG$ of defect 0. There are precisely two isomorphism classes of simple $kG$-modules that belong to $B_0$, represented by the trivial simple $kG$-module $T_0$ and a simple $kG$-module $T_1$ of $k$-dimension $\ell$. For each $t\in\{1,\ldots,(n'-1)/2\}$, there is precisely one isomorphism class of simple $kG$-modules belonging to $B_t$, represented by a simple $kG$-module $\widetilde{T}_t$ of $k$-dimension $\ell+1$.

As in \cite[\S 6.4.2]{BleherChinburgKontogeorgis2020}, it follows that the Green correspondent of $S_0^{(N_1)}$ is $T_0$, the Green correspondent of $S_1^{(N_1)}$ is $T_1$, and the Green correspondent of $\widetilde{S}_t^{(N_1)}$ is $\widetilde{T}_t$, for $1\le t\le (n'-1)/2$. Following the irreducible homomorphisms in the stable Auslander-Reiten quiver of $B_0$ (resp. $B_t$, for $1\le t\le (n'-1)/2$), we obtain the indecomposable $kG$-modules that are the Green correspondents of the indecomposable $kN_1$-modules occurring in Equation (\ref{eq:Greenmixed1}). Because of our assumptions on $\epsilon$ and $\epsilon'$, all indecomposable $kG$-modules are uniserial. Hence we obtain that the non-projective indecomposable direct $kG$-module summands of $\HH^0(X,\Omega_X^{\otimes m})$, with their multiplicities, are as stated in part (i)(2) of Theorem \ref{thm:modularresult}.
It follows that for every possible value of $\delta_0,\delta_1,\delta_m$, the $kG$-module $\HH^0(X,\Omega_X^{\otimes m})$ has, for each block $B$ of $kG$, at most one non-projective indecomposable $kG$-module belonging to $B$ as a direct summand.


\subsubsection{The stable $kG$-module structure when $\epsilon=\epsilon'$ and $\epsilon'=1$}
\label{sss:Green_equal3}

This is the case when $\ell\equiv 1\mod 4$ and $\ell\equiv 1\mod 3$. By \S\ref{sss:N1_equal}, the non-projective indecomposable direct summands of $\mathrm{Res}_{N_1}^G\,\HH^0(X,\Omega_X^{\otimes m})$, with their multiplicities, are those in the direct sum in Equation (\ref{eq:Greenequal3}). As recorded in  \cite[\S 6.4.3]{BleherChinburgKontogeorgis2020}, it follows from \cite[\S III]{Burkhardt1976} that there are $1+(n'/2)$ blocks of $kG$ of maximal defect $n$, consisting of the principal block $B_{00}$, another block $B_{01}$, and $(n'/2-1)$ blocks $B_1,\ldots,B_{(n'/2-1)}$, and there are $(\ell-1)/4$ blocks of $kG$ of defect 0. There are precisely two isomorphism classes of simple $kG$-modules that belong to $B_{00}$ (resp. $B_{01}$), represented by the trivial simple $kG$-module $T_0$ and a simple $kG$-module $T_1$ of $k$-dimension $\ell$ (resp. by two simple $kG$-modules $T_{0,1}$ and $T_{1,0}$ of $k$-dimension $(\ell+1)/2$). For each $t\in\{1,\ldots,(n'/2-1)\}$, there is precisely one isomorphism class of simple $kG$-modules belonging to $B_t$, represented by a simple $kG$-module $\widetilde{T}_t$ of $k$-dimension $\ell+1$.

As in \cite[\S 6.4.3]{BleherChinburgKontogeorgis2020}, it follows that the Green correspondent of $S_{0,0}^{(N_1)}$ is $T_0$, the Green correspondent of $S_{1,1}^{(N_1)}$ is $T_1$, and the Green correspondent of $\widetilde{S}_t^{(N_1)}$ is $\widetilde{T}_t$, for $1\le t\le (n'/2-1)$. On the other hand, the Green correspondent of $S_{0,1}^{(N_1)}$ is one of $T_{0,1}$ or $T_{1,0}$. We relabel the simple $kG$-modules, if necessary, to be able to assume that the Green correspondent of $S_{0,1}^{(N_1)}$ is $T_{0,1}$ and the Green correspondent of $S_{1,0}^{(N_1)}$ is $T_{1,0}$. Because of our assumptions on $\epsilon$ and $\epsilon'$, all indecomposable $kG$-modules are uniserial. Using similar arguments as in \S\ref{sss:Green_mixed2}, we obtain that the non-projective indecomposable direct $kG$-module summands of $\HH^0(X,\Omega_X^{\otimes m})$, with their multiplicities, are as stated in part (i)(3) of Theorem \ref{thm:modularresult}.
It follows that for every possible value of $\delta_0,\delta_1,\delta_m$, the $kG$-module $\HH^0(X,\Omega_X^{\otimes m})$ has, for each block $B$ of $kG$, at most one non-projective indecomposable $kG$-module belonging to $B$ as a direct summand.


\subsubsection{The stable $kG$-module structure when $\epsilon=\epsilon'$ and $\epsilon'=-1$}
\label{sss:Green_equal4}

This is the case when $\ell\equiv -1\mod 4$ and $\ell\equiv -1\mod 3$. By \S\ref{sss:N1_equal}, the non-projective indecomposable direct summands of $\mathrm{Res}_{N_1}^G\,\HH^0(X,\Omega_X^{\otimes m})$, with their multiplicities, are again those in the direct sum in Equation (\ref{eq:Greenequal3}). As recorded in  \cite[\S 6.4.4]{BleherChinburgKontogeorgis2020}, it follows from \cite[\S VI]{Burkhardt1976} that there are $1+(n'/2)$ blocks of $kG$ of maximal defect $n$, consisting of the principal block $B_{00}$, another block $B_{01}$, and $(n'/2-1)$ blocks $B_1,\ldots,B_{(n'/2-1)}$, and there are $(\ell-3)/4$ blocks of $kG$ of defect 0. There are precisely two isomorphism classes of simple $kG$-modules that belong to $B_{00}$ (resp. $B_{01}$), represented by the trivial simple $kG$-module $T_0$ and a simple $kG$-module $\widetilde{T}_0$ of $k$-dimension $\ell-1$ (resp. by two simple $kG$-modules $T_{0,1}$ and $T_{1,0}$ of $k$-dimension $(\ell-1)/2$). For each $t\in\{1,\ldots,(n'/2-1)\}$, there is precisely one isomorphism class of simple $kG$-modules belonging to $B_t$, represented by a simple $kG$-module $\widetilde{T}_t$ of $k$-dimension $\ell-1$. 

Using Notation \ref{not:PSL}, it follows, as in \cite[\S 6.4.4]{BleherChinburgKontogeorgis2020}, that the Green correspondent of $S_{0,0}^{(N_1)}$ is $T_0$, the Green correspondent of $S_{1,1}^{(N_1)}$ is $U_{\widetilde{T}_0,(3^n-1)/2}$, and the Green correspondent of $\widetilde{S}_t^{(N_1)}$ is $U_{\widetilde{T}_t,3^n-1}$, for $1\le t\le (n'/2-1)$. On the other hand, the Green correspondent of $S_{0,1}^{(N_1)}$ is a uniserial $kG$-module of length $3^n-1$ whose socle is isomorphic to either $T_{0,1}$ or $T_{1,0}$. We relabel the simple $kG$-modules, if necessary, to be able to assume that the Green correspondent of $S_{0,1}^{(N_1)}$ is $U_{T_{0,1},3^n-1}$ and the Green correspondent of $S_{1,0}^{(N_1)}$ is $U_{T_{1,0},3^n-1}$. Because of our assumptions on $\epsilon$ and $\epsilon'$, the Green correspondents of $U_{0,0,3^{n-1}}^{(N_1)}$, $U_{0,0,2\cdot 3^{n-1}}^{(N_1)}$ and $U_{1,1,2\cdot 3^{n-1}}^{(N_1)}$ are indecomposable $kG$-modules that are not uniserial. Similarly to \S\ref{sss:Green_mixed1}, we obtain, using Notation \ref{not:PSL}, that the Green correspondent of $U_{0,0,3^{n-1}}^{(N_1)}$ is $U_{T_0,T_0,3^{n-1}+1}^{(G)}$, the Green correspondent of $U_{0,0,2\cdot 3^{n-1}}^{(N_1)}$ is $U_{T_0,\widetilde{T}_0,(3^{n-1}+1)/2}^{(G)}$, and the Green correspondent of $U_{1,1,2\cdot 3^{n-1}}^{(N_1)}$ is $U_{\widetilde{T}_0,T_0,(3^{n-1}+1)/2}^{(G)}$. Using similar arguments as in \S\ref{sss:Green_mixed1}, we obtain that the non-projective indecomposable direct $kG$-module summands of $\HH^0(X,\Omega_X^{\otimes m})$, with their multiplicities, are as stated in part (i)(4) of Theorem \ref{thm:modularresult}.
It follows that for every possible value of $\delta_0,\delta_1,\delta_m$, the $kG$-module $\HH^0(X,\Omega_X^{\otimes m})$ has, for each block $B$ of $kG$, at most one non-projective indecomposable $kG$-module belonging to $B$ as a direct summand.


\subsection{The Brauer character of the holomorphic poly-differentials}
\label{ss:brauer}

We use the same strategy as in \cite[\S6.3]{BleherChinburgKontogeorgis2020} to compute the values of the Brauer character of $\HH^0(X,\Omega_X^{\otimes m})$ at all elements $g\in G$ that are 3-regular, i.e. whose order is not divisible by 3. 

The main difference when $m>1$ is that the values of the Brauer character of $\HH^0(X,\Omega_X^{\otimes m})$ at elements of $G$ of order $\ell$ depend on the congruence class of $m$ modulo $\ell$.


\subsubsection{The values of the Brauer character at elements of order $\ell$}
\label{sss:brauerell}
As recorded in \cite[\S6.3]{BleherChinburgKontogeorgis2020}, it follows from \cite[\S II.8]{Huppert1967} that the elements of order $\ell$ fall into 2 conjugacy classes. Moreover, we can choose representatives $r_1$ and $r_2$ of these classes in such a way that $R=\langle r_1\rangle=\langle r_2\rangle$ and $r_2=r_1^{\mu}$ where $\mu\in\{1,2,\ldots,\ell-1\}$ with $\mathbb{F}_\ell^*=\langle \mu\rangle$. 
Let $Y=X$ and $Z=X/R$, and let $\lambda:Y=X\to Z$ be the corresponding tamely ramified Galois cover with Galois group $R$. We apply Theorem \ref{thm:main3} to $E=mK_X$ and the group $R$. Since $\#(I\cap R)=1$, $n_I=0$ and $E_0=mK_X$ on $Y=X$. As in Table \ref{tab:canonicalmodular}, we write $Z_{\mathrm{br}}=\{z_1,\ldots,z_{(\ell-1)/2}\}$. Because $E_0=mK_X$, the information from this table implies, for $1\le i \le \frac{\ell-1}{2}$, 
$$\ell_{x(z_i),0} = \ell-1-m_\ell,\quad\mbox{where $m_\ell\in\{0,1,\ldots,\ell-1\}$ satisfies $m_\ell\equiv m-1 \mod\ell$}.$$
Since for all $1\le i \le \frac{\ell-1}{2}$, $R_{x(z_i)}=R$ and since $Y=X$,  Equations (\ref{eq:naj}) and  (\ref{eq:nab}) imply that the Brauer character of $\mathrm{Res}_R^G\,\HH^0(X,\Omega_X^{\otimes m})$ equals
\begin{equation}
\label{eq:RBrauer}
\sum_{i=1}^{(\ell-1)/2}\left(\sum_{d=1}^{\ell-1-m_\ell}\theta_{x(z_i)}^{-d} - \sum_{d=1}^{\ell-1} \,\frac{d}{\ell}\,\theta_{x(z_i)}^d\right) + n_0(R)\,\beta(kR).
\end{equation}
Since $\beta(kR)$ has zero value at all non-identity elements of $R$, we do not need to determine $n_0(R)$.

By \cite[\S II.8]{Huppert1967}, the normalizer $N_G(R)$ is a semidirect product with normal subgroup $R$ and cyclic quotient group of order $(\ell-1)/2$. This implies, by \cite[p. 193]{Moreno1993}, that there is precisely one point on $X/N_G(R)$ that lies below the points $z_1,\ldots,z_{(\ell-1)/2}$ of $Z=X/R$. Hence $N_G(R)$ permutes transitively the points $x(z_1),\ldots, x(z_{(\ell-1)/2})$ of $X$ lying above them. Let $\xi_\ell$ be a primitive $\ell^{\mathrm{th}}$ root of unity such that $\theta_{x(z_1)}(r_1)=\xi_\ell$. Since $\theta_{x(z_1).g}(r_1) = \theta_{x(z_1)}(g\,r_1g^{-1})$ for all $g\in N_G(R)$ and since, by \cite[\S II.8]{Huppert1967}, the conjugates of $r_1$ by $N_G(R)$ are precisely the elements $(r_1)^{a^2}$ for $1\le a\le (\ell-1)/2$, we obtain
\begin{equation}
\label{eq:arrgh!!}
\{\theta_{x(z_i)}(r_1)\;:\;1\le i \le (\ell-1)/2\} \;=\; \{ (\xi_\ell)^{a^2}:\; 1\le a \le (\ell-1)/2\}.
\end{equation}

Using Gauss sums, we see, similarly to \cite[\S6.3.1]{BleherChinburgKontogeorgis2020}, that there exists a choice of square root of $\epsilon' \ell$, say $\sqrt{\epsilon'\ell\,}$, such that, for all positive integers $d$,
\begin{equation}
\label{eq:Gauss1}
\sum_{a=1}^{(\ell-1)/2} (\xi_\ell)^{a^2d}=\frac{-1+\left(\frac{d}{\ell}\right)\sqrt{\epsilon'\ell\,}}{2}
\quad\mbox{ and } \quad
\sum_{a=1}^{(\ell-1)/2} (\xi_\ell)^{\mu a^2d}=\frac{-1-\left(\frac{d}{\ell}\right)\sqrt{\epsilon'\ell\,}}{2}
\end{equation}
where $\left(\frac{d}{\ell}\right)$ denotes the Legendre symbol. Using Equations (\ref{eq:arrgh!!}) and (\ref{eq:Gauss1}), we get 
\begin{equation}
\label{eq:ValueModular_r0}
\sum_{i=1}^{(\ell-1)/2}\;\sum_{d=1}^{\ell-1-m_\ell}\theta_{x(z_i)}^{-d}(r_1) =
-\frac{\ell-1}{2}+\frac{m_\ell -\sum_{d=1}^{m_\ell}\left(\frac{d}{\ell}\right)\sqrt{\epsilon'\ell\,}}{2}.
\end{equation}
Letting $h_\ell=h_{\mathbb{Q}(\sqrt{-\ell})}$ be the class number of $\mathbb{Q}(\sqrt{-\ell})$, it follows by \cite[\S6.3.1]{BleherChinburgKontogeorgis2020} that
\begin{equation}
\label{eq:ValueModular_r1}
\sum_{i=1}^{(\ell-1)/2} \left(-\sum_{d=0}^{\ell-1} \,\frac{d}{\ell}\,\theta_{x(z_i)}^d\right)(r_1)=
\left\{\begin{array}{ll}
\displaystyle \frac{\ell-1}{4}, &\epsilon'=1,\\[1ex]
\displaystyle \frac{\ell-1}{4}+\frac{h_\ell}{2}\sqrt{-\ell},& \epsilon'=-1.\end{array}
\right.
\end{equation}
Replacing $\xi_\ell$ by $\xi_\ell^\mu$ results in replacing $\sqrt{\epsilon'\ell\,}$ by $-\sqrt{\epsilon'\ell\,}$ in Equation (\ref{eq:Gauss1}), and hence also in Equation (\ref{eq:ValueModular_r0}). Running the computations in \cite[\S6.3.1]{BleherChinburgKontogeorgis2020} with $\xi_\ell^\mu$ instead of $\xi_\ell$, we see that this also results in replacing $\sqrt{-\ell\,}$ by $-\sqrt{-\ell\,}$ in Equation (\ref{eq:ValueModular_r1}) when $\epsilon'=-1$.
Since $\theta_{x(z_1)}(r_2)=\theta_{x(z_1)}(r_1^\mu)=\xi_\ell^\mu$, we therefore obtain from Equations (\ref{eq:RBrauer}), (\ref{eq:ValueModular_r0}) and (\ref{eq:ValueModular_r1}) that
\begin{equation}
\label{eq:BrauerValue_rb}
\beta(\HH^0(X,\Omega_X^{\otimes m}))(r_b) = \left\{\begin{array}{ll}
\displaystyle -\frac{\ell-1}{4}+\frac{m_\ell+(-1)^b\sum_{d=1}^{m_\ell}\left(\frac{d}{\ell}\right)\sqrt{\ell\,}}{2}, &\epsilon'=1,\\[1ex]
\displaystyle -\frac{\ell-1}{4}+\frac{m_\ell-(-1)^b\left(h_\ell-\sum_{d=1}^{m_\ell}\left(\frac{d}{\ell}\right)\right)\sqrt{-\ell\,}}{2},& \epsilon'=-1,\end{array}\right.
\end{equation}
for $b\in\{1,2\}$, where $m_\ell\in\{0,1,\ldots,\ell-1\}$ satisfies $m_\ell\equiv m-1 \mod\ell$. 


\subsubsection{The values of the Brauer character at all $3$-regular elements of $G$}
\label{sss:brauercharNotEll}

By Notation \ref{not:modular3} and Remark \ref{rem:modular3}, $v$ is an element of order $(\ell-\epsilon)/2 = 3^n\cdot n'$, where $n'$ is not divisible by 3, $s$ is an element of order 2, and $w$ is an element of order $(\ell+\epsilon)/2$. Let $v''=v^{3^n}$ be of order $n'$. By \cite[\S II.8]{Huppert1967}, a full set of representatives of the conjugacy classes of 3-regular elements of $G$ is given by Table \ref{tab:3regular}.
\renewcommand{\arraystretch}{1.25}
\begin{table}[ht]
\caption{Representatives of the $3$-regular conjugacy classes of $G=\mathrm{PSL}(2,\mathbb{F}_\ell)$ together with their class lengths.}
\label{tab:3regular}
$
\begin{array}{c||c|c|c|c|c}
\mbox{representative}& 1_G &r_b& s &(v'')^i&w^j 
\\ \hline
\mbox{class length}&1&\frac{\ell^2-1}{2}&\frac{\ell(\ell+\epsilon')}{2} & \ell(\ell+\epsilon)&\ell(\ell-\epsilon)
\end{array}$\\[1ex]
where $b\in\{1,2\}$, $1\le i< \frac{n'}{2}$, $1\le j< \frac{\ell+\epsilon}{4}$.
\end{table}
\renewcommand{\arraystretch}{1}

From Equation (\ref{eq:BrauerValue_rb}), we know the values of $\beta(\HH^0(X,\Omega_X^{\otimes m}))$ at $r_b$, for $b\in\{1,2\}$. The other values of $\beta(\HH^0(X,\Omega_X^{\otimes m}))$ are as follows:
\begin{eqnarray}
\label{eq:Br1}
\beta(\HH^0(X,\Omega_X^{\otimes m}))(1_G)&=&(2m-1)\,\frac{(\ell^2-1)(\ell-6)}{24},\\
\label{eq:Br2}
\beta(\HH^0(X,\Omega_X^{\otimes m}))(s)&=& (1- 2 \delta_m) \,\frac{\ell-\epsilon'}{4},\\
\label{eq:Br3}
\beta(\HH^0(X,\Omega_X^{\otimes m}))((v'')^i)&=&0,\\
\label{eq:Br4}
\beta(\HH^0(X,\Omega_X^{\otimes m}))(w^j) &=& 0.
\end{eqnarray}
where $1\le i< \frac{n'}{2}$ and $1\le j < \frac{\ell+\epsilon}{4}$. 

Note that we obtain Equation (\ref{eq:Br1}) by using the Riemann-Roch theorem, together with the formula for the genus $g(X)$ from \cite[Cor. 3.2]{BendingCaminaGuralnick2005}. Equation (\ref{eq:Br2}) follows from Equations (\ref{eq:sBrauer1}) and (\ref{eq:sBrauer2}), and Equation (\ref{eq:Br3}) follows from Equations (\ref{eq:VBrauer1}) and (\ref{eq:VBrauer2}).

For Equation (\ref{eq:Br4}), we consider the cyclic subgroup $W=\langle w\rangle$ of $G$. Since $\#W=(\ell+\epsilon)/2$ is not divisible by 3, there are $\#W$ isomorphism classes of indecomposable $kW$-modules, which are all simple of $k$-dimension one. We denote representatives of these by $S_a^{(W)}$, for $0\le a\le \#W-1$. 
Let $Y=X$ and $Z=X/W$, and let $\lambda:Y=X\to Z$ be the corresponding tamely ramified Galois cover with Galois group $W$. We apply Theorem \ref{thm:main3} to  $E=mK_X$ and the group $W$. Since $\#(I\cap W)=1$, $n_I=0$ and $E_0=mK_X$ on $Y=X$. As in Table \ref{tab:canonicalmodular}, if $\epsilon=-\epsilon'$ then $Z_{\mathrm{br}}=\{z_1,z_2\}$, and if $\epsilon=\epsilon'$ then $Z_{\mathrm{br}}=\emptyset$. In the former case, we use that $E_0=mK_X$, together with the information from Table \ref{tab:canonicalmodular}, to obtain that $\ell_{y(z_i),0}=\delta_m$ for $i=1,2$.
Let $0\le a\le \#W-1$. Using that, for $\epsilon=-\epsilon'$ and $i=1,2$, $\mu_{a,-1}(y(z_i))=\mu_{a,1}(y(z_i))=1$ if and only if $a$ is odd, we get from Equations (\ref{eq:naj}) and (\ref{eq:nab}) that 
$$n_{a,1}^{(W)}-n_0(W) = \left\{\begin{array}{cl}
2\left(\delta_m-\frac{1}{2}\right),&\mbox{$\epsilon=-\epsilon'$ and $a$ odd},\\
0,&\mbox{$\epsilon=\epsilon'$ or $a$ even}.
\end{array}\right.$$
This implies 
\begin{equation}
\label{eq:WBrauer1}
\mathrm{Res}_W^G\,\HH^0(X,\Omega_X^{\otimes m})\cong
\delta_{\epsilon,-\epsilon'} \,(2\delta_m-1)\bigoplus_{t=0}^{\lfloor\frac{\#W}{2}-1\rfloor} S_{2t+1}^{(W)} \;\oplus\; n_0(W) \,kW.
\end{equation}
Let $\xi_w$ be a fixed primitive $(\#W)^{\mathrm{th}}=(\frac{\ell+\epsilon}{2})^{\mathrm{th}}$ root of unity so that,  for $0\le a\le \#W-1$, $w$ acts as multiplication by $\xi_w^a$ on $S_a^{(W)}$. Then Equation (\ref{eq:WBrauer1}) implies, for $1\le j <\frac{\ell+\epsilon}{4}$,
$$\beta(\HH^0(X,\Omega_X^{\otimes m}))(w^j)=\delta_{\epsilon,-\epsilon'} \,(2\delta_m-1)\sum_{t=0}^{\lfloor\frac{\#W}{2}-1\rfloor} \xi_w^{j(2t+1)}=\delta_{\epsilon,-\epsilon'} \,(2\delta_m-1)\,\xi_w^j\,\sum_{t=0}^{\lfloor\frac{\#W}{2}-1\rfloor} (\xi_w^2)^{jt}=0,$$
where the last equality follows, since, if $\epsilon=-\epsilon'$ then $\#W$ is even, and hence $\lfloor\frac{\#W}{2}-1\rfloor=\frac{\# W}{2}-1$ and $\xi_{w}^2$ is a primitive $(\frac{\# W}{2})^{\mathrm{th}}$ root of unity.


\subsection{The $kG$-module structure of the holomorphic poly-differentials}
\label{ss:fullmodular}

We use the same strategy as in \cite[\S6.4]{BleherChinburgKontogeorgis2020}, i.e. we determine the $kG$-module structure of $\HH^0(X,\Omega_X^{\otimes m})$ using \S\ref{ss:Greencorrespondence} and \S\ref{ss:brauer}. In \S\ref{sss:Green_mixed1}-\ref{sss:Green_equal4}, we already determined the non-projective indecomposable $kG$-modules, with their multiplicities, that are direct  summands of $\HH^0(X,\Omega_X^{\otimes m})$ and we showed that these are as stated in parts (i)(1) - (i)(4) of Theorem \ref{thm:modularresult}. Hence, to determine the full $kG$-module structure of $\HH^0(X,\Omega_X^{\otimes m})$, it remains to determine the decomposition of the projective $kG$-module $Q_\ell$ in part (i) of Theorem \ref{thm:modularresult} into a direct sum of projective indecomposable $kG$-modules. We will do this in \S\ref{sss:full_mixed} and \S\ref{sss:full_equal}. The remainder of Theorem \ref{thm:modularresult} is then proved based on these results, using similar arguments as in \cite[\S6.5]{BleherChinburgKontogeorgis2020}.

As in \cite[\S6.4]{BleherChinburgKontogeorgis2020}, let $\widetilde{\beta}$ denote the Brauer character of the largest projective direct summand of $\HH^0(X,\Omega_X^{\otimes m})$. In other words, $\widetilde{\beta}$ is the Brauer character of the projective $kG$-module $Q_\ell$ in part (i) of Theorem \ref{thm:modularresult}. Let $\mathrm{IBr}(kG)$ denote the set of Brauer characters of simple $kG$-modules, and for each $\phi\in \mathrm{IBr}(kG)$, let $E(\phi)$ be a simple $kG$-module with Brauer character $\phi$ and let $P(G,E(\phi))$ be a projective $kG$-module cover of $E(\phi)$. 

The strategy for \S\ref{sss:full_mixed} and \S\ref{sss:full_equal} is as follows. We will first determine the Brauer character $\widetilde{\beta}$, by using Equations (\ref{eq:BrauerValue_rb}) - (\ref{eq:Br4}), together with the description of the non-projective indecomposable direct $kG$-module summands of $\HH^0(X,\Omega_X^{\otimes m})$ from \S\ref{sss:Green_mixed1} - \S\ref{sss:Green_equal4}. For all $\phi \in \mathrm{IBr}(kG)$, we will then calculate
\begin{equation}
\label{eq:innerell}
\langle \widetilde{\beta},\phi\rangle =\frac{1}{\#G}\sum_{x\in G_3'}\widetilde{\beta}(x)\phi(x^{-1})
\end{equation}
where $G_3'$ denotes the set of $3$-regular elements of $G$. Since $\langle \widetilde{\beta},\phi\rangle$ equals the multiplicity of the projective $kG$-module  $P(G,E(\phi))$ as a direct summand of $Q_\ell$, we obtain
\begin{equation}
\label{eq:Qell!}
Q_\ell=\bigoplus_{\phi\in \mathrm{IBr}(kG)} \langle \widetilde{\beta},\phi\rangle\, P(G,E(\phi))
\end{equation}
which will lead to the precise $kG$-module structure of $\HH^0(X,\Omega_X^{\otimes m})$.

It follows from \cite[\S III-\S VI]{Burkhardt1976} that all elements in $\mathrm{IBr}(kG)$ can be described in terms of ordinary irreducible characters of $G$. To give these descriptions in \S\ref{sss:full_mixed} and \S\ref{sss:full_equal}, we list the relevant ordinary irreducible characters, together with their values at all $3$-regular conjugacy classes, in Table \ref{tab:ordinarycharsmodular}. Here $\xi_{n'}$ is a fixed primitive $(n')^{\mathrm{th}}$ root of unity and $\xi_w$ is a fixed primitive $(\frac{\ell+\epsilon}{2})^{\mathrm{th}}$ root of unity. 
Note that Table \ref{tab:ordinarycharsmodular} combines all cases of $\epsilon,\epsilon'\in\{\pm 1\}$ and uses the labels of the ordinary characters from \cite[\S IV]{Burkhardt1976}, where we replace $\delta^*$ by $\widetilde{\delta}^*$ to avoid confusion with other uses of $\delta$.
\renewcommand{\arraystretch}{1.25}
\begin{table}[ht]
\caption{Restrictions of important ordinary irreducible characters of $G=\mathrm{PSL}(2,\mathbb{F}_\ell)$ to the $3$-regular conjugacy classes.}
\label{tab:ordinarycharsmodular}
$\begin{array}{c||c|c|c|c|c}
& 1_G &r_b& s &(v'')^i&w^j 
\\ \hline\hline
\widetilde{\delta}_0^*&\ell+\epsilon&\epsilon&\epsilon\,|\epsilon'+\epsilon|&2\,\epsilon&0
\\ \hline
\widetilde{\delta}_t^*&\ell+\epsilon&\epsilon&\epsilon\,|\epsilon'+\epsilon|\, (-1)^t&\epsilon\left(\xi_{n'}^{it}+\xi_{n'}^{-it} \right)&0
\\ \hline
\gamma_a&\frac{\ell+\epsilon'}{2}&\frac{\epsilon'+(-1)^{a+b}\sqrt{\epsilon'\ell}}{2}&\epsilon'(-1)^{(\ell-\epsilon')/4}&\epsilon'\frac{|\epsilon'+\epsilon|}{2}(-1)^i&\epsilon'\frac{|\epsilon'-\epsilon|}{2}(-1)^j
\\ \hline
\eta_u^G&\ell-\epsilon&-\epsilon&-\epsilon\,|\epsilon'-\epsilon|\,(-1)^u&0&-\epsilon\left(\xi_w^{ju}+\xi_w^{-ju}\right)
\end{array}$\\[1ex]
where $a,b\in\{1,2\}$, $1\le i,t< \frac{n'}{2}$, $1\le j,u< \frac{\ell+\epsilon}{4}$.
\end{table}
\renewcommand{\arraystretch}{1}


\subsubsection{The largest projective direct $kG$-module summand of $\HH^0(X,\Omega_X^{\otimes m})$ when $\epsilon=-\epsilon'$}
\label{sss:full_mixed}

We use \cite[\S IV and \S V]{Burkhardt1976} to give a description of $\mathrm{IBr}(kG)$ using the restrictions of the ordinary irreducible characters in Table \ref{tab:ordinarycharsmodular}.
Let $\psi_0$ denote the Brauer character of the trivial simple $kG$-module $T_0$. If $\epsilon=-1$ then $\widetilde{\delta}_0^*$ gives the Brauer character of the simple $kG$-module $\widetilde{T}_0$. If $\epsilon=1$ then the Brauer character of the simple $kG$-module $T_1$ is given by $\psi_1=\widetilde{\delta}_0^*-\psi_0$. In both cases, for $1\le t\le \frac{n'-1}{2}$, the Brauer character of the simple $kG$-module $\widetilde{T}_t$ is given by $\widetilde{\delta}_t^*$.  There are $1+\frac{\ell+\epsilon}{4}$ additional Brauer characters of simple $kG$-modules that are also projective, given by $\gamma_a$, $a\in\{1,2\}$, and $\eta_u^G$, $1\le u\le \frac{\ell+\epsilon}{4}-1$. Therefore, if $\epsilon=-\epsilon'$ then
$$\mathrm{IBr}(kG)=\left\{\textstyle\psi_0, \psi_0',\widetilde{\delta}_t^*, \gamma_a, \eta_u^G\;:\;1\le t\le \frac{n'-1}{2}, 1\le a\le 2, 1\le u\le \frac{\ell+\epsilon}{4}-1\right\}$$
where $\psi_0'=\widetilde{\delta}_0^*$ if $\epsilon=-1$ and $\psi_0'=\psi_1$ if $\epsilon=1$, i.e. $\psi_0'=\widetilde{\delta}_0^* - \frac{1+\epsilon}{2}\psi_0$.

We determine the Brauer character $\widetilde{\beta}$ by using Equations (\ref{eq:BrauerValue_rb}) - (\ref{eq:Br4}), together with Table \ref{tab:ordinarycharsmodular}, applied to the composition factors of the non-projective indecomposable direct summands of $\HH^0(X,\Omega_X^{\otimes m})$ from \S\ref{sss:Green_mixed1} and \S\ref{sss:Green_mixed2}, which are as in parts (i)(1) and (i)(2) of Theorem \ref{thm:modularresult}. 

Suppose first that $(\epsilon,\epsilon')=(-1,1)$, so we are in part (i)(1) of Theorem \ref{thm:modularresult}. Let $\beta$ be the Brauer character of $\HH^0(X,\Omega_X^{\otimes m})$. Using Notation \ref{not:PSL}, together with the above description of $\mathrm{IBr}(kG)$, it follows that
\begin{eqnarray*}
\widetilde{\beta} &=& \beta 
- \delta_0(1-\delta_m)\left((3^{n-1}+1)\widetilde{\delta}_0^*+2\psi_0\right)
- \delta_0\delta_m\left(3^{n-1}\widetilde{\delta}_0^*\right) 
- \delta_1\delta_m\left(\frac{3^{n-1}+1}{2}\,\widetilde{\delta}_0^*+\psi_0\right) \\
&& - \delta_1(1-\delta_m)\left(\frac{3^{n-1}+1}{2}\,\widetilde{\delta}_0^*+\psi_0\right) 
- \sum_{t=1}^{(n'-1)/2}\delta_0 \left(2\cdot3^{n-1}\widetilde{\delta}_t^*\right)
- \sum_{t=1}^{(n'-1)/2}\delta_1 \left(3^{n-1}\widetilde{\delta}_t^* \right).
\end{eqnarray*}
To evaluate $\widetilde{\beta}$ at, for example, $(v'')^i$, for $1\le i\le \frac{n'-1}{2}$, we use Equation (\ref{eq:Br3}) and Table \ref{tab:ordinarycharsmodular} to obtain
\begin{eqnarray*}
\widetilde{\beta}((v'')^i) &=& 0
- \delta_0(1-\delta_m)\left((3^{n-1}+1)2\epsilon+2\right)
- \delta_0\delta_m\left(3^{n-1}2\epsilon\right) 
- \delta_1\delta_m\left(\frac{3^{n-1}+1}{2}2\epsilon+1\right)  \\
&& - \delta_1(1-\delta_m)\left(\frac{3^{n-1}+1}{2}2\epsilon+1\right) 
-  \sum_{t=1}^{(n'-1)/2}\left(\delta_0\, 2 \cdot 3^{n-1}+\delta_1\,3^{n-1}\right)\epsilon \left(\xi_{n'}^{it}+\xi_{n'}^{-it} \right)\\
&=&-\left(2\delta_0(1-\delta_m)+\delta_1\right)(1+\epsilon)  -3^{n-1}\epsilon(2\delta_0+\delta_1)
-3^{n-1}\epsilon(2\delta_0+\delta_1)\sum_{t=1}^{n'-1} \xi_{n'}^{it}
\end{eqnarray*}
where the second equality follows since $\xi_{n'}^{-it}=\xi_{n'}^{i(n'-t)}$. Since $\epsilon=-1$ and since $\xi_{n'}^i$ is an $(n')^{\mathrm{th}}$ root of unity, i.e. $\sum_{t=0}^{n'-1}(\xi_{n'})^{it}=0$, this leads to $\widetilde{\beta}((v'')^i)=0$.
Using similar computations to evaluate $\widetilde{\beta}$ at all $3$-regular elements of $G$, we obtain the following values when $(\epsilon,\epsilon')=(-1,1)$:
\begin{eqnarray*}
\widetilde{\beta}(1_G) &=&(2m-1)\,\frac{(\ell^2-1)(\ell-6)}{24}-\frac{\ell+1}{2}(2\delta_0(1-\delta_m)+\delta_1)-\frac{\ell^2-1}{12}(2\delta_0+\delta_1), \\
\widetilde{\beta}(r_b)&=&-\frac{\ell-1}{4}+\frac{m_\ell+(-1)^b\sum_{d=1}^{m_\ell}\left(\frac{d}{\ell}\right)\sqrt{\ell\,}}{2}+\frac{\ell+1}{12}(2\delta_0+\delta_1)-\frac{2\delta_0(1-\delta_m)+\delta_1}{2},\\
\widetilde{\beta}(s)&=&(1-2\delta_m)\frac{\ell-1}{4}-(2\delta_0(1-\delta_m)+\delta_1),\\
\widetilde{\beta} ((v'')^i)&=&0,\\
\widetilde{\beta}(w^j)&=&-(2\delta_0(1-\delta_m)+\delta_1),
\end{eqnarray*}
for $b\in\{1,2\}$, $1\le i\le \frac{n'-1}{2}$, and $1\le j \le \frac{\ell-5}{4}$.

To determine $\langle \widetilde{\beta} , \phi\rangle$ for all $\phi\in\mathrm{IBr}(kG)$ when $(\epsilon,\epsilon')=(-1,1)$, we use Equation (\ref{eq:innerell}), together with the conjugacy class lengths from Table \ref{tab:3regular}. For example, we have, for $0\le t\le \frac{n'-1}{2}$:
\begin{eqnarray*}
\langle \widetilde{\beta} ,\widetilde{\delta}_t^* \rangle &=&\frac{1}{\#G}
\left\{\widetilde{\beta}(1_G)\widetilde{\delta}_t^*(1_G) + \frac{\ell^2-1}{2}\widetilde{\beta}(r_1)\widetilde{\delta}_t^*(r_1^{-1}) + \frac{\ell^2-1}{2}\widetilde{\beta}(r_2)\widetilde{\delta}_t^*(r_2^{-1})  \right\}\\  
&=& \frac{2}{\ell(\ell^2-1)}
\left\{ \left[ (2m-1)\,\frac{(\ell^2-1)(\ell-6)}{24}-\frac{\ell+1}{2}(2\delta_0(1-\delta_m)+\delta_1)-\frac{\ell^2-1}{12}(2\delta_0+\delta_1)\right] (\ell-1)\right.\\
&&\;\;\qquad+\frac{\ell^2-1}{2}\left[-\frac{\ell-1}{4}+\frac{m_\ell-\sum_{d=1}^{m_\ell}\left(\frac{d}{\ell}\right)\sqrt{\ell\,}}{2}+\frac{\ell+1}{12}(2\delta_0+\delta_1)-\frac{2\delta_0(1-\delta_m)+\delta_1}{2}\right](-1)\\
&&\;\;\qquad\left.+\frac{\ell^2-1}{2}\left[-\frac{\ell-1}{4}+\frac{m_\ell+\sum_{d=1}^{m_\ell}\left(\frac{d}{\ell}\right)\sqrt{\ell\,}}{2}+\frac{\ell+1}{12}(2\delta_0+\delta_1)-\frac{2\delta_0(1-\delta_m)+\delta_1}{2}\right](-1)\right\}\\
&=&\frac{(2m-1)(\ell-7)-4(2\delta_0+\delta_1)+6}{12} + \frac{m-1-m_\ell}{\ell}.
\end{eqnarray*}

Performing similar computations in the case when $(\epsilon,\epsilon')=(1,-1)$, we obtain the following result, which gives a detailed description of $Q_\ell$ in Theorem \ref{thm:modularresult}(i)(1),(2).

\begin{subprop}
\label{prop:full_mixed}
Suppose $\epsilon=-\epsilon'$ in Notation $\ref{not:modular3}$. Let $\widetilde{\beta}$ be the Brauer character of the largest projective direct $kG$-module summand of $\HH^0(X,\Omega_X^{\otimes m})$, which is denoted by $Q_\ell$  in Theorem $\ref{thm:modularresult}(i)$. Then 
$Q_\ell$ is as in Equation $(\ref{eq:Qell!})$,
where $\langle \widetilde{\beta},\phi\rangle$ is given by:
\begin{eqnarray*}
\label{eq:psi06.5.1}
\langle \widetilde{\beta},\psi_0\rangle&=&\frac{m-2-2(2\delta_0+\delta_1)+3\delta_m(2\delta_0-1)}{6}-\frac{m-1-m_\ell}{\ell},\\
\label{eq:psi0'6.5.1}
\langle \widetilde{\beta},\psi_0'\rangle&=&\frac{(2m-1)(\ell-6+\epsilon)-4(\frac{3-\epsilon}{2}\delta_0+\frac{3+\epsilon}{2}\delta_1)-6\epsilon}{12}-\epsilon\frac{m-1-m_\ell}{\ell} - \frac{1+\epsilon}{2}\,\langle \widetilde{\beta},\psi_0\rangle,\\
\label{eq:delta6.5.1}
\langle \widetilde{\beta},\widetilde{\delta}_t^*\rangle&=&\frac{(2m-1)(\ell-6+\epsilon)-4(\frac{3-\epsilon}{2}\delta_0+\frac{3+\epsilon}{2}\delta_1)-6\epsilon}{12}-\epsilon\frac{m-1-m_\ell}{\ell},\\
\label{eq:gamma6.5.1}
\langle \widetilde{\beta},\gamma_a\rangle&=&\frac{(2m-1)(\ell-6-\epsilon)+6\epsilon\left(1-(1-2\delta_m) (-1)^{(\ell-\epsilon')/4}\right)}{24} + \epsilon \frac{m-1-m_\ell}{2\ell}\\
\nonumber
&&+ \frac{(-1)^{a-1}}{2}  \left(\frac{1+\epsilon}{2}h_\ell - \sum_{d=1}^{m_\ell}\left(\frac{d}{\ell}\right)\right),\\
\label{eq:eta6.5.1}
\langle \widetilde{\beta}, \eta_u^G\rangle&=&\frac{(2m-1)(\ell-6-\epsilon)+6\epsilon\left(1-(1-2\delta_m) (-1)^u\right)}{12} +\epsilon \frac{m-1-m_\ell}{\ell},
\end{eqnarray*}
\noindent for $1\le t\le \frac{n'-1}{2}$, $a\{1,2\}$, and $1\le u\le \frac{\ell+\epsilon}{4}-1$.
\end{subprop}


\subsubsection{The largest projective direct $kG$-module summand of $\HH^0(X,\Omega_X^{\otimes m})$ when $\epsilon=\epsilon'$}
\label{sss:full_equal}
We use \cite[\S III and \S VI]{Burkhardt1976} to give a description of $\mathrm{IBr}(kG)$ using the restrictions of the ordinary irreducible characters in Table \ref{tab:ordinarycharsmodular}.
Let $\psi_0$ denote the Brauer character of the trivial simple $kG$-module $T_0$. If $\epsilon=1$ then the Brauer character of the simple $kG$-module $T_1$ is given by $\psi_1=\widetilde{\delta}_0^*-\psi_0$. If $\epsilon=-1$ then $\widetilde{\delta}_0^*$ gives the Brauer character of the simple $kG$-module $\widetilde{T}_0$. In both cases, for $1\le t\le \frac{n'}{2}-1$, the Brauer character of the simple $kG$-module $\widetilde{T}_t$ is given by $\widetilde{\delta}_t^*$.  Moreover, the Brauer characters of the simple $kG$-modules $T_{0,1}$ and $T_{1,0}$ are given by $\gamma_1$ and $\gamma_2$. Note that these characters only differ with respect to their values on the elements of order $\ell$ in $G$. Since we have already chosen a square root of $\epsilon'\ell$ to obtain Equations (\ref{eq:Gauss1}) and (\ref{eq:BrauerValue_rb}), we define $s_{01}\in\{\pm 1\}$ such that the Brauer character $\psi_{0,1}$ of $T_{0,1}$ satisfies
\begin{equation}
\label{eq:wretcheddelta_equal}
\psi_{0,1}(r_1)=\frac{\epsilon'+s_{01}\sqrt{\epsilon'\ell}}{2}.
\end{equation}
There are $\frac{\ell+\epsilon-2}{4}$ additional Brauer characters of simple $kG$-modules that are also projective, given by $\eta_u^G$, $1\le u\le \frac{\ell+\epsilon-2}{4}$. Therefore, if $\epsilon=\epsilon'$ then
$$\mathrm{IBr}(kG)=\left\{\textstyle \psi_0, \psi_0',\psi_{0,1},\psi_{1,0},\widetilde{\delta}_t^*, \eta_u^G\;:\;1\le t\le \frac{n'}{2}-1, 1\le u\le \frac{\ell+\epsilon-2}{4}\right\}$$
where $\psi_0'=\psi_1$ if $\epsilon=1$ and $\psi_0'=\widetilde{\delta}_0^*$ if $\epsilon=-1$, i.e. $\psi_0'=\widetilde{\delta}_0^* - \frac{1+\epsilon}{2}\psi_0$.

Similarly to \S\ref{sss:full_mixed}, we determine the Brauer character $\widetilde{\beta}$ by using Equations (\ref{eq:BrauerValue_rb}) - (\ref{eq:Br4}), together with Table \ref{tab:ordinarycharsmodular}, applied to the composition factors of the non-projective indecomposable direct summands of $\HH^0(X,\Omega_X^{\otimes m})$ from \S\ref{sss:Green_equal3} and \S\ref{sss:Green_equal4}, which are as in parts (i)(3) and (i)(4) of Theorem \ref{thm:modularresult}.

Performing similar computations as in \S\ref{sss:full_mixed}, we obtain the following result, which gives a detailed description of $Q_\ell$ in Theorem \ref{thm:modularresult}(i)(3),(4).

\begin{subprop}
\label{prop:full_equal}
Suppose $\epsilon=\epsilon'$ in Notation $\ref{not:modular3}$. Let $\widetilde{\beta}$ be the Brauer character of the largest projective direct $kG$-module summand of $\HH^0(X,\Omega_X^{\otimes m})$, which is denoted by $Q_\ell$  in Theorem $\ref{thm:modularresult}(i)$. Then 
$Q_\ell$ is as in Equation $(\ref{eq:Qell!})$,
where $\langle \widetilde{\beta},\phi\rangle$ is given by:
\begin{eqnarray*}
\label{eq:psi06.5.2}
\langle \widetilde{\beta},\psi_0\rangle&=&\frac{m-2-2(2\delta_0+\delta_1)+3\delta_m(2\delta_0-1)}{6}-\frac{m-1-m_\ell}{\ell},\\
\label{eq:psi0'6.5.2}
\langle \widetilde{\beta},\psi_0'\rangle&=&\frac{(2m-1)(\ell-6+\epsilon)-4(\frac{3-\epsilon}{2}\delta_0+\frac{3+\epsilon}{2}\delta_1)-12\epsilon\delta_m}{12}-\epsilon \frac{m-1-m_\ell}{\ell}  - \frac{1+\epsilon}{2}\,\langle \widetilde{\beta},\psi_0\rangle,\\
\label{eq:psi016.5.2}
\langle \widetilde{\beta},\psi_{i,1-i}\rangle&=&\frac{(2m-1)(\ell-6+\epsilon)-4(\frac{3-\epsilon}{2}\delta_0+\frac{3+\epsilon}{2}\delta_1)}{24} \\
\nonumber
&&+\frac{-6\epsilon\left(1 - (1-2\delta_m)\left((-1)^{(\ell-\epsilon')/4}-(-1)^i\left((1+\epsilon)\delta_0 + (1-\epsilon)\delta_1\right)\right)\right)}{24} -\epsilon  \frac{m-1-m_\ell}{2\ell} \\
\nonumber
&&+ \frac{(-1)^is_{01}}{2}\left(\frac{1-\epsilon}{2}h_\ell - \sum_{d=1}^{m_\ell}\left(\frac{d}{\ell}\right)\right),\\
\label{eq:delta6.5.2}
\langle \widetilde{\beta},\widetilde{\delta}_t^*\rangle&=&\frac{(2m-1)(\ell-6+\epsilon)-4(\frac{3-\epsilon}{2}\delta_0+\frac{3+\epsilon}{2}\delta_1)-6\epsilon\left(1-(1-2\delta_m)(-1)^t\right)}{12} 
 -\epsilon\frac{m-1-m_\ell}{\ell} ,\\
\label{eq:eta6.5.2}
\langle \widetilde{\beta}, \eta_u^G\rangle&=&\frac{(2m-1)(\ell-6-\epsilon)+6\epsilon}{12} + \epsilon\frac{m-1-m_\ell}{\ell},
\end{eqnarray*}
for $i\in\{0,1\}$, $1\le t\le \frac{n'}{2}-1$, and $1\le u\le \frac{\ell+\epsilon-2}{4}$.
\end{subprop}


\bibliographystyle{plain}

\begin{thebibliography}{10}

\bibitem{Alperin1986}
J.~L. Alperin.
\newblock {\em Local representation theory}.
\newblock Cambridge
  Studies in Advanced Mathematics, No. 11.
\newblock Cambridge University Press, Cambridge, 1986.

\bibitem{AuslanderReitenSmalo1997}
M.~Auslander, I.~Reiten, and S.~O. Smal{\o}.
\newblock {\em Representation theory of {A}rtin algebras}, volume~36 of {\em
  Cambridge Studies in Advanced Mathematics}.
\newblock Cambridge University Press, Cambridge, 1997.
\newblock Corrected reprint of the 1995 original.

\bibitem{BendingCaminaGuralnick2005}
P.~Bending, A.~Camina, and R.~Guralnick.
\newblock Automorphisms of the modular curve.
\newblock In {\em Progress in Galois theory}, volume~12 of {\em Dev. Math.},
  pages 25--37. Springer, New York, 2005.
  
\bibitem{BertinMezard2000}
J.~Bertin and A.~M\'{e}zard.
\newblock D\'{e}formations formelles des rev\^{e}tements sauvagement
  ramifi\'{e}s de courbes alg\'{e}briques.
\newblock {\em Invent. Math.}, 141(1):195--238, 2000.

\bibitem{BleherChinburgKontogeorgis2020}
F.~M. Bleher, T.~Chinburg, and A.~Kontogeorgis.
\newblock Galois structure of the holomorphic differentials of curves.
\newblock {\em J. Number Theory}, 216:1--68, 2020.

\bibitem{Borne2006}
N.~Borne.
\newblock Cohomology of {$G$}-sheaves in positive characteristic.
\newblock {\em Adv. Math.}, 201(2):454--515, 2006.

\bibitem{Burkhardt1976}
R.~Burkhardt.
\newblock Die {Z}erlegungsmatrizen der {G}ruppen {${\rm PSL}(2,p^{f})$}.
\newblock {\em J. Algebra}, 40(1):75--96, 1976.

\bibitem{CaroccaVasquez2019}
A.~Carocca and D.~V\'{a}squez.
\newblock Group actions on {R}iemann-{R}och space.
\newblock Preprint, ArXiv:1904.02748, 2015.

\bibitem{ChevalleyWeil1934}
C.~Chevalley, A.~Weil, and E.~Hecke.
\newblock \"{U}ber das {V}erhalten der {I}ntegrale 1. {G}attung bei
  {A}utomorphismen des {F}unktionenk\"orpers.
\newblock {\em Abh. Math. Sem. Univ. Hamburg}, 10(1):358--361, 1934.

\bibitem{CRII}
C.~W. Curtis and I.~Reiner.
\newblock {\em Methods of representation theory. {V}ol. {II}}.
\newblock Pure and Applied Mathematics (New York). John Wiley \& Sons, Inc.,
  New York, 1987.

\bibitem{Diamond1997}
F.~Diamond.
\newblock Congruences between modular forms: raising the level and
              dropping Euler factors.
\newblock {\em Proc. Nat. Acad. Sci. U.S.A.}, 94(21):11143--11146, 1997. \newblock Elliptic curves and modular forms (Washington, DC, 1996).

\bibitem{DieulefaitUrrozRibet2015}
L.~V.~Dieulefait, J.~Jim\'{e}nez Urroz, and K.~A.~Ribet. 
\newblock Modular forms with large coefficient fields via congruences. \newblock {\em Res. Number Theory}, 1:Art. 2, 14, 2015.

\bibitem{GlassJoynerKsir2010}
D.~Glass, D.~Joyner, and A.~Ksir.
\newblock Codes from {R}iemann-{R}och spaces for {$y^2=x^p-x$} over {${\rm
  GF}(p)$}.
\newblock {\em Int. J. Inf. Coding Theory}, 1(3):298--312, 2010.

\bibitem{Hartshorne1977}
R.~Hartshorne.
\newblock {\em Algebraic geometry}.
\newblock Graduate Texts in Mathematics, No. 52.
\newblock Springer-Verlag, New York, 1977.

\bibitem{Hecke1928}
E.~Hecke.
\newblock \"{U}ber ein {F}undamentalproblem aus der {T}heorie der elliptischen
  {M}odulfunktionen.
\newblock {\em Abh. Math. Sem. Univ. Hamburg}, 6(1):235--257, 1928.

\bibitem{Huppert1967}
B.~Huppert.
\newblock {\em Endliche {G}ruppen. {I}}.
\newblock Springer-Verlag, Berlin, 1967.
\newblock Grundlehren der Math. Wissenschaften, Band 134.

\bibitem{Hurwitz1892}
A.~Hurwitz.
\newblock {\"U}ber algebraische {G}ebilde mit eindeutigen {T}ransformationen in
  sich.
\newblock {\em Math. Ann.}, 41(3):403--442, 1892.


\bibitem{JoynerKsir2007}
D.~Joyner and A.~Ksir.
\newblock Decomposition representations of finite groups on {R}iemann-{R}och
  spaces.
\newblock {\em Proc. Amer. Math. Soc.}, 135(11):3465--3476, 2007.

\bibitem{Kani1986}
E.~Kani.
\newblock The {G}alois-module structure of the space of holomorphic
  differentials of a curve.
\newblock {\em J. Reine Angew. Math.}, 367:187--206, 1986.

\bibitem{Karanikolopoulos2012}
S.~Karanikolopoulos.
\newblock On holomorphic polydifferentials in positive characteristic.
\newblock {\em Math. Nachr.}, 285(7):852--877, 2012.

\bibitem{KaranikolopoulosKontogeorgis2013}
S.~Karanikolopoulos and A.~Kontogeorgis.
\newblock Representation of cyclic groups in positive characteristic and
  {W}eierstrass semigroups.
\newblock {\em J. Number Theory}, 133(1):158--175, 2013.

\bibitem{Katz1973}
N.~M. Katz.
\newblock {$p$}-adic properties of modular schemes and modular forms.
\newblock In {\em Modular functions of one variable, {III} ({P}roc. {I}nternat.
  {S}ummer {S}chool, {U}niv. {A}ntwerp, {A}ntwerp, 1972)}, pages 69--190.
  Lecture Notes in Math., Vol. 350. Springer, Berlin, 1973.

\bibitem{KatzMazur1985}
N.~M. Katz and B.~Mazur.
\newblock {\em Arithmetic moduli of elliptic curves}.
\newblock Princeton University Press, Princeton, 1985.

\bibitem{Kock2004}
B.~K\"{o}ck.
\newblock Galois structure of {Z}ariski cohomology for weakly ramified covers
  of curves.
\newblock {\em Amer. J. Math.}, 126(5):1085--1107, 2004.

\bibitem{KockKontogeorgis2012}
B.~K\"{o}ck and A.~Kontogeorgis.
\newblock Quadratic differentials and equivariant deformation theory of curves.
\newblock {\em Ann. Inst. Fourier (Grenoble)}, 62(3):1015--1043, 2012.

\bibitem{KockTait2015}
B.~K\"{o}ck and J.~Tait.
\newblock Faithfulness of actions on {R}iemann-{R}och spaces.
\newblock {\em Canad. J. Math.}, 67(4):848--869, 2015.

\bibitem{Kontogeorgis2007B}
A.~Kontogeorgis.
\newblock Polydifferentials and the deformation functor of curves with
  automorphisms.
\newblock {\em J. Pure Appl. Algebra}, 210(2):551--558, 2007.

\bibitem{MarquesWard2018}
S.~Marques and K.~Ward.
\newblock Holomorphic differentials of certain solvable covers of the
  projective line over a perfect field.
\newblock {\em Math. Nachr.}, 291(13):2057--2083, 2018.

\bibitem{Moreno1993}
C.~Moreno.
\newblock {\em {Algebraic Curves Over Finite Fields}}.
\newblock Cambridge Tracts in Mathematics. Cambridge University Press, 1993.

\bibitem{Nakajima1986}
S.~Nakajima.
\newblock Galois module structure of cohomology groups for tamely ramified
  coverings of algebraic varieties.
\newblock {\em J. Number Theory}, 22(1):115--123, 1986.

\bibitem{Ribet1984}
K.~A. Ribet.
\newblock Congruence relations between modular forms.
\newblock In {\em Proceedings of the {I}nternational {C}ongress of
  {M}athematicians, {V}ol.\ 1, 2 ({W}arsaw, 1983)}, pages 503--514. PWN,
  Warsaw, 1984.
  
 \bibitem{RzedowskiCVillaSMadan1996}
M.~Rzedowski-Calder{\'o}n, G.~Villa-Salvador, and M.~L. Madan.
\newblock Galois module structure of holomorphic differentials in
  characteristic {$p$}.
\newblock {\em Arch. Math. (Basel)}, 66(2):150--156, 1996.

\bibitem{SerreCorpsLocaux1968}
J.-P. Serre.
\newblock {\em Corps locaux}.
\newblock Publications de l'Universit\'{e} de Nancago, No. VIII. Hermann,
  Paris, 1968.
\newblock Troisi\`eme \'{e}dition.

\bibitem{Shimura1971}
G.~Shimura.
\newblock {\em Introduction to the arithmetic theory of automorphic functions},
  volume~11 of {\em Publications of the Mathematical Society of Japan}.
\newblock Princeton University Press, Princeton, 1994.
\newblock Reprint of the 1971 original, Kano Memorial Lectures, 1.

\bibitem{ValentiniMadan1981}
R.~C. Valentini and M.~L. Madan.
\newblock Automorphisms and holomorphic differentials in characteristic {$p$}.
\newblock {\em J. Number Theory}, 13(1):106--115, 1981.

\bibitem{WoodThesis2020}
A.~Wood.
\newblock {\em Galois {S}tructure of {H}olomorphic {P}olydifferentials in
  {P}ositive {C}haracteristic}.
\newblock ProQuest LLC, Ann Arbor, MI, 2020.
\newblock Thesis (Ph.D.)--The University of Iowa.

\end{thebibliography}

\end{document}